\documentclass[a4paper,english,11pt,reqno]{amsart}

\usepackage[utf8]{inputenc}
\usepackage[T1]{fontenc}
\usepackage{babel}
\usepackage[babel]{csquotes}

\usepackage{amsmath}
\usepackage{amsfonts}
\usepackage{amsthm}
\usepackage{bbm}
\usepackage{amssymb}
\usepackage[top=3.4cm, bottom=3.6cm, left=3.6 cm, right=3.6 cm]{geometry}
\usepackage{mathrsfs}
\usepackage[pdftex, pdfstartview=FitW,colorlinks=true,linkcolor=blue,%
urlcolor=blue,citecolor=blue]{hyperref}
\usepackage{stmaryrd}
\usepackage{cleveref}
\usepackage{mathtools}
\usepackage{soul}

\usepackage{url}
\usepackage{xcolor}
\usepackage{enumitem}

\newtheorem{th.}{Theorem}[section]
\newtheorem{thm}[th.]{Theorem}

\newtheorem{lemma}[th.]{Lemma}
\newtheorem{fact}[th.]{Fact}

\newtheorem{proposition}[th.]{Proposition}
\newtheorem{corollary}[th.]{Corollary}

\theoremstyle{definition}
\newtheorem{definition}[th.]{Definition}
\newtheorem*{remark}{Remark}
\newtheorem*{example}{Example}

\numberwithin{equation}{section}

\newcommand{\N}{\mathbb{N}}
\newcommand{\Z}{\mathbb{Z}}

\newcommand{\R}{\mathbb{R}}
\newcommand{\G}{\mathcal{G}}
\newcommand{\eps}{\varepsilon}
\renewcommand{\epsilon}{\varepsilon}
\newcommand{\kg}{\mathfrak{g}}

\newcommand{\kb}{\mathfrak{b}}

\newcommand {\cA} {{\mathcal A}}
\newcommand {\cB} {{\mathcal B}}

\newcommand {\cD} {{\mathcal D}}
\newcommand {\cE} {{\mathcal E}}
\newcommand {\cF} {{\mathcal F}}

\newcommand {\cI} {{\mathcal I}}
\newcommand {\cL} {{\mathcal L}}
\newcommand {\cK} {{\mathcal K}}

\newcommand {\cN} {{\mathcal N}}

\newcommand {\cP} {{\mathcal P}}
\newcommand {\cQ} {{\mathcal Q}}
\newcommand {\cR} {{\mathcal R}}
\newcommand {\cS} {{\mathcal S}}

\newcommand {\sR} {{\mathscr R}}
\newcommand {\sM} {{\mathscr M}}
\newcommand {\sN} {{\mathscr N}}
\newcommand {\sD} {{\mathscr D}}

\newcommand {\tA} {\tilde A}

\newcommand {\tnu} {{\tilde \nu}}

\newcommand {\uphi} {{\underline{\phi}}}
\newcommand {\sS} {\mathscr{S}}
\newcommand {\sT} {\mathscr{T}}
\newcommand {\sV} {\mathscr{V}}

\newcommand {\uF} {{\underline{F}}}
\DeclareMathOperator{\MNC}{(MNC)}

\newcommand {\bt} {{\bf t}}

\DeclareMathOperator{\leb}{Leb}

\DeclareMathOperator{\Proj}{Proj}
\DeclareMathOperator{\supp}{supp}

\DeclareMathOperator{\SL}{SL}

\renewcommand{\sl}{\mathfrak{sl}}

\DeclareMathOperator{\SO}{SO}

\DeclareMathOperator{\M}{M}

\DeclareMathOperator{\codim}{codim}
\DeclareMathOperator{\dist}{dist}
\DeclareMathOperator{\End}{End}
\DeclareMathOperator{\ad}{ad}

\DeclareMathOperator{\tr}{tr}
\DeclareMathOperator{\Leb}{Leb}

\DeclareMathOperator{\Span}{span}

\newcommand{\1}{{\mathbbm{1}}}
\DeclareMathOperator{\Ad}{Ad}

\DeclareMathOperator{\diag}{diag}

\DeclareMathOperator{\Sim}{Sim}
\DeclareMathOperator{\GCD}{GCD}

\DeclareMathOperator{\Aut}{Aut}
\DeclareMathOperator{\inj}{inj}

\DeclareMathOperator{\Id}{Id}
\DeclareMathOperator{\Kb}{K_{big}}
\DeclareMathOperator{\Ks}{K_{small}}

\newcommand{\bs}{\boldsymbol{s}}
\newcommand{\bp}{\boldsymbol{p}}

\newcommand{\dd}{\,\mathrm{d}}

\newcommand{\E}{\mathbb{E}}
\DeclareMathOperator{\Gr}{Gr}

\renewcommand{\setminus}{\smallsetminus}
\renewcommand{\emptyset}{\varnothing}    
\renewcommand{\subset}{\subseteq}

\newcommand{\set}[1]{\{\, #1 \,\}}   
\newcommand{\setbig}[1]{\bigl\{\, #1 \,\bigr\}}
\newcommand{\setBig}[1]{\Bigl\{\, #1 \,\Bigr\}}

\newcommand{\Lyap}{\ell}

\newcommand {\ttr} {\mathtt{r}}
\newcommand {\ttb} {\mathtt{b}}

\newcommand{\norm}[1]{\lVert#1\rVert}  
\newcommand{\normbig}[1]{\bigl\lVert#1\bigr\rVert}

\newcommand{\abs}[1]{\lvert#1\rvert}  
\newcommand{\abse}[1]{\left\lvert#1\right\rvert}
\newcommand{\overprec}[1]{\overset{#1}{\prec}}

\title[Khintchine dichotomy and Schmidt estimates]{Khintchine dichotomy and Schmidt estimates for self-similar measures on $\R^d$.}

\author{Timoth\'ee B\'enard }
\address{CNRS – LAGA, Universit\'e Sorbonne Paris Nord, 99 avenue J.-B.
Cl\'ement, 93430 Villetaneuse}
\email{benard@math.univ-paris13.fr}

\author{Weikun He}
\address{State Key Laboratory of Mathematical Sciences, Academy of Mathematics and Systems Science, Chinese Academy of Sciences, Beijing 100190, China}
\email{heweikun@amss.ac.cn}
\thanks{W.H. is supported by the National Key R\&D Program of China (No. 2022YFA1007500) and the National Natural Science Foundation of China (No. 12288201).}

\author{Han Zhang}
\address{School of Mathematical Sciences, Soochow University, Suzhou 215006, China
}
\email{hzhang.math@suda.edu.cn}
\thanks{H.Z. is supported by the startup grant of Soochow University and the National Natural Science Foundation of China (No. 12501250)}


\date{}

\begin{document}

\begin{abstract}
We extend the classical theorems of Khintchine and Schmidt in metric Diophantine approximation to the context of self-similar measures on $\R^d$. For this, we establish effective equidistribution of associated random walks on $\SL_{d+1}(\R)/\SL_{d+1}(\Z)$.

Our result strengthens that of \cite{BHZ24} which requires $d=1$ and restricts Schmidt-type counting estimates to approximation functions which decay fast enough.

Novel techniques include a bootstrap scheme for the associated random walks despite algebraic obstructions, and a refined treatment of Dani's correspondence.
 Along the way, we also establish non-concentration properties of self-similar measures near algebraic subvarieties of $\R^d$.
\end{abstract}
\maketitle

\tableofcontents
\setcounter{tocdepth}{1}
\large

\section{Introduction}

Let $d\geq 1$ be an integer. Denote by $\Sim(\R^d)$ the group of similarities of $\R^d$, i.e., transformations $\phi :\R^d\rightarrow \R^d$ of the form $\bs \mapsto \ttr_{\phi} O_{\phi} \bs + \ttb_{\phi} $ where $\ttr_{\phi}>0$,   $O_{\phi} \in O(d)$ is a linear orthogonal transformation of $\R^d$, and $\ttb_{\phi}\in \R^d$.
A probability measure $\lambda$ on $\Sim(\R^d)$ is called a \emph{randomized self-similar} IFS where IFS stands for \emph{iterated function system}.
We will always assume that $\lambda$ is \emph{irreducible}, i.e. no proper  affine subspace of $\R^d$  is invariant under $\lambda$-almost every $\phi$. We also suppose that $\lambda$ has a \emph{finite exponential moment}, meaning there exists $\eps>0$ such that
$$\int_{\Sim(\R^d)} \ttr_{\phi}^\eps+\ttr_{\phi}^{-\eps}+\|\ttb_{\phi}\|^{\eps} \dd \lambda(\phi) \,< +\infty. $$

A  probability measure $\sigma$ on $\R^d$ is called \emph{$\lambda$-stationary} if it satisfies
$$\sigma =\int_{\Sim(\R^d)} \phi_{\star}\sigma \dd \lambda(\phi). $$
By \cite[Theorem 2.5]{BP92},  there exists such a measure $\sigma$ if\footnote{In fact the ``if'' part does not require irreducibility, see \cite[Theorem 3.1]{Kloeck22}. Moreover, whether $\lambda$ is irreducible or not  can then be read in the associated stationary measure $\sigma$, as it amounts to $\supp \sigma$  not being included in any  affine hyperplane  (the direct implication is by \Cref{strong-irr}, the converse is trivial).} and only if
$\lambda$ is \emph{contractive in average}, i.e. 
$$\int_{\Sim(\R^d)} \log \ttr_{\phi} \dd \lambda(\phi) <0.$$
In this case, the  $\lambda$-stationary probability measure on $\R^d$  is unique \cite{BP92}. 
Throughout this paper, we consider a probability measure $\sigma$ on $\R^d$ which is \emph{self-similar} in the sense that it is stationary under some randomized self-similar IFS $\lambda$ which is  irreducible and has a finite exponential moment.



\begin{example}
On the real line $\R$, classical examples of self-similar measures include the Lebesgue measure, the normalized Hausdorff measure on a missing digit Cantor set, or Bernouilli convolutions. In higher dimension, one may consider powers of missing digit Cantor measures, the normalized Hausdorff measure on a Sierpiński triangle, Sierpiński carpet, etc.

\end{example}

The goal of the paper is to study the Diophantine properties of typical points distributed according to a self-similar measure on $\R^d$.
This topic originates from a question of Mahler \cite{Mahler84}, asking how well points in the middle-thirds Cantor set can be approximated by rationals. Mahler's question is later recast by Kleinbock-Lindenstrauss-Weiss \cite{KLW04} who ask whether self-similar measures may satisfy a dichotomy in the spirit of Khintchine's theorem. Building upon earlier works, we provide an affirmative answer.

Let us first recall the classical Khintchine  theorem. Given a non-increasing function $\psi : \N\rightarrow \R_{\geq0}$, let us write $W(\psi) $ the set of \emph{$\psi$-approximable} vectors in $\R^d$. In other terms, a vector $\bs\in \R^d$ belongs to $W(\psi)$ if for infinitely many $(\bp, q)\in \Z^d\times \N$, one has
\begin{equation} \label{Dioph-ineq}
\|q\bs- \bp\|_{\infty}<\psi(q),
\end{equation}
where $\norm{\cdot}_{\infty}$ denotes the supremum norm on $\R^d$. The celebrated Khintchine theorem \cite{Khintchine24,Khintchine26} states that $W(\psi)$ has either null or full Lebesgue measure, and that each scenario can be read simply on the function $\psi$, as they respectively correspond to $\sum_{q\in \N} \psi(q)^d<\infty$ and $\sum_{q\in \N} \psi(q)^d=\infty$. In the divergent case, Khintchine's theorem has been further refined by Schmidt \cite{Schmidt} who provided an asymptotic estimate for the number of solutions to the Diophantine inequality \eqref{Dioph-ineq} with bounded $q$.

In this paper, we extend the theorems of Khintchine and Schmidt to the context of self-similar measures. Our main theorem is the following.

\begin{thm}[Khintchine and Schmidt for self-similar measures]
\label{Khintchine-self-similar}
Let $\lambda$ be a probability measure on $\Sim(\R^d)$, and assume $\lambda$ is  irreducible with finite exponential moment.
Let $\sigma$ be a $\lambda$-stationary probability measure on $\R^d$. Let $\psi : \N\rightarrow \R_{\geq0}$ be a non-increasing function.

Then we have the dichotomy
\begin{align} \label{Kh-dic}
  \sigma(W(\psi))=\begin{cases}
    0 &\text{ if } \sum_{q \in \N} \psi(q)^d<\infty,\\
    1 & \text{ if } \sum_{q \in \N}\psi(q)^d=\infty.
  \end{cases}
\end{align}
Moreover, in the divergent case $\sum_{ q\in \N}\psi(q)^d=\infty$, we have the following asymptotic:
 for $\sigma$-a.e. $\bs\in \R^d$,  as $n\to +\infty$:
 \begin{align*}
  \big|\{(\boldsymbol{p},q)\in \Z^{d}\times \llbracket 0,n\rrbracket \,:\, \|q \bs -\bp\|_{\infty}<\psi(q) \} \big| \sim_{\bs,\psi} 2^d\sum_{q=0}^n \psi(q)^d.
 \end{align*}
\end{thm}

Here we use the notation $\llbracket 0,n\rrbracket =[0,n]\cap \N$, and given a finite set $E$, we write $|E|$ for its cardinality.
An asymptotic estimate of the number of \emph{primitive} solutions of the Diophantine inequality is also provided, see \eqref{primitiv-count}.

\bigskip
Let us now explain how the topic has evolved since Mahler's question in the 80's to the above \Cref{Khintchine-self-similar}.
In the early literature, the convergence and divergence aspects of \Cref{Khintchine-self-similar} were addressed separately for specific approximation functions. The first significant result in this direction was obtained by Weiss \cite{Weiss01} for $\psi(q)=1/q^{1+\eps}$ and measures on the real line satisfying a certain decay condition, including the middle thirds Cantor measure.
This work was later generalized by Kleinbock-Lindenstrauss-Weiss \cite{KLW04} to a broader class of measures on $\R^d$, known as friendly measures.
Subsequent important developments, adopting similar terminology, include \cite{PV05,DFSU18,DFSU21, Cohen-Matalon}.
The study of the case $\psi(q)=\eps/q$ was conducted by Einsiedler-Fishman-Shapira \cite{EFS11} for missing digit Cantor measures, and generalized significantly by Simmons-Weiss \cite{SW19} to arbitary self-similar measures.

Without restriction on the non-increasing function $\psi$, Khalil-Luethi \cite{KL23} obtained the Khintchine dichotomy \eqref{Kh-dic} under the condition that the self-similar measure $\sigma$ has a large enough dimension, and $\lambda$ is rational, finitely supported, contractive, and satisfies the open set condition.
For instance, the Cantor measure with single-missing digit in base $5$ is in this class, but not the middle thirds Cantor measure, see also the related works \cite{Yu21,DJ24}, which all assume large dimension for $\sigma$ in some sense.
In our previous paper \cite{BHZ24}, we proposed a new approach, based on projection theorems à la Bourgain, which led to the Khintchine dichotomy \eqref{Kh-dic} in the case where $d=1$, and with no further restriction, i.e., dealing with any self-similar measure $\sigma$ and any non-increasing function $\psi$.
The assumption $d=1$ was however needed at the most crucial step of dimension bootstrapping in the proof.
We also obtained Schmidt's counting theorem for \emph{primitive} solutions under the additional assumption $\psi(q)<1/q$.
Pushing one step further, \Cref{Khintchine-self-similar}  establishes the Khintchine dichotomy for self-similar measures in \emph{arbitrary dimension}, along with a \emph{full Schmidt counting theorem}, i.e., without any restriction on $\psi$ nor asking solutions of \eqref{Dioph-ineq} to be primitive.

\medskip

\noindent{\emph{Other related topics}.} Mahler \cite{Mahler84} also suggested the study of intrinsic Diophantine approximations on fractal sets, i.e., approximation by rationals sitting in the fractal itself. For works in this direction, see e.g. \cite{TWW24,CVY24}.
Beside fractal sets, Diophantine approximation on embedded submanifolds has also attracted much attention over the past years, see e.g. \cite{KM98, VV06, Beresnevich12, BY23}.
See also \cite{Saxce, pfitscher24, HuangSaxce} for a Khintchine's theorem and other results of Diophantine approximation on the Grassmannians and other flag varieties.

\medskip

\Cref{Khintchine-self-similar} will be deduced from a dynamical statement which we now present.
Let $G=\SL_{d+1}(\R)$ and let $\Lambda \subseteq G$ be a lattice.
Consider the quotient space $X=G/\Lambda$ and equip it with the Haar probability measure $m_X$ and a canonical Riemannian metric (more precisely defined in \Cref{Sec-notations}).
For $x\in X$, we denote by $\inj(x)$ the injectivity radius of $x$.
For $l\in \N$, let $B^{\infty}_{\infty,l}(X)$ be the collection of smooth functions on $X$ all of whose derivatives up to order $l$ are bounded, and set $\cS_{\infty,l}(\cdot)$ the associated $C^l$-norm (see \Cref{Sec-notations} for more details on these conventions).
For $t>0$ and $\bs=(s_1,\cdots,s_d)\in \R^d$, consider $a(t),u(\bs)\in G$ given by
\begin{align*}
  a(t)=\begin{pmatrix}
    t^{\frac{1}{d+1}} & & &\\
     & \ddots & & \\
     & & t^{\frac{1}{d+1}}&\\
     & & & t^{-\frac{d}{d+1}}
  \end{pmatrix},
  \,\,\,\,\,\quad \quad u(\bs)=\begin{pmatrix}
    1 & & & s_{1}\\
    & \ddots && \vdots\\
    & &1 & s_d\\
    & && 1
  \end{pmatrix}.
\end{align*}

\begin{thm}[Effective equidistribution of expanding fractals]  \label{Khintchine-dyn}Let $\sigma$ be as in \Cref{Khintchine-self-similar}. For every $x\in X$, $t>1$, $f\in B^\infty_{\infty, l}(X)$, we have
$$\abse{\int_{\R^d} f\bigl(a(t) u(\bs) x\bigr) \dd \sigma(\bs) - m_{X}(f) } \leq C \cS_{\infty, l}(f) \inj(x)^{-1} t^{-c}.$$
where the constants $C,c>0$ only depend on $\Lambda, \sigma$, and $l=\lceil \frac{1}{2}\dim \SO(d+1)\rceil$.
\end{thm}

\Cref{Khintchine-dyn} establishes the exponential equidistribution of the measure $\sigma$, viewed along a unipotent orbit based at an arbitrary point $x$ and expanded under the associated expanding diagonal flow.
The term $\inj(x)^{-1}$ in the error term reflects that equidistribution takes longer when the basepoint $x$ is high in a cusp.

The qualitative convergence (i.e. without rate) implied by \Cref{Khintchine-dyn} resonates with Khalil-Luethi-Weiss \cite{KLW25}, which establishes qualitative equidistribution under expansion by a broader class of diagonal flows, but at the cost of restricting the IFS to be ``carpet'', i.e. rational with equal contraction ratios and no rotation part.

From a quantitative point of view, \Cref{Khintchine-dyn} can be seen as an effective version of Ratner's theorem for the multiparameter unipotent flow $(u(\bs))_{\bs\in \R^d}$ in the context of fractal measures. The non-fractal case (i.e. $\sigma$ is absolutely continuous with respect to Lebesgue) is due to Kleinbock-Margulis, in the broader context of expanding translates of horospherical subgroups \cite{KM96, KM12}, see also \cite{KM20} for the explicit error term involving injectivity radius. Related works in this direction for the non-horospherical case include \cite{Strombergsson15,CY24,Kim24, LMW-EffEq,Yang-EffEq,LMWY25}. In the context of fractal measures, Khalil-Luethi \cite{KL23} obtained \Cref{Khintchine-dyn} under the additional assumption the point $x$ lies in a certain countable set determined by the IFS. They also required the IFS to be rational, contractive, finitely supported, to satisfy the open set condition and has large enough dimension. See also \cite{DJ24} for a different approach for $d=1$. Provided $d=1$, those constraints were eliminated in our previous work \cite{BHZ24}. We now generalize \cite{BHZ24} to arbitrary dimensions, achieving the theorem in full generality.

\bigskip

The connection between \Cref{Khintchine-dyn} and the Khintchine dichotomy in \Cref{Khintchine-self-similar} comes from \emph{Dani's philosophy} \cite{Dani85}: roughly speaking, the Diophantine properties of a vector $\bs\in \R^d$ can be read in the dynamics of the trajectory $(a(t)u(\bs)x_{0})_{t>1}$ on $\SL_{d+1}(\R)/\SL_{d+1}(\Z)$, where $x_0$ denotes the identity coset.
The correspondence is rooted in the simple computation
\begin{align*}\label{equ-Dani}
a(t)u(\bs)(-\bp, q) =(t^{\frac{1}{d+1}}(q\bs -\bp), \,t^{-\frac{d}{d+1}}q),
\end{align*}
which yields that for every $t>0$, $\bs\in \R^d$, $I\subseteq \R, J\subseteq \R^d$, the statement
$$\exists (\bp, q)\in \Z^{d+1}\,:\, q\in I \text{ and } q\bs-\bp \in J $$
is equivalent to the lattice $a(t)u(\bs)\Z^{d+1} \subseteq \R^{d+1}$ intersecting the product set $t^{\frac{1}{d+1}}J\times t^{-\frac{d}{d+1}}I$. Now, identifying $\SL_{d+1}(\R)/\SL_{d+1}(\Z)$ to the space of covolume $1$ lattices in $\R^{d+1}$ and provided the product set $t^{\frac{1}{d+1}}J\times t^{-\frac{d}{d+1}}I$ looks like a ball, the latter condition can be interpreted dynamically, at which point \Cref{Khintchine-dyn} can be used. Since Dani's insight \cite{Dani85}, many works have exploited this connection, e.g. \cite{KM98,KSY22,CY24,KL23, BHZ24}. In \Cref{Sec-Khintchine-dich}, we will show how effective decorrelation of expanding translates (consequence of \Cref{Khintchine-dyn}) can be utilized to obtain the divergent case of the Khintchine dichotomy, along with a rate as in Schmidt's counting theorem. This extends \cite{BHZ24} which assumed $d=1$ and restricted counting to primitive solutions, both assumptions being required to deal with \emph{bounded} Siegel transforms. In this paper, we will tackle Siegel transforms which are not even $L^2$.

\bigskip

To prove \Cref{Khintchine-dyn}, we exploit the self-similarity of $\sigma$ to see that the translate $a(t)u(\bs)x\dd\sigma(s)$ is roughly given by the $\log t$-step of a random walk on $X$ associated to $\lambda$ (see \Cref{lm:cocycle}). This change of view point originates in the work of Simmons-Weiss \cite{SW19}, see also \cite{PS20,PSS23,KL23,DGW24,AG24} for further developments in this direction. The random walk is defined as follows. Assume $\lambda$-a.e. $\phi$ is orientation preserving (i.e. $\det O_{\phi}=1$), write $k_{\phi}=\begin{pmatrix}
  O_\phi & \\
  & 1
\end{pmatrix}$, and let $\mu$ be the probability measure on $\SL_{d+1}(\R)$ given by
\begin{equation} \label{defmu}
\mu=\int_{\Sim(\R^d)} \delta_{k^{-1}_{\phi}a(\ttr^{-1}_{\phi})u(\ttb_{\phi})} \dd \lambda(\phi).
\end{equation}
We show
\begin{thm}[Effective equidistribution of the $\mu$-random walk]
\label{mu^n-equidistribution}
Let $\lambda$ be as in \Cref{Khintchine-self-similar} and orientation preserving. Let $\mu$ be the associated probability measure on $\SL_{d+1}(\R)$ as in \eqref{defmu}.
Then for every $x \in X$, $n\geq 1$ and $f \in B^\infty_{\infty, l}(X)$, we have
\[\abse{\mu^{*n}*\delta_{x}(f)  - m_{X}(f) } \leq  C \cS_{\infty, l}(f) \inj(x)^{-1} e^{-c n} \]
where the constants $C,c>0$ only depend on $\Lambda, \lambda$, and $l=\lceil \frac{1}{2}\dim \SO(d+1)\rceil$.
\end{thm}

Let us explain the steps of the proof of \Cref{mu^n-equidistribution}.
The overall strategy is similar to that in \cite{BHZ24}, which itself is inspired by \cite{BH24}. It is done in three phases, each analyzing the dimension of the $\mu$-random walk. In contrast to the one-dimensional scenario in \cite{BHZ24}, the present higher dimensional setting imposes new difficulties in the second phase, as we explain below.

The first phase is to show that the random walk acquires some (small) positive dimension (\Cref{Proposition: initial dimension}): there exists $A,\kappa>0$ depending on $\Lambda, \lambda$ such that for any $\rho>0$ small, $x,y\in X$, and any $n\geq \abs{\log \rho}+A\abs{\log \inj x}$, we have
\[\mu^{*n}*\delta_x(B_{\rho}y)\leq \rho^{\kappa}.\]
The proof of this statement closely follows the argument in \cite[Proposition 3.1]{BHZ24}, and is based on effective recurrence of the random walk (\Cref{effective-recurrence}).

In the second phase, we bootstrap the initial dimension $\kappa$ arbitrarily close to the dimension of the ambient space by convolving with $\mu$ suitably many times (\Cref{high-dim}). However, unlike the one-dimensional case in \cite{BHZ24}, the multislicing method proposed in \cite[Section 2]{BH24} cannot be applied directly to implement this bootstrapping argument in the current higher dimensional setting. This limitation arises because the non-concentration hypothesis described in \cite[Theorem 2.1]{BH24} is never satisfied for $d\geq 2$ due to algebraic obstructions (see \Cref{obstacle}). To resolve this limitation, we promote a mild non-concentration hypothesis (MNC) which also enables the dimensional bootstrapping (\Cref{MNC-def}, \Cref{sup-mult})
\footnote{Such a strategy resonates with the concurrent and independent work of Zuo Lin \cite{Zuo25}, which manages algebraic obstructions to dimensional bootstrapping in the context of homogeneous spaces of $\SL_{4}(\R)$ acted upon by certain two-parameter unipotent flows.}.

To validate (MNC) in our setting, we first need to rule out potential algebraic obstructions, which is done in Propositions \ref{relative-angle-geom}, \ref{relative-angle-geom2}. We then upgrade the absence of obstruction to the non-concentration property (MNC).
That part of the argument requires regularity properties of the self-similar measure $\sigma$, namely that $\sigma$ is Hölder regular with respect to proper algebraic subvarieties (\Cref{non-acc-var}).
We establish this property in \Cref{Sec-regularity}. It is of independent interest and complements \cite{Mattila1982, Kaenmaki2003, MU2003, KLW04,  FK2018, HLV2025}.


In the concluding phase, as the dimension approaches that of the ambient space, we complete the proof by applying the spectral gap property of the convolution operator $f \mapsto \mu*f$ acting on $L^2(X)$. Finally, \Cref{Khintchine-dyn} follows from \Cref{mu^n-equidistribution}, via the connection between random walks and expanding translates of self-similar measures given in \Cref{lm:cocycle}.

To derive the Khintchine dichotomy and Schmidt's counting theorem, we remark that the convergence part follows from the single effective equidistribution theorem in \Cref{Khintchine-dyn}, as explained in \cite[Theorem 9.1]{KL23}.
In order to handle the asymptotic counting in the divergence part, we need to truncate the associated Siegel's transform appropriately, and then apply different counting strategies according to the values of $\psi$. More precisely, for the case where $\psi(q)$ is not too large, we adopt a refined version of the counting method in \cite{BHZ24}, which is based on effective double equidistribution and is inspired by the original papers of Schmidt \cite{Schmidt, Schmidt60}. To handle the part where $\psi(q)$ is large, inspired by Huang-Saxc\'e \cite{HuangSaxce} and Pfitscher \cite{pfitscher24}, we make use of the fact that for $\sigma$-a.e. $\bs$, the lattice $a(t)u(\bs)\Z^{d+1}$ is not too ``distorted'' as $t\to +\infty$, which guarantees the number of lattice points in $a(t)u(\bs)\Z^{d+1}$ intersecting a large ball is asymptotic to the volume of the ball. This analysis is conducted in \Cref{Sec-Khintchine-dich}. The main challenge for this section, which is new compared to \cite{BHZ24}, is that we have to deal with Siegel transforms which are not bounded, and not even $L^2$ when $d=1$.

\bigskip

\noindent{\bf Structure of the paper}. In \Cref{Sec-notations}, we set up notations for the rest of the paper, and recall some basic facts on self-similar measures. In \Cref{Sec-moment} we discuss moment estimates for $\sigma$ or $\lambda$-trajectories, and derive exponential convergence of $(\lambda^{*n}*\delta_{0})_{n\geq0}$ to $\sigma$.
In \Cref{Sec-regularity}, we prove the H\"older regularity of self-similar measures with respect to proper algebraic subvarieties of $\R^d$. In \Cref{Sec-recurrence}, we recall the effective recurrence property of the $\mu$-random walk on $X$. In \Cref{Sec-postdim}, we show that the distribution $\mu^{*n}*\delta_x$ acquires small positive dimension at exponentially small scales as long as $n$ is large enough depending on $\inj(x)$. In \Cref{Sec-bootstrap}, we promote a mild non-concentration property and implement it in our context to bootstrap the dimension arbitrarily close to $\dim X$. In \Cref{Sec-equidistribution}, we deduce Theorems \ref{Khintchine-dyn}, \ref{mu^n-equidistribution} using the spectral gap of associated Markov operator, we also derive a double equidistribution property. Finally in \Cref{Sec-Khintchine-dich}, we show \Cref{Khintchine-self-similar}.

\bigskip

\noindent{\bf Acknowledgements}.
We thank René Pfitscher and Shucheng Yu for insightful discussions concerning Schmidt's counting theorem and moment estimates of Siegel transforms.
We also thank Bing Li for informing us about existing literatures around the Hölder estimates of self-similar measures near submanifolds.
\section{Notations} \label{Sec-notations}

Throughout the paper, we fix the following notations.
\bigskip

We let $d\geq 1$ be an integer.
We write $G=\SL_{d+1}(\R)$, fix $\Lambda\subset G$ to be a lattice, and set $X=G/\Lambda$ the quotient space. We let $m_G$ denote the $G$-invariant Borel measure on $G$ normalised so that the induced finite measure $m_X$ on $X$ has total mass $1$. Both $m_G$ and $m_X$ are referred to as Haar measures.

\bigskip

\noindent{\bf Metric.} Given integers $m,n\geq1$, we write $\M_{m, n}(\R)$ the space of real matrices with $m$ rows and $n$ columns.
We equip $\M_{m, n}(\R)$ with its standard Euclidean structure.
More precisely, writing $E_{i,j}$ the matrix with coefficient $1$ at the position $(i,j)$ and null coefficients elsewhere, the collection $\{E_{i,j} \,:\, 1\leq i\leq m, 1\leq j\leq n\}$ is an orthonormal basis of $\M_{m\times n}(\R)$.
This Euclidean structure extends naturally to the exterior algebra $\bigwedge^* \M_{m, n}(\R)$.
We denote by $\|\cdot\|$ the associated Euclidean norm on $\M_{m, n}(\R)$, and more generally on $\bigwedge^* \M_{m, n}(\R)$.

Set $\kg:=\sl_{d+1}(\R)$ the Lie algebra of $G$.
We equip $G$ with the unique right $G$-invariant Riemannian metric which coincides with $\|\cdot\|_{|\kg}$ at $\kg=T_{\Id}G$.
We write $\dist(\cdot,\cdot)$ the induced distance on $G$, or the quotient distance on $X$.

Given $\rho>0$, we write $B_{\rho}$ the open ball of radius $\rho$ centered at the neutral element $\Id$ in $G$. In particular, given a point $x\in X$, the set $B_{\rho}x$ is the ball in $X$ of radius $\rho$ and center $x$.

We define the injectivity radius of $X$ at $x$ by
\[\inj(x)=\sup\{\, \rho>0 \,:\,\text{the map $B_{\rho}\rightarrow X, g\mapsto gx$ is injective} \,\}.\]

In the Euclidean space $\R^d$, we set $B^{\R^d}_{\rho}$ the open ball of radius $\rho$ centered at the origin. On some rare occasions (\Cref{Sec-regularity}), we might use $B_{\rho}$ as a shorthand for $B^{\R^d}_{\rho}$. We will explicitely warn about this exception at the few places it occurs.

\bigskip

\noindent{\bf Sobolev norms.}
Write $\cA=\{E_{i,j}\,:\, 1\leq i, j\leq d+1, \, i\neq j \}\cup \{E_{i,i}-E_{i+1,i+1}\,:\, i=1, \dots d\}$ the standard basis of $\kg$. Given $l\in \N$, write $\Xi_{l}$ the set of words of length $l$ with letters in $\cA$.
Each $D\in \Xi_{l}$ acts as a differential operator on the space of smooth functions $C^\infty(X)$.
Given $f\in C^\infty(X)$, $p \in [1, \infty]$, we set
\[\cS_{p,l}(f)=\sum_{D\in \Xi_{l}} \|Df\|_{L^p},\]
where $\|\cdot\|_{L^p}$ refers to the $L^p$-norm for the Haar probability measure $m_{X}$ on $X$.
We let $B^{\infty}_{p,l}(X)$ denote the space of smooth functions $f$ on $X$ such that $\cS_{p,l}(f)<\infty$.

\bigskip
\noindent{\bf Driving measures $\lambda$ and $\mu$.}
Let $\Sim(\R^d)^+$ be the set of orientation preserving similarities of $\R^d$. Every $\phi\in \Sim(\R^d)^+$ can be written uniquely
\[\phi(\bs)=\ttr_{\phi}O_{\phi}\bs+\ttb_{\phi},\quad \bs\in \R^d,\]
for some $O_{\phi}\in \SO_d(\R)$, $\ttr_{\phi}>0$ and $\ttb_{\phi}\in \R^d$.

We set
\begin{align} 
  &K'=\begin{pmatrix}
    \SO_d(\R) & \\
     & 1
  \end{pmatrix}, \nonumber\\
  &A'=\{\, a(t):=\diag(t^{\frac{1}{d+1}},\dots,t^{\frac{1}{d+1}},t^{-\frac{d}{d+1}}): t\in \R_{>0}\,\}\nonumber\\
  &U=\left\{\, u(\bs):=\begin{pmatrix}
    1 & & & s_{1}\\
    & \ddots & & \vdots\\
    & & 1 & s_d\\
    & & & 1
  \end{pmatrix}: \bs\in \R^d  \,\right\}.   \nonumber\\
  \nonumber
\end{align}
Here, we use the implicit convention the $s_{i}$'s are the coordinates of $\bs$, more precisely $\bs=(s_1, \dots, s_{d})$.
Throughout the paper, this convention applies.

Consider the subgroup $P'=K'A'U\subseteq G$. Every $g\in P'$ can be uniquely written as
\[g=k_g^{-1} a(\ttr_g^{-1})u(\ttb_g)\]
where $k_{g}=\begin{pmatrix}
  O_g& \\
  & 1
\end{pmatrix}\in K'$, $\ttr_{g}>0$, and $\ttb_{g}\in \R^d$.
There is an anti-isomorphism\footnote{That is $\phi_{g_1g_2}= \phi_{g_2}\phi_{g_1}$ for all $g_{1},g_{2}\in P'$.} between $P'$ and $\Sim(\R^d)^+$ given by
\[g\in P'\mapsto \phi_g\in \Sim(\R^d)^+,\]
where
\[\phi_g(\bs)=\ttr_g O_g\bs+ \ttb_g.\]
\medskip

Throughout this paper, we fix a probability measure $\lambda$ on $\Sim(\R^d)^+$ and denote by $\mu$ the corresponding probability measure on $P'$ via the above anti-isomorphism. Note that $\lambda$ and $\mu$ determine each other.

We assume that $\lambda$, and equivalently $\mu$, has \emph{finite exponential moment}, which means that there exists $\eps>0$ such that
\[\int_{P'} \ttr_g^{\eps}+\ttr_g^{-\eps}+\norm{\ttb_g}^{\eps} \dd\mu(g)<\infty.\]
We assume that \emph{$ \lambda$ is  irreducible}. This means that for every affine subspace $E\subseteq \R^d$ which satisfies $\phi E=E$ for all $\phi \in \supp \lambda$, we have $E= \R^d$.

\bigskip
\noindent{\bf Self-similar measure $\sigma$.} We fix a probability measure $\sigma$ on $\R^d$ which is $\lambda$-stationary, i.e.
\[
\sigma = \int_{\Sim(\R^d)^+} \phi_{\star} \sigma \dd\lambda(\phi).
\]

\bigskip
\noindent{\bf Lyapunov exponent.} Let $\Ad:G\to \Aut(\kg)$ be the adjoint representation. The quantity $\Lyap$ given by
\begin{equation} \label{def-Lyap}
\Lyap=-\int_{P'} \log \ttr_g \dd \mu(g)
\end{equation}
is the top Lyapunov exponent associated to $\Ad_{\star} \mu$.

By a theorem of Bougerol-Picard \cite[Theorem 2.5]{BP92}, the existence of a $\lambda$-stationary probability measure is equivalent to the condition
$\Lyap>0$,
i.e., the random walk on $\R^d$ driven by $\lambda$ is \emph{contractive in average}. Moreover, in this case, the $\lambda$-stationary probability measure is unique, see \cite[Corollary 2.7]{BP92}.

\bigskip

\noindent{\bf Intervals.}
We write $\N^*$ for the set of positive integers.
For real numbers $a < b$, we write $\llbracket a, b \rrbracket$ to denote $\Z \cap [a,b]$.
Similarly, we set $\rrbracket a, b \rrbracket := \Z \cap (a, b]$.

\bigskip
\noindent{\bf Asymptotic notations.}
We use the Landau notation $O(\cdot)$ and the Vinogradov symbol $\ll$.
Given \( a, b > 0 \), we write \( a \simeq b \) to denote \( a \ll b \ll a \). Furthermore, we say that a statement involving \( a \) and \( b \) holds under the condition \( a \lll b \) if it is valid whenever \( a \leq \varepsilon b \) for some sufficiently small constant \( \varepsilon > 0 \).
The notations \( O(\cdot) \), \( \ll \), \( \simeq \), and \( \lll \) refer to implicit constants that may depend on the dimension $d$, the lattice \( \Lambda \), and the measure \( \lambda \) (or equivalently on \( \mu \), as one determines the other under our conventions). Dependence on other parameters will be indicated explicitly via subscripts.

\section{Moment and finite approximation of self-similar measures}\label{Sec-moment}

We record standard facts about $\sigma$ and how it is approximated by the translation component $\ttb_{\phi}$ of $\phi \sim \lambda^n$ when $n$ is large.

\begin{lemma}[Finite moment]\label{sigma-moment}
There exists $\gamma>0$ such that
\begin{equation*}
\int_{\R^d} \norm{\bs}^\gamma\dd\sigma(\bs) < + \infty.
\end{equation*}
\end{lemma}

\begin{proof}
As $\lambda$ has finite exponential moment and is contractive in average, this follows from Kloeckner \cite[Theorem 3.1 \& Lemma 3.9]{Kloeck22}.
\end{proof}

By Bougerol-Picard \cite[Theorem 2.4]{BP92},  there exists a measurable map $(\Sim(\R^d))^{\N^*}\rightarrow \R^d, \uphi \mapsto \ttb_{\uphi}$ such that for $\lambda^{\otimes \N^*}$-almost every $\uphi=(\phi_{i})_{i\geq 1}$, one has 
$$\ttb_{\uphi} =\lim_{n\to +\infty} \ttb_{\phi_{1}\dotsm \phi_{n}}.$$
Moreover, the variable $(\ttb_{\uphi})_{\uphi \sim \lambda^{\otimes \N^*}}$ has law $\sigma$.
The next lemma quantifies the rate of approximation of $\ttb_{\uphi}$ by $\ttb_{\phi_{1}\dotsm \phi_{n}}$.

\begin{lemma}[Finite approximation] \label{approx-sigma-n} 
There exists $\gamma>0$ such that for every $p\geq1$, 
  \begin{equation*}
\lambda^{\otimes \N^*}\set{\uphi \,:\, \exists n\geq p,\, \|\ttb_{\phi_{1}\dotsm \phi_{n}}- \ttb_{\uphi}\| \geq e^{-\gamma n} }\ll  e^{-\gamma p}.
\end{equation*}
\end{lemma}

\begin{proof}
It is enough to argue for $p=n$, in other words to show that
  \begin{equation*}
\lambda^{\otimes \N^*}\set{\uphi \,:\,  \|\ttb_{\phi_{1}\dotsm \phi_{n}}- \ttb_{\uphi}\| \geq  e^{-\gamma n} }\ll  e^{-\gamma n}.
\end{equation*}
Note that $\ttb_{\phi_{1}\dotsm \phi_{n}}=\phi_{1}\dotsm\phi_{n}(0)$ and $ \ttb_{\uphi}= \phi_{1}\dotsm\phi_{n}(\ttb_{\uphi_{>n}})$ where $\uphi_{>n}:=(\phi_{i})_{i\geq n+1}$. 
In particular, for $\uphi\sim \lambda^{\otimes \N^*}$, the variable $\ttb_{\uphi_{>n}}$ has law $\sigma$ and is independent from $\phi_{1}, \dotsc, \phi_{n}$. 
We deduce 
\begin{equation*}
\lambda^{\otimes \N^*}\set{\uphi \,:\,  \norm{\ttb_{\phi_{1}\dotsm \phi_{n}}- \ttb_{\uphi}} \geq  e^{-\gamma n} }\leq (\lambda^{*n} \otimes \sigma) \set{ (\phi, \bs)\,:\, \norm{\ttr_{\phi} \bs} \geq e^{-\gamma n} } .
\end{equation*}
The result then follows from the large deviation principle for $(\log \ttr_{\phi})_{\phi \sim \lambda^{*n}}$ and the moment control on $\sigma$ from \Cref{sigma-moment}.
\end{proof}

As a consequence, we obtain a joint moment estimate along a $\lambda$-trajectory. 

\begin{lemma}[Moment estimate on right trajectories] \label{moment-traj}
  There exists $\gamma>0$ such that for every $R>1$,
  \begin{equation*}
\lambda^{\otimes \N^*}\set{\uphi \,:\, \exists n\geq 1,\, \norm{\ttb_{\phi_{1}\dotsm \phi_{n}}}  \geq R}\ll  R^{-\gamma}.
\end{equation*}
\end{lemma}

\begin{proof}
For $\eta\in (0, 1)$ and $0<\gamma\lll \eta$, the domination
$$\lambda^{\otimes \N^*}\set{\uphi \,:\, \exists n\geq 1+\eta \log R,\, \norm{\ttb_{\phi_{1}\dotsm \phi_{n}}} \geq R} \ll  R^{-\gamma}$$
is a direct consequence of Lemmas \ref{sigma-moment}, \ref{approx-sigma-n}.
It remains to show that choosing $\eta\lll1$, we have 
\begin{equation}\label{lowrangdom}
\lambda^{\otimes \N^*}\set{\uphi \,:\, \exists  n \leq 1+ \eta\log R,\, \norm{\ttb_{\phi_{1}\dotsm \phi_{n}}}  \geq R} \ll  R^{-\gamma}.
\end{equation}
For this, observe $\ttb_{\phi_{1}\dotsm \phi_{n}} = \sum_{i=0}^{n-1} \ttr_{\phi_{1}}\dotsm \ttr_{\phi_{i}}O_{\phi_1}\cdots O_{\phi_i} \ttb_{\phi_{i+1}}$.  Using the  Markov inequality, we may control the norm of each summand: for every $\eps>0$,
\begin{align*}
\lambda^{\otimes \N^*}\set{\uphi \,:\, \ttr_{\phi_{1}}\dotsm \ttr_{\phi_{i}} \norm{\ttb_{\phi_{i+1}}}  \geq R^{1/2}} \leq  R^{-\eps/2} \E(\ttr_{\phi}^\eps)^i \E(\norm{\ttb_{\phi_{i+1}}}^{\eps} ).
\end{align*}
Fixing $\eps=\eps(\lambda)>0$ so that the right-hand side is finite, then choosing $\eta\lll \eps$, the right hand side is smaller than $R^{-\eps/4}$, and summing over $i\in \llbracket 1, 1+ \eta\log R\rrbracket$, Equation \eqref{lowrangdom} follows. This concludes the proof.
\end{proof}


\section{Regularity of self-similar measures}\label{Sec-regularity}

The main goal of this section to is to establish \Cref{non-acc-var} stating that the self-similar measure $\sigma$ cannot be concentrated near 
proper algebraic subvarieties of $\R^d$. 


The property that self-similar measures or sets avoid submanifolds dates back to \cite{Hut81, Mattila1982}, then \cite{Kaenmaki2003, FK2018} in the context of self-conformal or self-affine measures. Those results however  all care about the mass given by $\sigma$ to a manifold, and ignore neighborhoods. Non-concentration estimates on neighborhoods can be found in \cite[Theorem 7.3]{KLW04} in the context of absolutely friendly measures. These comprise Hausdorff measures on self-similar sets induced by finite contractive IFS's satisfying the open set condition \cite[Theorem 2.3]{KLW04}, but not the generality allowed in this article. In the very recent paper \cite{HLV2025}, the  case of Gibbs measures on self-conformal sets is investigated for neighborhoods of spheres \cite[Theorem 6.1]{HLV2025}. The present section  complements these works by dealing with arbitrary self-similar measures (coming from a randomized  irreducible self-similar IFS with finite exponential moment) and arbitrary subvarieties.

\subsection{H\"older regularity near affine subspaces}
In the next lemma, we first observe that $\sigma$ cannot charge any proper affine subspace.
\begin{lemma}  \label{strong-irr}
For every  proper affine subspace $\cL\subsetneq \R^d$, we have $\sigma(\cL)=0$. Equivalently, any set $E\subseteq \R^d$ which is a finite union of affine subspaces and is invariant under $\supp \lambda$ must satisfy $E=\R^d$.
\end{lemma}

The second characterization is  a  property of \emph{strong irreducibility}. It strengthens the irreducibility assumption on $\lambda$ introduced in \Cref{Sec-notations}.

\begin{proof}
We start by establishing the equivalence.

$(\implies)$ Assume there exists $E\subsetneq \R^d$ which is a finite union  of affine subspaces and $\supp \lambda$-invariant. Given any $x\in E$, we have $\lambda^{*n}*\delta_{x} \rightharpoonup^* \sigma$. Therefore  $\sigma(E)=1$ by invariance, and it follows $\sigma$ gives positive mass to one of the proper affine subspaces making $E$.

($\impliedby$) 
Assume $\sigma (\cL)>0$ for some proper affine subspace $\cL$.
Among these $\cL$, consider only those with minimal dimension.
They intersect with each other in $\sigma$-null sets.
Hence some $\cL$ maximizes the value $\sigma(\cL)$ and only finitely many $\cL'$ of same dimension can attain $\sigma(\cL') = \sigma(\cL)$.
By $\lambda$-stationarity of $\sigma$ and the maximality condition on $\sigma(\cL)$, we have $\sigma(\phi \cL)=\sigma(\cL)$ for all $\phi \in \langle\supp \lambda \rangle$.
It follows that $\cL$ has finite $\langle \supp {\lambda} \rangle$-orbit
, contradicting strong irreducibility.

Now that we have established the equivalence, we  show strong irreducibility.
Let $E$ be as in the statement.
It can be written in a unique way as $E=\bigcup_{i\in I}E_{i}$, where $I$ is finite and $E_i$ are affine subspaces satisfying $E_{i}\not \subseteq E_{j}$ for all $i \neq j\in I$.
To show $E=\R^d$, we may assume that $\supp\lambda$ permutes transitively the $E_{i}$'s and hence they all have same dimension, say $k$.
We may also assume that $k$ is minimal among all possible choices of $E$.
The minimality implies the $E_{i}$'s are mutually disjoint.
We show that in fact $|I|=1$, i.e., there is only one component $E_{i}$.
Indeed, otherwise, the quantity $\delta := \inf\set{\norm{x-y} \,:\, {i\neq j},\, x\in E_{i},\, y\in E_{j} }$ would be strictly positive.
But recall that $\lambda$ is contractive in average, whence there is some similitude $\phi\in \supp \lambda$ with contraction ratio $\ttr_{\phi} \in (0,1)$.
Then $\delta$ gets strictly smaller when we replace each $E_{i}$ by $\phi E_{i}$, thus contradicting invariance of the union. Therefore $|I|=1$. By irreducibility of $\lambda$, we deduce $E=\R^d$.
\end{proof}

We upgrade \Cref{strong-irr} to a Hölder control on the mass granted by $\sigma$ to neighborhoods of affine subspaces.
\begin{proposition}[H\"older regularity near affine subspaces] \label{non-conc-sigma-aff}
There exists $C,c>0$ such that for every $\eps>0$, every proper affine subspace $\cL\subsetneq \R^d$,
$$\sigma(\cL^{(\eps)})\leq C\eps^c,$$
where $\cL^{(\eps)}$ is the $\eps$-neighborhood of $\cL$ in $\R^d$.
\end{proposition}

In the case where $\sigma$ arises from a self-similar IFS which is finite, contractive, and satisfies the open set condition, the result is a consequence of \cite[Lemmas 8.2, 8.3]{KLW04}. In \cite{KLW04}, the argument is roughly as follows. Assume for simplicity $\lambda$ has equal contraction ratios $\eps\in (0, 1)$. Consider $\underline{\phi}=(\phi_{i})_{i\geq 1}\sim \lambda^{\otimes \N}$ and observe $x_{\uphi}:=\lim_{k}\phi_{1}\dotsm \phi_{k}(0)$ has law $\sigma$.
As $\|x_{\uphi}-\phi_{1}\dotsm \phi_{n}(0)\|\ll \eps^n$, the inclusion $x_{\uphi}\in \cL^{(\eps^n)}$ implies for every $i\leq n$ that $\phi_{1}\dotsm \phi_{i}(0) \in \cL^{(O(\eps^i))}$, i.e., $\phi_{i}(0) \in (\phi_{1}\dotsm \phi_{i-1})^{-1}\cL^{(O(\eps^i))}$.
Up to taking $\eps$ small enough (by considering $\lambda^{*n}$ instead of $\lambda$), the probability of the latter event conditionally to previous steps is smaller than $1/2$, whence the desired decay.

In our situation, the support of $\sigma$ may be unbounded, hence knowing that $x_{\uphi}$ is close to $\cL$ does not mean that all steps leading to $x_{\uphi}$ will be as well. To deal with this difficulty, we will use an induction scheme that both takes into  account the position of $\phi_{1}\dotsm \phi_{n}(0)$ with respect to $\cL$ but also how far $\cL$ is from the origin.


\bigskip
As a preparation for the proof, we first introduce a stopping time which will be helpful in the case where the contraction ratio $\ttr_\phi$ is not uniform for $\phi \in \supp \lambda$.
Given $\eps>0$, we define $\tau_{\eps}: \Sim(\R^d)^{\N^*}\rightarrow \N\cup \{\infty\}$ to be the stopping time
$$\tau_{\eps}(\uphi)= \inf\{n\geq 1\,:\, \ttr_{\phi_{1}\dotsm \phi_{n}}<\eps \}.$$
where $\phi_{i}$ denotes the $i$-th coordinate of $\uphi$, i.e. $\uphi=(\phi_{i})_{i\geq1}$.  As $\lambda$ is contracting in average, we have that $\tau_{\eps}$ is $\lambda^{\otimes \N^*}$-almost surely finite. Denote by $\lambda^{*\tau_{\eps}}$ the distribution of
$\phi_{1}\dotsm \phi_{\tau_{\eps}(\uphi)}$ as $\uphi\sim \lambda^{\otimes \N^*}$.
The following lemma summarises what we need about this measure.
\begin{lemma}
\label{lm:st}
For all $\delta, \gamma>0$ small enough depending on $\lambda$, for all $\eps\in (0, 1)$, we have the following. 
\begin{enumerate}
\item The measure $\sigma$ is $\lambda^{*\tau_{\eps}}$-stationary.
\item The variable $\eps/\ttr_{\phi}$ where $\phi\sim \lambda^{*\tau_{\eps}}$ has a moment of  order $\gamma$ bounded by $2$:
\begin{equation}
\label{eq:st moment}
 \int (\eps/\ttr_{\phi})^\gamma \dd\lambda^{*\tau_{\eps}}(\phi) \leq 2.
\end{equation}

\item Provided that $\eps \in (0, \delta)$, for all proper affine subspaces $\cL \subset \R^d$,
\begin{equation}
\label{eq:lambdacL}
\lambda^{*\tau_{\eps}}\set{\phi \,:\, \dist(\ttb_{\phi}, \cL)\leq \delta} \leq 2^{-9}.
\end{equation}

\item For all $T\geq \delta^{-1}/2$,
\begin{equation}
\label{eq:lambdattb}
\lambda^{*\tau_{\eps}}\set{\phi \,:\, \norm{\ttb_{\phi}}  \geq T } \leq 2^{-9} (1+T)^{-\gamma}.
\end{equation}
\end{enumerate}
\end{lemma}

\begin{proof}
\begin{enumerate}

\item This is a consequence of \cite[Lemma A.2]{Ben22}.
\item By \cite[Lemma A.19]{Benard24}, we have for every $C>1$,
$$\lambda^{* \tau_{\eps}}\{\ttr_{\phi}< \eps 2^{-C} \}\leq \sum_{n=0}^{+\infty} \lambda^{* \tau_{1/2}}\{\ttr_{\phi}< 2^{-n-C}\},$$
and it follows that
$$\lambda^{* \tau_{\eps}}\{\ttr_{\phi}< \eps 2^{-C} \} \leq \int_{|\log \ttr_{\phi}|\geq C \log 2} 1+|\log \ttr_{\phi}| \dd \lambda^{* \tau_{1/2}}(\phi)$$
This integral  is bounded by $O(e^{-\alpha C})$ for $\alpha\lll1$, because  finite exponential moment holds for $\lambda$, whence for $\tau_{1/2}$  (by the large deviation principle), and then for $\lambda^{* \tau_{1/2}}$. This justifies  \eqref{eq:st moment}.
\item In view of \Cref{strong-irr} and using a compactness argument, we can find $\eta = \eta(\sigma)>0$ such that $\sup_{\cL} \sigma(\cL^{(2\eta)}) \leq 2^{-10}$. On the other hand,  \Cref{approx-sigma-n} implies that for $\delta>0$ small enough depending on $\lambda, \eta$, for $\eps\in(0, \delta)$,
$$
\lambda^{\otimes \N^*}\{\uphi \, :\,\norm{\ttb_{\phi_1 \dotsm \phi_{\tau_\eps(\uphi)}}  - \ttb_\uphi} \geq \eta\} \leq 2^{-10}.
$$
The claim \eqref{eq:lambdacL} follows by combining the two inequalities.

\item It is a direct consequence of \Cref{moment-traj} and the Markov inequality.
\qedhere
\end{enumerate}
\end{proof}

We may now pass to the proof of \Cref{non-conc-sigma-aff}.

\begin{proof}[Proof of \Cref{non-conc-sigma-aff}]
Let $c, \eps \in (0,1)$ be some parameters that we will specify below. Provided $c$ is sufficiently small depending on $\lambda$, and $\eps$ is well chosen\footnote{It will be important to $\eps$ be  fine tuned depending on $c$, in particular not too small and not too large.} depending on $c$,  we show the following statement (which clearly implies the desired upper bound):  
For every $k \in \N$, for all proper affine subspaces $\cL \subsetneq \R^d$, all $s \geq \eps^k$, 
\begin{equation} \label{claim-IsT}
\sigma\bigl(\cL^{(s)}\bigr)\leq C_k s^c \bigl(1 + \dist(0,\cL)\bigr)^{- c},
\end{equation}
where $C_k = \eps^{-k c / 2} C_0$ for  $k \geq 1$ and $C_0 = C_0(\sigma)$ depends only on $\sigma$.
 
We argue by induction on $k$. The case $k=0$ is clear because $\sigma$ has a moment of positive order (\Cref{sigma-moment}) and we are allowed to pick $c > 0$ small compared to the constant $\gamma$ in \Cref{sigma-moment}.

Assume \eqref{claim-IsT} is true for some $k \in \N$ and let $s \geq \eps^{k + 1}$.
Note that by the $\lambda^{*\tau_\eps}$-stationarity of $\sigma$ (\Cref{lm:st} (1)), we have
\[
\sigma(\cL^{(s)})=\int \sigma\bigl(\phi^{-1} (\cL^{(s)})\bigr) \dd \lambda^{*\tau_\eps}(\phi).
\]
Since $\phi$ is a similarity of contraction ratio $\ttr_\phi$,
we have $\phi^{-1} (\cL^{(s)}) = (\phi^{-1} \cL)^{(\ttr_\phi^{-1} s)}$
as well as $\ttr_\phi \dist(0, \phi^{-1} \cL) = \dist( \phi(0), \cL) = \dist(\ttb_\phi, \cL)$.
By the definition of the stopping time, for $ \lambda^{*\tau_\eps}$-almost every $\phi$, we have $\ttr_\phi^{-1} s \geq \eps^k$, therefore we may invoke the induction hypothesis to deduce
\begin{align}
\sigma(\cL^{(s)}) &= \int \sigma( (\phi^{-1} \cL)^{(\ttr_\phi^{-1} s)}) \dd \lambda^{*\tau_\eps}(\phi) \nonumber\\
&\leq C_k \int (\ttr_\phi^{-1} s)^{c} \bigl(1 + \dist(0,\phi^{-1} \cL)\bigr)^{- c} \dd \lambda^{*\tau_\eps}(\phi) \nonumber\\
& = C_k s^c \int \bigl(\ttr_\phi + \dist(\ttb_\phi, \cL) \bigr)^{-c} \dd \lambda^{*\tau_\eps}(\phi). \label{intCkscttphi}
\end{align}
Note that in  \eqref{intCkscttphi}, the integral is  bounded by $O(\eps^c)$ (by ignoring  the term $\dist(\ttb_\phi, \cL)$ and applying  \eqref{eq:st moment}), but this is not good enough to conclude, as it would yield $C_{k+1}\geq \eps^{-c}C_{k}$, so the sequence $C_{k}$ would grow too fast. To upper bound  the integral  in  \eqref{intCkscttphi}, we need to exploit the contribution of $\dist(\ttb_\phi, \cL)$.

Let $\delta = \delta(\lambda) \in (0,1/10)$ be as in \Cref{lm:st}, write $T := \dist(0, \cL)$ for a shorthand  and set 
\[
A := \setbig{ \phi \in \Sim(\R^d) \,:\, \dist(\ttb_\phi, \cL) \leq \delta^2 ( 1 + T ) }.
\]
Using
\[
\ttr_\phi + \dist(\ttb_\phi, \cL) \geq
\begin{cases}
\ttr_\phi & \text{ if } \phi \in A,\\
\delta^2 (1 + T) & \text{ otherwise,}
\end{cases}
\]
and \eqref{intCkscttphi}, followed by the Cauchy-Schwarz inequality and \eqref{eq:st moment}, we obtain
\begin{align}
\frac{\sigma(\cL^{(s)})}{C_k s^c} & \leq \int_A \ttr_\phi^{-c} \dd \lambda^{*\tau_\eps}(\phi) + \delta^{-2 c}(1 + T)^{-c} \nonumber\\
& \leq 2 \eps^{-c}  \lambda^{*\tau_\eps}(A)^{1/2} + \delta^{-2 c}(1 + T)^{- c}. \label{domEepsclte}
\end{align}
We claim that
\begin{equation}
\label{eq:lambdaA}
\lambda^{*\tau_\eps}(A) \leq 2^{-8} (1 + T)^{-2 c},
\end{equation}
provided $c$ is small enough depending on $\lambda$.
Assuming this claim, one can conclude the proof as follows:  Equations \eqref{domEepsclte}, \eqref{eq:lambdaA} together imply
\[
\sigma(\cL^{(s)}) \leq \bigl(8^{-1}\eps^{-c} + \delta^{-2c}\bigr) C_k s^c (1 + T)^{-c}.
\]
Choosing $c$  small so that $\delta^{-2 c} \leq 5/4$, and $\eps=\eps(c) \in (0,1)$ so that $\eps^{-c} = 4$, we find 
$8^{-1}\eps^{-c} + \delta^{-2c}\leq \eps^{-c/2}$, and this finishes the proof of the induction step for \eqref{claim-IsT}.

It remains to show the claim~\eqref{eq:lambdaA}.
To do so, note that if  $\delta^2 (1 + T) \leq \delta$, that is if $T \leq \delta^{-1} - 1$, then by \eqref{eq:lambdacL},
\(
\lambda^{*\tau_\eps}(A) \leq 2^{-9}
\).
For such $T$, we can ensure $2^{-1} \leq (1 + T)^{-2 c}$ by asking $c \lll 1$ (depending on $\delta=\delta(\lambda)$).
In the other case where $T > \delta^{-1} - 1 \geq \delta^{-1}/2$, note that the triangle inequality implies for all $\phi \in A$, 
\[
\norm{\ttb_\phi} \geq \dist(0, \cL) - \dist(\ttb_\phi, \cL) \geq T - \delta^2 (1 + T) \geq T/2.
\]
In this case, we conclude using \eqref{eq:lambdattb}: for $c \lll 1$,
\[
\lambda^{*\tau_\eps}(A) \leq \lambda^{*\tau_\eps} \set{\phi \,:\, \norm{\ttb_\phi} \geq T/2} \leq 2^{-8} (1 + T)^{-2 c}.\qedhere
\]
\end{proof}

\bigskip

\subsection{H\"older regularity near subvarieties}
We upgrade \Cref{non-conc-sigma-aff} into a H\"older control on the concentration of $\sigma$ near algebraic subvarieties of $\R^d$. Given $l\in \N$, we set $\mathcal{P}_{d, l}$ the vector space of real polynomial functions on $\R^d$ of degree at most $l$. We equip $\mathcal{P}_{d, l}$ with the supremum norm on the coefficients, which we write $\|\cdot\|$.

\begin{thm}[H\"older regularity near subvarieties]\label{non-acc-var}
For every $l\in \N$, $P\in \mathcal{P}_{d,l}$  and $\eps >0$, we have
\[ \sigma \setbig{\bs\in \R^d\,:\, \abs{P(\bs)}\leq \eps \norm{P} } \leq C \eps^{c},\]
where $C,c>0$ depend only on $\sigma, l$.
 \end{thm}


The idea of the proof of \Cref{non-acc-var} is to use the self-similarity of $\sigma$ to write $\sigma$ as a convex combination of measures $(\sigma_{j})_{j}$ obtained from $\sigma$ by pushing via similarities, and with each $\sigma_{j}$ living at scale $\eps^{3/4}$, say roughly supported in  $B^{\R^d}_{\eps^{3/4}}+x_{j}$ for some $x_{j}$.
For each $\sigma_{j}$, we may approximate the polynomial map $P$ by its Taylor expansion up to order $1$ at $x_{j}$, and apply non-concentration near affine hyperplanes to deduce the estimate in \Cref{non-acc-var}.
In fact, to make this argument work, we also need to make sure that most of the $x_{j}$'s are located where the gradient $\nabla P$ of $P$ is not too small.
But $\nabla P$ is a polynomial of smaller degree, and the distribution of the $x_{j}$'s resembles that of $\sigma$, whence we may guarantee this using an inductive approach.
A related strategy is exploited in \cite[Section 7]{KLW04} in the context of absolutely decaying measures.

\begin{remark}
In the same manner, one can show that $\sigma$ is not concentrated near submanifolds $M$ of $\R^d$ such that $\dim M<d$, as long as $M$ is not too badly approximated by its tangent subspaces (e.g. if $\exp_{x}: B^{T_{x}M}_{1}\rightarrow M$ has uniformly bounded order $2$ derivatives for all $x$ in the manifold $M$).
\end{remark}


\begin{proof}[Proof of \Cref{non-acc-var}] We argue by induction on the degree $l$. The case $l=0$ is clear. We assume the result known for degrees $0, \dots, l-1$ and prove it for $l\geq 1$. We fix $P\in \cP_{d,l}$ with $\|P\|=1$, given $\eps \in (0, 1)$, we set
$$E_{\eps}:= \setbig{\bs \in \R^d: |P(\bs)|\leq \eps}.$$
We use the shorthand $B_{r}=B^{\R^d}_{r}$. Given a function $f:\R^d\rightarrow V$ where $V$ is a normed vector space, we write $\|f\|_{B_{r}}:=\sup_{x\in B_{r}}\|f(x)\|$ the supremum norm of the restriction $f_{|B_{r}}$. Note that the family $\|\cdot\|_{B_{r}}$ induces a collection of norms on $\cP_{d,l}$, and these norms are mutually equivalent by finite dimensionality of $\cP_{d,l}$.

Let $\alpha, \beta \in (0,1)$ be parameters to be specified later, with $\alpha$ depending only on $\sigma, l$, and $\beta$ absolute. For convenience, we will write $R:=\eps^{-\alpha}$ and $\delta =\eps^\beta$.
We also set
\[ F_{\delta}:=\setbig{\bs\in \R^d: \|\nabla P(\bs)\|\leq \delta}\]
where $\nabla P : \R^d\rightarrow \R^d $ refers to the gradient of $P$.

 We decompose $\sigma$ as a combination of measures living at scale $\eps^{3/4}$, and group them into 3 categories, distinguishing whether they are centered around a point outside of $B_{R}$, or within $B_{R}$, and in the second case whether the point is in $F_{\delta}$ or not.
More precisely, recalling the Lyapunov exponent $\Lyap$ from \eqref{def-Lyap}, we set $n=\lfloor \frac{3}{4\Lyap}  \abs{\log \eps} \rfloor$, so that for $\phi\sim \lambda^{*n}$, we have $\ttr_{\phi}$ close to $\eps^{3/4}$ most of the time.
We then have by the $\lambda$-stationarity of $\sigma$:
\begin{equation*}
\sigma (E_{\eps} ) = \int_{\Sim(\R^d)} \phi_{\star} \sigma (E_{\eps}) \dd \lambda^{*n}(\phi).
\end{equation*}
Decomposing the integral depending on the location of $\ttb_\phi$,
we obtain
\begin{equation*}
\sigma (E_{\eps} ) \leq I_{1}+I_{2}+I_{3},
\end{equation*}
where
\begin{align*}
&I_{1}= \lambda^{*n} \{ \phi \,:\, \ttb_\phi \notin B_{R}\},
\quad I_{2}=\lambda^{*n}\{\phi \,:\, \ttb_\phi \in B_{R}\cap F_{\delta} \} \quad \text{and}\\
& I_{3}=\int_{\ttb_\phi \notin F_{\delta}} \phi_{\star} \sigma (E_{\eps}) \dd \lambda^{*n}(\phi).
\end{align*}

By \Cref{moment-traj}, we have for some $\gamma=\gamma(\lambda)>0$,
\begin{align}\label{boundI1}
I_{1}\ll R^{-\gamma}.
\end{align}

\medskip
We now bound $I_{2}$. As a preliminary, note that the assumption $\|P\|= 1$ implies $\|\nabla P\|_{B_{R+1}}\ll_{l} R^l$. On the other hand, we may assume
\begin{align}\label{nablaPlowerbound}
\|\nabla P\|\gg_{l} R^{-(1+l)}.
\end{align}
To see why, note first that $\|P\|= 1$ implies $\sup_{B_{1}}|P|\geq \eta$ where $\eta=\eta(d,l)>0$.
Now if $\|\nabla P\|_{B_{R}}\leq R^{-1}\eta/4$, we would have $\inf_{B_{R}}|P|\geq \eta/2$.
Provided moreover that $\eps\leq \eta/4$, we would have $E_{\eps}\subseteq \R^d\smallsetminus B_{R}$, and the desired upper bound on $\sigma(E_\eps)$ follows from the finite moment of $\sigma$ (\Cref{sigma-moment}).
Hence we may suppose without loss of generality that $\norm{\nabla P}_{B_{R}}> R^{-1}\eta/4$, and \eqref{nablaPlowerbound} follows by passing to the norm $\norm{\cdot}$.


Let $c_{1}>0$. Note that the upper bound $\norm{\nabla P}_{B_{R+1}}\ll_{l} R^l$ implies that the $\eps^{c_{1}}$-neighborhood of $B_{R}\cap F_{\delta}$ is included in
$F_{\delta +O_{l}(\eps^{c_{1}}R^l)}$.
Choosing $c_{1}=c_{1}(\sigma)$ small enough, we deduce from \Cref{approx-sigma-n} that
\begin{align*}
I_{2} & \leq \lambda^{*n}*\delta_{0}(B_{R}\cap F_{\delta}) \leq \sigma(F_{\delta +O_{l}(\eps^{c_{1}}R^l)  }) +O(\eps^{c_{1}}).
\end{align*}
Applying the induction hypothesis and the lower bound \eqref{nablaPlowerbound} on $\|\nabla P\|$, we deduce
\begin{align}\label{boundI2}
I_{2} \ll_{l} R^{(1+l)c_{2}}(\delta + \eps^{c_{1}}R^l)^{c_{2}} + \eps^{c_{1}}
\end{align}
where $c_{2}=c_{2}(\sigma, l-1)>0$.

\medskip
We now bound $I_{3}$.
We further decompose the integral according to the value of $\ttr_\phi$.
By the choice of $n$ and the large deviation estimate (Crámer's theorem) for $(\log \ttr_{\phi})_{\phi \sim \lambda^{*n}}$,
there is $c_3 = c_3(\lambda) > 0$ such that
\[
\lambda^{*n}\setbig{\phi \,:\, \ttr_\phi \notin [\eps^{7/8}, \eps^{5/8}]} \ll \eps^{c_3}.
\]
Then
\[
I_3 = \int_{\ttb_\phi \notin F_\delta,\, \ttr_\phi \in [\eps^{7/8}, \,\eps^{5/8}]} \phi_{\star} \sigma (E_{\eps}) \dd \lambda^{*n}(\phi) \,+\, O(\eps^{c_3}).
\]

In the next step, we remove from $E_\eps$ anything outside the $\eps^{1/2}$-ball centered at $\ttb_\phi = \phi(0)$.
Indeed, for every $\phi$ satisfying $\ttr_\phi \leq \eps^{5/8}$, we have
\begin{align*}
\phi_*\sigma\bigl( \R^d \smallsetminus (\ttb_\phi + B_{\eps^{1/2}}) \bigr) &= \sigma\bigl( \R^d \smallsetminus \phi^{-1}(\ttb_\phi + B_{\eps^{1/2}}) \bigr)\\
&= \sigma \bigl(\R^d \smallsetminus B_{\ttr_\phi^{-1} \eps^{1/2}}\bigr)\\
&\leq \sigma \bigl(\R^d \smallsetminus B_{\eps^{-1/8}}\bigr).
\end{align*}
It follows then from \Cref{sigma-moment} that there is $c_4 = c_4(\sigma)$ such that
\[
\int_{\ttr_\phi \leq \eps^{5/8}} \phi_{\star} \sigma \bigl( \R^d \smallsetminus (\ttb_\phi + B_{\eps^{1/2}})\bigr) \dd \lambda^{*n}(\phi) \leq \sigma \bigl(\R^d \smallsetminus B_{\eps^{-1/8}}\bigr) \ll \eps^{c_4}.
\]
Thus,
\[
I_3 = \int_{\ttb_\phi \notin F_\delta,\, \ttr_\phi \geq \eps^{7/8}} \phi_{\star} \sigma \bigl(E_{\eps} \cap (\ttb_\phi + B_{\eps^{1/2}}) \bigr) \dd \lambda^{*n}(\phi) + O(\eps^{c_3} + \eps^{c_4}).
\]

Consider $\phi$ as in the above integral and $\bs\in E_{\eps}\cap (\ttb_\phi + B_{\eps^{1/2}})$.
By Taylor's expansion and the fact that $\norm{\cdot}_{C^2(B_{R+1})}\ll_{l} R^l\norm{\cdot}$ on $\cP_{d,l}$, the conditions $\|P\|=1$ and $\|\bs-\ttb_\phi\|<\eps^{1/2}$ imply
$$\abs{P(\bs) -P(\ttb_\phi) -\big\langle\nabla P(\ttb_\phi) ,\,\, \bs-\ttb_\phi \big\rangle} \ll_{l} \eps R^l.$$
As $\abs{P(\bs)} \leq \eps$, setting $v_{P, \phi}:= P(\ttb_\phi) - \big\langle \nabla P(\ttb_\phi), \,\,\ttb_\phi \big\rangle$, we deduce
$$ \abs{ \big\langle \nabla P(\ttb_\phi) , \,\,\bs \big\rangle - v_{P,\phi} } \ll_{l} \eps R^l$$
But $\ttb_\phi \notin F_\delta$ means $\norm{\nabla P(\ttb_\phi)} > \delta$, whence $\bs$ belongs to the $O_{l}(\eps R^l \delta^{-1})$-neighborhood of an affine hyperplane. Applying $\phi^{-1}$ then the non-concentration near affine subspaces from \Cref{non-conc-sigma-aff}, and remembering $\ttr_\phi \geq \eps^{7/8}$, we obtain
$$\phi_{\star} \sigma \bigl(E_{\eps} \cap (\ttb_\phi + B_{\eps^{1/2}}) \bigr) \ll_{l} (\ttr_{\phi}^{-1} \eps R^l \delta^{-1})^{c_5} \leq (\eps^{1/8} R^l\delta^{-1})^{c_5}$$
where $c_{5}=c_{5}(\sigma)>0$.

We can conclude for $I_3$:
\begin{equation}\label{boundI3}
I_{3}\ll_{l}(\eps^{1/8} R^l\delta^{-1})^{c_5} + \eps^{c_{3}} + \eps^{c_4}.
\end{equation}

\medskip

In the end, combining \eqref{boundI1}, \eqref{boundI2}, \eqref{boundI3}, and choosing $\delta=\eps^{1/16}$, $R=\eps^{-\alpha}$ with $\alpha \lll_{\sigma,l} 1$, we have proven that for $c'\lll_{\gamma,c_{1}, \dots, c_{5}}1$,
$$\sigma (E_{\eps} ) \ll_{l} \eps^{c'}.$$
This proves the induction step, whence the theorem.
\end{proof}

It is easy to deduce from the previous theorem that a product $\sigma^{\otimes k}$ is not concentrated near subvarieties of $\R^{dk}$. We record this observation for future use.

\begin{corollary}\label{non-acc-var-product}
Let $k, l\geq 1$. Let $P:\R^{dk}\rightarrow \R$ be a polynomial map of degree at most $l$. Then for every $\eps>0$,
$$\sigma^{\otimes k}\setbig{(\bs_{i})_{i=1}^k \in (\R^d)^k\,:\, |P(\bs_{1}, \dots, \bs_{k})| \,\leq\, \eps \|P\|}\leq C\eps^c $$
where $C,c>0$ depend only on $\sigma, k,l$.
\end{corollary}

\begin{proof} We consider $l$ fixed and argue by induction on $k$.
The case $k=1$ is \Cref{non-acc-var}.
Now given $k\geq 2$, we assume the result holds for $k-1$, with exponent  $c=c(\sigma, k - 1, l)$, and we prove it for $k$.

We may assume $\|P\|\geq 1$. We view $P$ as a polynomial in the indeterminates $(\bs_1, \dotsc, \bs_{k-1})$ and with coefficients in the ring $\R[\bs_k]$.
Since $\norm{P} \geq 1$, one of the coefficients $Q \in \R[\bs_k]$ (of some monomial in $(\bs_1, \dotsc, \bs_{k-1})$) satisfies $\norm{Q} \geq 1$ (as an element of $\cP_{d,l}$).

For any individual $\bs_k \in \R^d$, the polynomial map
\[
\begin{array}{ccc}
\R^{d(k-1)} &\to & \R\\
(\bs_1, \dotsc, \bs_{k-1}) & \mapsto & P(\bs_1, \dotsc, \bs_{k-1}, \bs_k)
\end{array}
\]
seen as an element of $\cP_{d(k-1),l}$ has norm at least $\abs{Q(\bs_k)}$.
Thus, by the induction hypothesis, whenever $\abs{Q(\bs_{k})} \geq \eps^{1/2}$, we have
\[
\sigma^{\otimes (k - 1)} \setbig{(\bs_{i})_{i=1}^{k - 1} \in (\R^d)^{k - 1} \,:\, \abs{P(\bs_{1}, \dots, \bs_{k})} \leq \eps } \ll_{k,l} \eps^{c/2}
\]

We conclude by the Fubini-Lebesgue theorem and the case $k = 1$ applied to $Q$:
\begin{align*}
 \sigma^{\otimes k}\setbig{(\bs_{i})_{i=1}^k & \in (\R^d)^k\,:\, \abs{P(\bs_{1}, \dots, \bs_{k})} \leq \eps } \\
&=  \int_{\R^d} \sigma^{\otimes (k - 1)} \setbig{(\bs_{i})_{i=1}^{k - 1} \in (\R^d)^{k - 1} \,:\, \abs{P(\bs_{1}, \dots, \bs_{k})} \leq \eps } \dd \sigma(\bs_k) \\
&\ll_{k,l}   \eps^{c/2} + \sigma\setbig{\bs_k \in \R^d\,:\, \abs{Q(\bs_k)} < \eps^{1/2} } \\
&\ll_{k,l}   \eps^{c/2}+ \eps^{c(\sigma, 1, l)/2}.
\end{align*}
This proves the case $k$ (with exponent $c(\sigma, k, l) = c(\sigma, k-1,l) /2$).
\end{proof}

\subsection{Finite time consequences}
Recalling that the finite time approximation $\sigma^{(n)}:=\lambda^{*n}*\delta_{0}$ converges to $\sigma$ exponentially fast, we may transfer the regularity properties of $\sigma$ to $\sigma^{(n)}$ provided we look at scales above an exponentially small threshold.

\begin{lemma}\label{finite-time-reg}
For $\gamma\lll 1$ and all $n\geq 1$, we have
\begin{itemize}
\item[(i)]
$\forall \eps>e^{-n}, \quad \sup_{\bs\in \R^d} \sigma^{(n)}(\bs+[-\eps, \eps]^d) \ll \eps^\gamma.$
\end{itemize}
Moreover, for $l\geq 1$, $c\lll_{l} 1$, $P\in \cP_{d,l}$ with $\norm{P} = 1$, $n\geq 1$, $\eps>e^{-n}$, we have
 \begin{itemize}
 \item[(ii)]
$\sigma^{(n)} \setbig{ \bs \in \R^d\,:\, \abs{P(\bs)} \leq \eps }\ll_{l} \eps^c.$
\end{itemize}

\end{lemma}

\begin{proof}
In view of \Cref{approx-sigma-n}, item (i)  follows from  \Cref{non-conc-sigma-aff}. For item (ii), given $R>1$, $\eps>0$, set $E_{R, \eps}:=\set{ \bs \in B_{R}\,:\, \abs{P(\bs)} \leq \eps }$ where $B_{R}:=B^{\R^d}_{R}$.
By \Cref{approx-sigma-n}, we have for $1\ggg \eps>e^{-n}$, for some $\gamma=\gamma(\lambda)\in (0,1)$,
$$\sigma^{(n)}(E_{R, \eps}) \leq \sigma(E_{2R, \,\eps+\eps^\gamma \|\nabla P\|_{B_{2R}}}) + e^{-\gamma n}. $$
Observe $\|\nabla P\|_{B_{2R}}\ll R^l$. Therefore, taking $R=\eps^{-\alpha}$ with $\alpha>0$, we have by \Cref{non-acc-var}
$$ \sigma(E_{2R, \,\eps+\eps^\gamma \|\nabla P\|_{B_{2R}}}) \ll_{l} \eps^{(\gamma-l\alpha)c}$$
where $c=c(\sigma, l)>0$. The result follows by taking $\alpha\lll_{l} 1$, and applying \Cref{moment-traj}  to allow restriction to $B_{\eps^{-\alpha}}$.
\end{proof}

\medskip

\section{Effective recurrence of the $\mu$-walk}\label{Sec-recurrence}

In this section, we establish that the $n$-step distribution of the $\mu$-random walk on $X$ is not concentrated near infinity, provided $n$ is large enough in terms of the starting point. We recall notations have been set up in \Cref{Sec-notations}, in particular $\inj(x)$ denotes the injectivity radius of $X$ at the point $x$.

\begin{proposition} \label{effective-recurrence}
There exist constants $C,c>0$ such that for every $x\in X$, $n\in \N$, $\rho>0$,
  \[\mu^{*n}*\delta_x \{\inj \leq \rho\} \ll \rho^{c}(e^{-c n}\inj(x)^{-C}+1).\]
\end{proposition}

For $d=1$, a short self-contained proof is given in \cite[Section 2.3]{BHZ24}. For arbitrary $d$, we explain how to deduce \Cref{effective-recurrence} from Prohaska-Sert-Shi \cite{PSS23}, which is itself inspired by the works of Benoist-Quint \cite{BQ4}, Eskin-Margulis \cite{EM04}, Eskin-Margulis-Mozes \cite{EMM98}.

\begin{lemma}\label{Zariski closure}
Denote by $H_{\mu}$ the Zariski closure of the semigroup generated by $\supp \mu$. Then $U\subset H_{\mu}$.
\end{lemma}

\begin{proof}
Recall that every $g\in P'$ can be written uniquely as $g=k_g^{-1} a(\ttr_g^{-1}) u(\ttb_g)$, where $k_g=\diag(O_{g},1)\in K', \ttr_g>0, \ttb_g\in \R^d$. We first deal with a particular case of the lemma.

\emph{Case $(*)$: there exists $ g'\in \supp \mu$ such that $\ttr_{g'}\in (0, 1)$ and $\ttb_{g'}=0$}. In this case, choose a sequence $n_i\to +\infty$ such that $k_{g'}^{-n_i}\to \Id$ as $i\to +\infty$. Then for every $g\in \supp \mu$, we have $\lim_{i\to +\infty} g'^{-n_i} g g'^{n_i}=k_g^{-1} a(\ttr_g^{-1})$,
 from which it follows that
\begin{align} \label{Hmu-contains}
k_g^{-1} a(\ttr_g^{-1}) \in H_{\mu}, \quad\quad \quad u(\ttb_g)\in H_{\mu}.
\end{align}
 Write $S$ the set of vectors $\bs \in \R^d$ such that $u(\bs)\in H_{\mu}$. Note $S$ is a Zariski-closed subgroup of $\R^d$, i.e., $S$ is a subspace. Using that $U\cap H_{\mu}$ is normalized by $H_{\mu}$, Equation \eqref{Hmu-contains}, and the relation $a(\ttr_g)k_g u(\bs)k_g^{-1} a(\ttr_g^{-1})u(\ttb_g)=u(\ttr_{g} O_{g}\bs+\ttb_{g})$,
 we get that $S$ is invariant under $\supp \lambda$. By irreducibility of $\lambda$, we deduce $S=\R^d$. This finishes the proof of Case $(*)$.

\emph{General case}. We now reduce the general case to Case $(*)$. As $\lambda$ is contractive in average (see discussion after \eqref{def-Lyap}), there exists $g'\in \supp \mu$ such that $\ttr_{g'}\in (0,1)$. In particular, the vector
 $$\bs_{0} =(\Id_{\R^d}-\ttr_{g'}O_{g'})^{-1}\ttb_{g'}$$
 is well defined.
By direct computation, we find that
\[u(\bs_{0} )g'u(-\bs_{0} )=k_{g'}^{-1}a(\ttr_{g'}^{-1}).\]
The measure $\mu':=\delta_{u(\bs_{0} )}*\mu*\delta_{u(-\bs_{0} )}$ on $P'$ corresponds to some   randomized self-similar IFS $\lambda'$. Namely, $\lambda'$ is the pushfoward of $\lambda$ under conjugation by $\R^d\rightarrow \R^d, v\mapsto v-s_{0}$; in particular $\lambda'$ is irreducible with finite exponential moment. By the analysis of Case $(*)$, we know that $H_{\mu'}\supseteq U$. On the other hand, $H_{\mu'}= u(\bs_{0} )H_{\mu}u(-\bs_{0} )$, whence $H_{\mu}\supseteq U$ as well.
\end{proof}

The previous lemma allows us to apply \cite{PSS23} to obtain some Margulis function, i.e., a proper positive function on $X$ which is uniformly contracted by the random walk and satisfies some growth control under the action of $G$.

\begin{lemma}[Height function \cite{PSS23}]\label{height function}
  There exists a function $\beta:X \to [1,+\infty)$ which is proper and satisfies:
\begin{itemize}
  \item[(1)] Contraction property: There exist $m\in \N$ and $\theta, M>0$ such that for all $x\in X$,
  \[\mu^{*m} *\delta_x(\beta)\leq e^{-\theta } \beta(x)+M.\]

  \item[(2)] Growth control: $\beta(gx)\leq \|\Ad(g)\|^{O(1)}\beta(x)$ for all $ g\in G, x\in X$.

\end{itemize}
\end{lemma}

\begin{proof}

The combination of \Cref{Zariski closure} and \cite[Corollary 3.8]{PSS23} guarantees that $\mu$ is $G$-expanding in the sense of \cite[Definition 2.7]{PSS23}. This allows to apply \cite[Theorem 6.1]{PSS23} to obtain the desired function $\beta$.


\end{proof}

 We justify that the above Margulis function can be compared to the injectivity radius.

\begin{lemma}\label{inj-comparison}
For any proper function $\Upsilon:X\to [1,+\infty)$ satisfying properties (1) (2) of \Cref{height function}, there exist $C,c>0$ such that for all $x\in X$,
  \[\Upsilon(x)^{-C}\ll \inj(x) \ll \Upsilon(x)^{-c}.\]
\end{lemma}

\begin{proof}
Using the assumptions on $\Upsilon$, the comparison between $\inj(x)$ and $\Upsilon(x)$ follows from the same argument as for \cite[Lemmas 3.13, 3.14]{BH24}.
\end{proof}

We are now able to conclude the proof of the effective recurrence property.

\begin{proof}[Proof of \Cref{effective-recurrence}]
Let $\beta$ and $m, \theta, M$ as in \Cref{height function}. We first replace $\beta$ by a suitable $\beta'$  satisfying the contraction property with $m=1$. For that, let $\gamma, \kappa\in (0,1)$ be  parameters to be specified below.
Given $f:X\rightarrow \R_{\geq 0}$, set $P_{\mu}f=\int_{G}f(g \, \cdot)\dd \mu(g)$. Consider $\beta'=\beta^\gamma +e^{\kappa}P_{\mu}(\beta^\gamma)+\dots + e^{(m-1)\kappa}P_{\mu}^{m-1}(\beta^\gamma)$ and observe it is also proper.

Note that \Cref{height function} (2) implies $\beta^\gamma(gx)\leq \|\Ad(g)\|^{O(\gamma)}\beta^\gamma(x)$.  Recalling $\mu$ has finite exponential moment and choosing $\gamma$ small enough accordingly, we obtain that $\beta'$ takes finite values, and has controlled growth:
$$\beta'(gx)\leq \|\Ad(g)\|^{O(1)}\beta'(x).$$
Note also that $\beta^\gamma$ satisfies the contraction property of \Cref{height function} (1) with the parameters $\theta,M$ replaced by $(\gamma\theta,  M^\gamma)$.
It follows that taking $\kappa:=\gamma\theta/m$, we have 
\begin{equation}\label{contract-m=1}
P_{\mu}\beta'\leq e^{-\kappa} \beta' +e^{(m-1)\kappa} M^\gamma.
\end{equation}
Now, iterating \eqref{contract-m=1}, we find for every $n \geq 0$,
$$P_{\mu}^n\beta' \leq e^{-\kappa n} \beta' +M' $$
where $M'=e^{(m-1)\kappa}M^\gamma/(1-e^{-\kappa})$.

Using the Markov inequality, we deduce for every $\rho>0$,
\begin{equation*}\label{beta'Mark}
\mu^{*n}*\delta_x\set{y \,:\, \beta'(y)> \rho^{-1}} \leq (e^{-\kappa n}\beta'(x)+M')\rho.
\end{equation*}
The proposition then follows from the comparison \Cref{inj-comparison} applied to $\beta'$.
\end{proof}

As a direct corollary of \Cref{effective-recurrence}, we bound the Haar measure of cusp neighborhoods. This estimate  will  be useful in \Cref{Sec-Khintchine-dich}.
\begin{lemma}
\label{lm:volCusp}
There are constants $C > 1$ and $c > 0$ depending on $\Lambda$ such that
\begin{enumerate}
\item for all $\rho > 0$,
\[
m_X\{ \inj \leq \rho \} \leq C \rho^c;
\]
\item writing $x_0 = \Lambda/\Lambda$ for the basepoint of $X$, we have for all $r > 0$,
\[
m_X\{ \dist( \cdot, x_0) \geq r \} \leq C e^{-c r}.
\]
\end{enumerate}
\end{lemma}

\begin{proof}
The Haar measure $m_X$ is an ergodic $\mu$-stationary measure.
Hence for $m_X$-almost every $x \in X$, the sequence $(\mu^{*n} * \delta_x)_{n\geq 0}$ converges to $m_X$ in Cesàro average with respect to the weak-$*$ topology.
Then the first estimate follows immediately from \Cref{effective-recurrence}. 
Note the constants $C$, $c$ only depend on $\Lambda$ (and not $\mu$) because $\mu$ does not play a role in the statement.

By  \cite[Lemma 3.14]{BH24}, we have for $x \in X$,
\begin{equation}
\label{eq:cominjdist}
\abs{\log \inj(x)} -1 \ll \dist(x, x_{0}) \ll \abs{\log \inj(x)} + 1.
\end{equation}
The second estimate  then follows from the first.
\end{proof}

\section{Positive dimension}\label{Sec-postdim}
In this section, we show that the $n$-step distribution of the $\mu$-random walk has positive dimension provided we look at scales above an exponentially small threshold and $n$ is large enough in terms of the starting point.

\begin{proposition}[Positive dimension]\label{Proposition: initial dimension}
There exists $A, \kappa>0$ such that for every $x\in X$, $\rho>0$, $n\geq \abs{\log \rho} + A \abs{\log \inj(x)}$, we have
$$ \forall y\in X,\quad \mu^{*n}*\delta_{x}(B_{\rho}y) \ll \rho^\kappa.$$
\end{proposition}

The case $d=1$ corresponds to \cite[Proposition 3.1]{BHZ24}. For arbitrary $d$, the argument of \cite{BHZ24} goes through with a few adaptations to deal with the rotation component $K'$. We provide the proof for completeness.

\begin{proof}
Let $\kappa>0$ be a parameter to be specified later. Let $x\in X$, $\rho \in (0, 1)$, $n\geq  \abs{\log \rho}$. Assume by contradiction that there exists some $y\in X$ such that
\begin{align}\label{postd-eq0}
  \mu^{*n}*\delta_x (B_{\rho}y)> \rho^{\kappa}.
\end{align}
Let $\alpha=\frac{1}{10(\ell+1)}$ where $\ell>0$ was defined in  \eqref{def-Lyap} as the top Lyapunov exponent of $\Ad_{\star}\mu$. Set $m=\lfloor \alpha \abs{\log \rho} \rfloor$.
Write
\[\mu^{*n}*\delta_x=\mu^{*m}*\mu^{*(n-m)}*\delta_x\]
and
\[Z:=\setbig{z\in X \,:\, \mu^{*m}*\delta_z(B_{\rho}y)\geq \rho^{2\kappa}}.\]
Then (\ref{postd-eq0}) implies that $\mu^{*(n-m)}*\delta_x(Z)\geq \rho^{2\kappa}$, provided $\rho\lll_{\kappa} 1$.

We are going to show that points in $Z$ have small injectivity radius. Fix $z\in Z$. By definition we have
\begin{align}\label{posd-eq1}
  \mu^{*m}\set{g: g z\in B_{\rho} y}\geq \rho^{2\kappa}.
\end{align}

Our first step is to see that among those $g$ taking $z$ into $B_{\rho} y$, we may impose further properties without ruining the measure estimate (precise statement in \eqref{post-eq4} below).
First, by the large deviation principle, there exists $\epsilon=\epsilon(\mu)>0$ such that for $\rho \lll1$,
\begin{align}\label{posd-eq3}
  \mu^{*m}\set{g \,:\, \log \ttr_g\in [-(\Lyap+1)m,-(\Lyap-1)m]}\geq 1-\rho^{\alpha \epsilon}.
\end{align}
Furthermore, considering $\gamma=\gamma(\mu)>0$ as in \Cref{moment-traj}, we have for all $\rho\lll_{\kappa} 1$,
\begin{align}\label{posd-eq2}
  \mu^{*m}\set{g \,:\, \norm{\ttb_g}\leq \rho^{-4\gamma^{-1}\kappa}}\geq 1-\rho^{3\kappa}.
\end{align}
Let $C>1$ be a large parameter to be specified below depending on $\mu$ only.
Partition the product set
$$K'\times [-(\Lyap+1)m,-(\Lyap-1)m] \times [-\rho^{-4\gamma^{-1}\kappa},\rho^{-4\gamma^{-1}\kappa}]^d$$
into subsets $(S_{i})_{i\in I}$ of the form $S_{i}=S_{i,1}\times S_{i,2}\times S_{i,3}$ where each $S_{i,j}$ has diameter less than $\rho^{C\kappa}$. Note we can arrange the number of elements in the partition to be controlled via
$$|I|\ll \rho^{-d(d-1)C\kappa/2} \cdot m\rho^{-C\kappa} \cdot \rho^{-4d\gamma^{-1}\kappa} \rho^{-dC\kappa}.$$
By the pigeonhole principle and \eqref{posd-eq1}, \eqref{posd-eq3}, \eqref{posd-eq2}, there exists $i_{0}\in I$ such that the set
\[E:=\set{g \,:\, g z\in B_{\rho} y \text{ and } (k_{g}, \log \ttr_{g}, \ttb_{g})\in S_{i_{0}} }\]
satisfies
\begin{align}\label{post-eq4}
  \mu^{*m}(E)\geq \frac{\rho^{2\kappa}-\rho^{\alpha \epsilon}-\rho^{3\kappa}}{|I|}\geq \rho^{10 d^2 C\kappa},
\end{align}
provided that $C\geq \gamma^{-1}$, $3\kappa\leq \alpha\eps$ and $\rho\lll_{\kappa} 1$.

Let $g_1,g_2\in E$. Note that $\dist(g_1 z, g_2 z)\ll \rho$. By the choice of $m$ and the bounds on $\ttr_{g_{1}}, \ttb_{g_{1}}$, we have $\norm{\Ad(g_1^{-1})}\leq \rho^{-1/2}$ provided $\kappa\lll 1$. It follows that
\begin{align}\label{post-eq-5}
  \dist(z,g_1^{-1}g_2 z)\ll \norm{\Ad(g_1^{-1})}\rho\ll \rho^{1/2}.
\end{align}
Using that $K'$ and $A'$ commute, we can further write $g_1^{-1} g_2$ as
\[ g_1^{-1}g_2=u(-\ttb_{g_2}) h u(\ttb_{g_2}) \quad \text{where } h=u(\ttb_{g_2}-\ttb_{g_1})a(\ttr_{g_1}\ttr_{g_2}^{-1})k_{g_1}k_{g_2}^{-1}.\]

To deduce from \eqref{post-eq-5} an estimate on the injectivity radius of $X$ at $z$, we now show that $g_{1},g_{2}$ can be chosen so that the distance between $g_1^{-1}g_2$ and $\Id$ is much bigger than $\rho^{1/2}$ (say at least $\rho^{1/4}$), but still smaller than a  power of $\rho$.
Assume $\gamma\lll1$ so that the non-concentration estimate from \Cref{finite-time-reg} (i) holds. Combining it with (\ref{post-eq4}), we can choose $g_1,g_2\in E$ such that
\begin{align}\label{post-eq7}
  \norm{\ttb_{g_1}-\ttb_{g_2}}\geq \rho^{11\gamma^{-1}d^2 C\kappa },
\end{align}
provided $\kappa\ll_C 1$ to ensure that $\rho^{11\gamma^{-1} d^2 C\kappa}\geq e^{-m} \simeq \rho^{\alpha}$ as required in \Cref{finite-time-reg}, and $\rho \lll_{\kappa}1$. For $g_1,g_2\in E$ satisfying \eqref{post-eq7}, we have
\begin{align*}
  \dist(h,\Id)\simeq \norm{k_{g_1}- k_{g_2}}+\norm{\ttb_{g_1}-\ttb_{g_2}}+|1-\ttr_{g_1}\ttr_{g_2}^{-1}|
  \in [\rho^{11\gamma^{-1}d^2 C\kappa}, 10\rho^{C\kappa}].
\end{align*}
Recalling that $\norm{\ttb_{g_2}}\leq \rho^{-4\gamma^{-1}\kappa}$, we get
\begin{align}\label{post-eq8} \rho^{1/4}\ll\rho^{(11d^2C+4(d+1))\gamma^{-1}\kappa}\ll\dist(g_1^{-1}g_2, \Id)\ll \rho^{(C-4(d+1)\gamma^{-1})\kappa},
\end{align}
where the lower bound assumes $\kappa\ll_C 1$. Provided $C>8(d+1)\gamma^{-1}$, Equations \eqref{post-eq-5}, \eqref{post-eq8} yield $\inj(z)\ll \rho^{C\kappa/2}+\rho^{1/2}$. When $\kappa \lll_C 1$ and $\rho \lll_\kappa 1$, this gives
\[
\inj(z) \leq \rho^{C \kappa/4}.
\]

In conclusion, we have shown that for $C\ggg1$, for $\kappa\lll_{C} 1$, $\rho\lll_{\kappa}1$, and $n \geq m =\lfloor \alpha \abs{\log \rho} \rfloor$, we have
\[\mu^{*(n-m)}*\delta_{x} \{ \inj \leq \rho^{C\kappa/4}\} \geq \rho^{2\kappa}. \]
By the effective recurrence statement from \Cref{effective-recurrence}, this is absurd if $n-m\ggg { \abs{\log \inj(x)}}$.
The proof of the proposition is complete.
\end{proof}

\section{Dimensional Bootstrap}\label{Sec-bootstrap}

In this section, we show that the $n$-step distribution $\mu^{*n}*\delta_{x}$ becomes high dimensional in $X$ exponentially fast as $n$ goes to infinity. The next definition will be useful.

\begin{definition}[Robust measure] \label{robust-measure}
Let $\alpha > 0, \tau \geq 0$ and $I\subseteq (0, 1]$.
A Borel measure $\nu$ on $X$ is said to be \emph{$(\alpha, \cB_{I}, \tau)$-robust} if $\nu$ can be written as the sum of two Borel measures $\nu=\nu'+\nu''$ such that $\nu''(X)\leq \tau$, and $\nu'$ satisfies
\begin{equation} \label{eq:injrad}
 \nu'\{\inj < \sup I \}=0,
 \end{equation}
as well as for all $ \rho\in I$, $y\in X $,
\begin{equation}
\label{eq:rob1}
\nu'(B_ {\rho}y)\leq \rho^{\alpha \dim X}
\end{equation}
 If $I$ is a singleton $I=\{\rho\}$, we simply say that $\nu$ is $(\alpha, \cB_{\rho}, \tau)$-robust.
\end{definition}

We aim to show the following. 
 \begin{proposition}[High dimension] \label{high-dim}
Let $\kappa \in (0, 1/10)$. For $\eta, \rho \lll_{\kappa}1$ and
 for all $n \ggg_{\kappa}  \abs{\log \rho} +  \abs{\log \inj(x)}$, the measure {$\mu^{*n}*\delta_{x}$} is $(1-\kappa, B_ {\rho}, \rho^\eta)$-robust.
\end{proposition}

\begin{remark}
The proof of \Cref{high-dim} yields a slightly more precise lower bound requirement on $n$: we only need $n\geq M\abs{\log \rho} +  A\abs{\log \inj(x)}$ where $M\ggg_{\kappa}1$ and $A\ggg1$. See the end of \S\ref{Sec-diminc-proof}.
\end{remark}

The strategy to prove \Cref{high-dim} is to show that convolution by $\mu$ (or a suitable power $\mu^{*n}$) improves the dimensional properties of any given Frostman measure $\nu$ on $X$. Iterating this phenomenon allows to reach high dimension. The dimensional increment property for random walks is rooted in the following key observation:
$$\mu^{*n}*\nu(B_{\rho}x)=\int_{G}\nu(g^{-1}B_{\rho}x)\,\dd\mu^{*n}(g)$$
and
$g^{-1}B_{\rho}x$ can be seen, in some exponential chart, as a Euclidean box $\Ad(g^{-1})B^\kg_{\rho}$ varying randomly with $g\sim \mu^{*n}$. Provided this random box satisfies suitable non-concentration properties, we can then derive a small dimensional increment via a multislicing theorem (which itself boils down to the sum product phenomenon). In \S\ref{Sec-nc-ineq}, we study the  non-concentration properties at our disposal. In  \S\ref{Sec-mult}, we develop the relevant multislicing machinery in an abstract Euclidean setting. In \S\ref{Sec-lin-charts},  we present suitable linearizing charts to allow application in the context of our homogeneous space $X$. In \S\ref{Sec-diminc-proof} we put all these results together to obtain \Cref{high-dim}.

\subsection{Non-concentration inequalities} \label{Sec-nc-ineq}

We establish  non-concentration properties for the partial flag carrying the box $\Ad(g^{-1})B^\kg_{\rho}$ as $g$ varies randomly with law $\mu^{*n}$.

First, let us see what this partial flag is explicitly. Consider the weight spaces decomposition $\kg=\kg_{-}\oplus \kg_{0}\oplus \kg_{+}$ with respect to $A'$.
More precisely, $\kg_{+}, \kg_{-}$ are respectively the Lie algebras of $U$ and $U^-$, where $U^-$ denotes the transpose of $U$, and $\kg_{0}$ is their orthogonal complement in $\kg$. Recalling $g=k_{g}^{-1}a(\ttr_{g}^{-1})u(\ttb_{g})$ and the norm on $\kg$ is $\Ad(K')$-invariant, and assuming $\ttr_{g}<1$, 
the box $\Ad(g^{-1})B^\kg_{\rho}$ can be written
\begin{align}
\Ad(g^{-1})B^\kg_{\rho} &=\Ad(u(-\ttb_{g}))\Ad(a(\ttr_{g}))B^\kg_{\rho} \nonumber\\
&\overset{O(1)}{\simeq} \Ad(u(-\ttb_{g}))\left(B^{\kg_{-}}_{\ttr^{-1}_{g}\rho} \oplus B^{\kg_{0}}_{\rho} \oplus B^{\kg_{+}}_{\ttr_{g}\rho}\right) \nonumber\\
&\overset{L_{g}}{\simeq} \left(B^{\Ad(u(-\ttb_{g}))\kg_{-}}_{\ttr^{-1}_{g}\rho} \oplus B^{\Ad(u(-\ttb_{g}))\kg_{0}}_{\rho} \oplus B^{\Ad(u(-\ttb_{g}))\kg_{+}}_{\ttr_{g}\rho}\right) \nonumber\\
&\overset{L_{g}}{\simeq} B^{\Ad(u(-\ttb_{g}))\kg_{-}}_{\ttr^{-1}_{g}\rho} + B^{\Ad(u(-\ttb_{g}))\kg_{\leq 0}}_{\rho} + B^{\kg}_{\ttr_{g}\rho} \label{comparison-box}
\end{align}
where $\kg_{\leq 0}:=\kg_{-}\oplus\kg_0$, $L_{g}=O(\|\Ad(u(\pm \ttb_{g}))\|)^{\dim \kg}=(2+\|\ttb_{g}\|)^{O(1)}$, and the notation $A \overset{L }{\simeq} B$ means that $A$ can be covered by less than $L$ additive translates of $B$, and conversely.
Note the norm $\|\ttb_{g}\|$ is controlled via \Cref{moment-traj}. We are left to examine the non-concentration properties of the partial flag given by 
\[V_{1}(g):=\Ad(u(-\ttb_{g}))\kg_{-} \quad \text{ and } \quad V_{2}(g):=\Ad(u(-\ttb_{g}))\kg_{\leq 0}\]
as $g$ varies with law $\mu^{*n}$.

In \cite{BHZ24} about the case $d=1$ (as well as in \cite{BH24}), a similar approach is exploited, but the non-concentration at disposal therein is very strong, namely: for $i=1,2$, any subspace $W\subseteq \kg$ with $\dim V_{i}+\dim W=\dim \kg$, for $\mu^{*n}$-many $g$, we have $V_{i}(g)\cap W=\{0\}$ with a large angle between $V_{i}(g)$ and $W$. 
This non-concentration requirement appeared naturally, as it is the hypothesis of the projection theorems~\cite{Bourgain2010,He2020JFG} that were previously put into the multislicing machinery of \cite{BH24}.
Unfortunately, such property fails for $d\neq 1$, as we see in the next lemma.

\begin{lemma}[Obstacle]\label{obstacle} Assume $d\geq 2$.
Let $W= \set{ M \in \kg \,:\, M e_1 = 0}$ be the subspace of matrices in $\M_{d+1}(\R)$ with zero trace and null first column. Then $\codim_{\kg} W=d+1> \dim \kg_{-}$ but
for every $g\in G$, we have
$$\Ad(g)\kg_{-}\cap W\neq \{0\}. $$
\end{lemma}

Here and below, we denote by $(e_1, \dotsc, e_{d + 1})$ the standard basis of $\R^{d + 1}$.

\begin{proof}
Observe that
$$\kg_{-} = \Span\set{E_{d+1, j} \,:\, 1 \leq j \leq d }  =\begin{pmatrix}
     0 & \cdots & 0 & 0\\
     \vdots & \ddots & \vdots & \vdots\\
     0 & \cdots & 0 & 0\\
     * & \cdots & * & 0
  \end{pmatrix}.$$
Therefore, for any $g \in G$, $\Ad(g)\kg_-$ corresponds to the collection of endomorphisms of $\R^{d+1}$ which are $2$-step nilpotent and with image in $g\R e_{d+1}$. Consider $m, m'\in \Ad(g)\kg_{-}$ non colinear. As $m e_{1}, m' e_{1}$ are colinear, there must exist $(t, t')\in \R^2\smallsetminus \{0\}$ with $(tm+ t'm')e_{1}=0$, whence $tm+t' m'\in W$. This justifies that any $2$-dimensional subspace of $\Ad(g)\kg_-$ intersects $W$, whence $\dim (\Ad(g)\kg_- \cap W)\geq d-1$.
\end{proof}

In this section, we consider an arbitrary $d\geq 1$ and show that the random subspaces $V_{1}(g), V_{2}(g)$ where $g\sim \mu^{*n}$ still satisfy a \emph{weak} form of non-concentration. It is presented as \Cref{relative-angle} below. We will explain afterwards how this can be utilized to perform the dimensional bootstrap.

Given subspaces $F_{1}, \dots, F_{k}\in \Gr(\kg)$, we write
$$\norm{F_{1}\wedge \dots \wedge F_{k}}  := \norm{ v_{1}\wedge \dots \wedge v_{k}} $$
where $v_{i}\in \bigwedge^* \kg$ is a unit vector spanning the line $\bigwedge^{\dim F_{i}}F_{i}.$

\begin{proposition}[Mild non-concentration] \label{relative-angle}
Let $W\in \Gr(\kg, d)$. Then for $n\geq 1$, $r\geq e^{-n}$,
$$(\mu^{*n})^{\otimes d+1}\left\{ (g_{i})_{i=1}^{d+1}\,:\, \| V_{1}(g_{1})\wedge\dots \wedge V_{1}(g_{d+1})\wedge W\| \leq r \right\} \leq C r^c $$
where $C, c>0$ are constants depending on $\mu$ only.
\end{proposition}

Observing $\dim \kg= d(d+2)$, \Cref{relative-angle} means that for most parameters $(g_{i})_{i=1}^{d+1}$ selected by $(\mu^{*n})^{\otimes d+1}$,
we have $\kg=\bigoplus_{i}V_{1}(g_{i}) \oplus W$, and each subspace makes a rather large angle with the complementary sum.

We may derive a similar non-concentration property for $V_{2}(g)^\perp$ as $g\sim \mu^{*n}$.
\begin{corollary} \label{relative-angle-2}
Let $W\in \Gr(\kg, d)$. Then for $n\geq 1$, $r\geq e^{-n}$,
$$(\mu^{*n})^{\otimes d+1}\left\{ (g_{i})_{i=1}^{d+1}\,:\, \| V_{2}(g_{1})^\perp\wedge\dots \wedge V_{2}(g_{d+1})^\perp \wedge W\| \leq r \right\} \leq C r^c $$
where $C, c>0$ are constants depending on $\mu$ only.
\end{corollary}

\begin{proof}[Proof of \Cref{relative-angle-2}] Recall that the Lie algebra $\kg=\sl_{d+1}(\R)$ is equipped with the scalar product given by $\langle A,B\rangle=\tr(A^T B)$ where $A^T$ denotes the transpose of the matrix $A$. Using $\tr(AB)=\tr(BA)$, it is direct to check for every $g\in G$, the adjoint of $\Ad(g)\in \End(\kg)$ for this Euclidean structure is given by $\Ad(g)^*= \Ad(g^T)$. Moreover, the eigenspaces $\kg_{-}$, $\kg_{0}$ and $\kg_{+}$ are mutually orthogonal. It follows that for $g\in P'$, we have 
$$V_{2}(g)^\perp=(\Ad(u(-\ttb_{g}))\kg_{\leq 0})^\perp = \Ad(u(\ttb_{g})^T)\kg_{+} =\left(\Ad(u(-\ttb_{g}))\kg_{-}\right)^T.$$
As the map $\kg \mapsto \kg, A\mapsto A^T$ is an isometry, the claim follows from \Cref{relative-angle}.
\end{proof}

We now focus on establishing \Cref{relative-angle}. We first reduce to a purely geometric version of that result.
\begin{proposition}[Geometric reduction] \label{relative-angle-geom} Let $W\in \Gr(\kg, d)$. Then there exists $(u_{i})_{i=1}^{d+1} \in U^{d+1}$ such that
$$\Ad(u_{1})\kg_{-}\oplus\Ad(u_{2})\kg_{-}\oplus\dots \oplus \Ad(u_{d+1})\kg_{-}\oplus W=\kg$$
\end{proposition}
\Cref{relative-angle-geom} is geometric in the sense that no random variable is involved. It turns out to be equivalent to \Cref{relative-angle}.

\begin{proof}[Proof that \Cref{relative-angle} $\iff$ \Cref{relative-angle-geom}] The direct implication being clear, we assume \Cref{relative-angle-geom} and check \Cref{relative-angle}. We write $P(\bs):=\Ad(u(- \bs))$ for conciseness. Observe that the angle function $(\R^d)^d \rightarrow \bigwedge^{\dim \kg} \kg\simeq \R, (\bs_{i})_{i}\mapsto \norm{P(\bs_{1})\kg_{-}\wedge \dots \wedge P(\bs_{d+1})\kg_{-}\wedge W}$ is Lipschitz continuous. In view of  \Cref{approx-sigma-n}, it suffices to show the existence of constants $C, c > 0$ depending only on $\sigma$ such that for every $r>0$,
$$
\sigma^{\otimes d+1}\left\{ (\bs_i)_{i=1}^{d+1}\,:\, \| P(\bs_1)\kg_- \wedge\dots \wedge P(\bs_{d+1})\kg_- \wedge W\| \leq r \right\} \leq C r^c.
$$

Let $v_-, w \in \bigwedge^{*}\kg$ be unit vectors spanning respectively the lines $\bigwedge^{\dim \kg_-}\kg_-$ and $\bigwedge^{\dim W}W$. Note that
\begin{align}\label{eq0-numden}
\| P(\bs_{1})\kg_{-}\wedge\dots \wedge P(\bs_{d+1})\kg_{-}\wedge W\|
= \frac{\| P(\bs_{1})v_-\wedge\dots \wedge P(\bs_{d+1})v_-\wedge w\|}{\| P(\bs_{1})v_-\|\dotsm \| P(\bs_{d+1})v_-\|}.
\end{align}
As the map $\bs\mapsto P(\bs)$ is polynomial, and $\sigma$ has finite moment of positive order (\Cref{sigma-moment}), we have for some $\gamma=\gamma(\sigma)>0$,
\begin{align}\label{eq1-den}
\sigma^{\otimes d+1}\left\{\| P(\bs_{1})v_-\|\dotsm \| P(\bs_{d+1})v_-\|\geq r^{-1/2}\right\}\ll_{\sigma} r^{\gamma}.
\end{align}
On the other hand, \Cref{relative-angle-geom} guarantees that the polynomial map $(\bs_{i})_{i}\mapsto P(\bs_{1})v_-\wedge\dots \wedge P(\bs_{d+1})v_-\wedge w$ is non-zero. As it depends continuously on $w$ and $\Gr(\kg,d)$ is compact, it must have the supremum norm on the coefficients bounded below by a constant $c_{d}>0$ depending only on $d$.
Combined with \Cref{non-acc-var-product}, this yields
\begin{align}\label{eq2-num}
\sigma^{\otimes d+1}\left\{\| P(\bs_{1})v_-\wedge\dotsm \wedge P(\bs_{d+1})v_-\wedge w\| <r^{1/2} \right\}\ll_{\sigma} r^{\gamma}
\end{align}
up to taking $\gamma$ smaller. \Cref{relative-angle} follows from the combination of \eqref{eq0-numden}, \eqref{eq1-den}, \eqref{eq2-num}.
\end{proof}

We further reduce to the case where the subspace $W$ is invariant under a Borel subgroup of $G$. We denote by $B$ the upper triangular subgroup of $G$.


\begin{proposition}[Borel-invariant reduction] \label{relative-angle-geom2} Let $W\in \Gr(\kg,d)$ be a subspace which is $\Ad(B)$-invariant. Then there exists $(g_{i})_{i=1}^{d+1} \in G^{d+1}$ such that
$$\Ad(g_{1})\kg_{-}\oplus\Ad(g_{2})\kg_{-}\oplus\dots \oplus \Ad(g_{d+1})\kg_{-}\oplus W=\kg.$$
\end{proposition}

\bigskip

Let us check that \Cref{relative-angle-geom} and \Cref{relative-angle-geom2} are equivalent.

\begin{proof}[Proof that \Cref{relative-angle-geom} $\iff$ \Cref{relative-angle-geom2}] The direct implication is clear.
We establish the converse. Assume by contradiction that \Cref{relative-angle-geom} fails for some $W\in \Gr(\kg, d)$. Write $Z_{G}(A')$ the centralizer of $A'$ in $G$. Noting that $\kg_{-}$ is $Z_{G}(A')U^-$-invariant, we obtain for every $(g_{i})_{i}\in (UZ_{G}(A')U^-)^{d+1}$ that
\begin{equation}
\label{eq:sumAdgkgW}
\Ad(g_{1})\kg_{-}+\Ad(g_{2})\kg_{-}+\dots + \Ad(g_{d+1})\kg_{-}+ W \neq \kg.
\end{equation}
This is a Zariski-closed condition in the variable $(g_{i})_{i=1}^{d+1}$.
By looking at Lie algebras, we see that $UZ_{G}(A')\supseteq B$, so by Bruhat's decomposition, $UZ_{G}(A')U^-$ is Zariski-dense in $G$.
It follows that \eqref{eq:sumAdgkgW} holds for all $(g_{i})_{i=1}^{d+1}\in G^{d+1}$.
Applying another $\Ad(g)$ on both side we see the set
\[
\setbig{ W \in \Gr(\kg,d)\,:\, \text{\eqref{eq:sumAdgkgW} holds for all } (g_{i})_{i=1}^{d+1}\in G^{d+1}}
\]
is preserved under the action of $\Ad(G)$, whence of $\Ad(B)$.
It is moreover Zariski-closed. More precisely, it is the set of $\R$-points of a complete $\R$-variety. On the other hand, $B$ is the set of $\R$-points of a $\R$-split connected solvable linear algebraic group which acts $\R$-morphically on this variety.
Thus, by a version of the Borel fixed point theorem (\cite[Proposition 15.2]{Borel91}), the above set contains a fixed point for $\Ad(B)$. This contradicts \Cref{relative-angle-geom2}, thus finishing the proof of the converse implication.
\end{proof}

The advantage of reducing to \Cref{relative-angle-geom2} is that it constrains $W$ to belong to a finite explicit family of subspaces.
\begin{lemma} \label{formW}
Let $W \subseteq \kg$ be a subspace of dimension at most $d$. Then $W$ is $\Ad(B)$-invariant if and only if we can write
$$W=\Span\set{E_{i,j} \,:\, (i,j)\in S} $$
for some subset $S\subseteq \set{(i,j)\,:\,1\leq i<j\leq d+1}$ that is stable under the operations $(i,j)\mapsto (i-1,j)$ and $(i,j)\mapsto (i,j+1)$ (provided that $i\geq 2$ and $j\leq d$ respectively).
\end{lemma}
In words, $W$ must be a sum of elementary subspaces that are strictly above the diagonal, and stable by moving upward or to the right in the matrix representation.

\begin{proof} Recall $\kg=\sl_{d+1}$. Write $\kb$ the Lie algebra of $B$, i.e., the subspace of upper triangular matrices in $\kg$. As $B$ is Zariski-connected, the subspace $W$ is $\Ad(B)$-invariant if and only if it is $\ad(\kb)$-invariant.
Observe the relation $[E_{i,j}, E_{k,l}]=\1_{j=k}E_{i,l}-\1_{i=l}E_{k,j}$ for all $i,j,k,l\in \{1, \dots, d+1\}$.
In particular, for $i<j$ and $k<l$, the matrix $[E_{i,j}, E_{k,l}]$ is either $0$ or up to a sign an elementary matrix located either to the right or above $E_{i,j}$. This justifies the ``if'' direction in \Cref{formW}.

We now assume $W$ to be $\Ad(B)$-invariant and establish the announced decomposition.
By invariance under diagonal matrices, $W$ must be of the form $W= E \oplus \Span\set{ E_{i,j} \,;\, (i,j)\in S}$ where $E$ is a subspace of diagonal matrices, and $S$ does not intersect the diagonal.
If $E\neq \{0\}$, then by $\ad(\kb)$-invariance, $W$ must contain a line $\R E_{i,i+1}$ for some $i\in \{1, \dots, d\}$. But the $\ad(\kb)$-invariant subspace spanned by such a line has dimension at least $d$, which is absurd because $\dim W\leq d$. Hence $E=\{0\}$. Noting that for every $i>j$, the $\ad(\kb)$-invariant subspace spanned by $\R E_{i,j}$ intersects the diagonal subspace, we further deduce $S\subseteq \set{(i,j)\,:\,1\leq i<j\leq d+1}$. The final claim on $S$ follows from $\ad(\kb)$-invariance and the bracket relation exhibited in the first paragraph.
\end{proof}

We are finally able to show \Cref{relative-angle-geom2}, thus completing the proofs of Propositions \ref{relative-angle}, \ref{relative-angle-geom}.

\begin{proof}[Proof of \Cref{relative-angle-geom2}]
We shall prove this proposition by induction on $d$.
\medskip

\noindent\underline{Base case $d=1$}. Here $\kg=\sl_{2}(\R)$ and $\kg_-=\R E_{21}$. By assumption,
\[ W=\R E_{12}. \]
Take $g_1=\Id, g_2=\Id+E_{12}\in \SL_{2}(\R)$. By direct computation, we can verify that
\[\Ad(g_1)\kg_{-}\oplus \Ad(g_2)\kg_{-}\oplus W=\kg.\]
\medskip

\noindent\underline{Induction step.} Let $d\geq 2$ be an integer. We assume the proposition has been proved for $\SL_{d}(\R)$ and establish it for $\SL_{d+1}(\R)$. Throughout the proof, we write $M_{d+1}$ the space of all $d+1$ by $d+1$ real matrices. We keep the notations $G, \kg, \kg_{-}$ related to $\SL_{d+1}(\R)$. We write $G' =\SL_{d}(\R)$, which we view as a subgroup of $G$ by embedding it in the lower-right corner (and imposing $1$ on the first diagonal entry). Accordingly, we define $\kg'$ (resp. $\kg_-'$) to be the intersection of $\kg$ (resp. $\kg_{-}$) with the lower-right $d$ by $d$ block of $M_{d+1}$. In particular,
\[\kg_-'= \Span \left\{  E_{d+1,2} ,\dotsc, E_{d+1,d} \right\}.\]
We denote by $\Proj_{\kg'}:M_{d+1}\to M_{d+1}$ the projection onto the lower right $d$ by $d$ block, and by $\Proj_{R_1}:M_{d+1}\to M_{d+1}$ the projection onto the the subspace of matrices with nonzero entries only on the first row.

\medskip

Let $W\subset \kg$ be a $d$-dimensional linear subspace that is $\Ad(B)$-invariant. To make use of the induction hypothesis, our strategy is to choose a suitable element $g_0\in G$ such that $\Ad(g_0) \kg_-\oplus W$ fills up the first row of $\kg$. This will enable us to work on $\kg'$ and apply the induction hypothesis.
\medskip

To begin with, choose
\[g_0=(\Id+E_{d+1,1})\cdot \omega_{1,d+1}.\]
where $\omega_{1,d+1}$ is an element in the standard Weyl group of $G$ satisfying that left multiplication by $\omega_{1,d+1}$ exchanges the first and $(d+1)$-th row.
Then by direct computation, we have
\begin{align*}
  \Ad(g_0)\kg_-=\left\{\begin{pmatrix}
    a_1& a_2 & \cdots & a_d & -a_1\\
    0 & 0 & \cdots & 0& 0\\
    \vdots &\vdots &\ddots &\vdots & \vdots\\
    0 & 0 & \cdots & 0 & 0\\
    a_1& a_2 & \cdots & a_d & -a_1
  \end{pmatrix}: a_1,\cdots,a_d\in \R \right\}.
\end{align*}

Now consider $k := \dim \Proj_{R_1}(W)$, which, by \Cref{formW}, satisfies $1 \leq k \leq d$.
We then decompose $\kg_-=V_0\oplus V_1$, where
\begin{align*}
  &V_0:=\Span\left\{\, E_{d+1,j} \,:\, 1 \leq j \leq d+1-k \,\right\},\\
  &V_1:=\Span\left\{\, E_{d+1,j} \,:\, d+1-k < j \leq d \,\right\},
\end{align*}
so that $\dim V_0 = d - (k - 1)$ and $\dim V_1 = k - 1$.
Using \Cref{formW}, we also decompose $W=\Proj_{R_1}(W)\oplus \mathrm{Proj_{\kg'}}(W)$.
Observe that $V_1$ and $\Ad(g_0)V_{1}$ coincide modulo $W$, leading to
\begin{align}\label{relative-angle-pf-eq1}
  \Ad(g_0)\kg_-\oplus W=\Ad(g_0)V_0\oplus \mathrm{Proj_{R_1}(W)}\oplus V_1\oplus \mathrm{Proj_{\kg'}(W)}.
\end{align}
Let
\[W'=V_1\oplus \mathrm{Proj_{\kg'}(W)}.\]
 Note that $W'\subset \kg'$ and
\[\dim W'=k-1+d-k=d-1.\]
Hence, we can apply the induction hypothesis (more precisely its equivalent version from \Cref{relative-angle-geom}) to the pair $(G' , W')$ to obtain $g'_1,\cdots,g'_d\in G'$ such that
\begin{align}\label{relative-angle-pf-eq2}
  \Ad(g'_1)\kg'_-\oplus \Ad(g'_2)\kg'
  _-\oplus \cdots \oplus \Ad(g'_d)\kg'_-\oplus W'=\kg'.
\end{align}
Let $C\subset \kg$ be the linear subspace defined by
\[C:=\Span\set{ E_{i,1} \,:\, 2 \leq i \leq d + 1}.\]
 Observe that $C$ is $\Ad(G')$-invariant, and the adjoint representation of $G'$ on $C$ is isomorphic to the standard representation $\R^d$ of $G'=\SL_d(\R)$. Therefore, we may find $g_1'',g_2'',\cdots,g_d''\in G'$ such that
\begin{align}\label{relative-angle-pf-eq3}
  \Ad(g_1'')\R E_{d+1,1}\oplus \Ad(g_2'')\R E_{d+1,1}\oplus \cdots \oplus \Ad(g_d'')\R E_{d+1,1}=C.
\end{align}

Observe that the collection of elements $(g'_{i})_{i=1}^d$ and $(g''_{i})_{i=1}^d$ satisfying respectively
\eqref{relative-angle-pf-eq2} and \eqref{relative-angle-pf-eq3} are (non-empty) Zariski-open subsets of $G'^d$. By irreducibility of $G'^d$ for the Zariski topology, they are dense, whence must intersect. This allows to choose $(g'_1,\cdots,g'_d)=(g''_1,\cdots,g''_d)$ in \eqref{relative-angle-pf-eq2} and \eqref{relative-angle-pf-eq3}. As
\[\kg_{-}=\kg'_{-}\oplus \R E_{d+1,1},\]
we obtain
\begin{align}\label{relative-angle-pf-eq4}
  \Ad(g'_1)\kg_{-}\oplus \Ad(g'_2)\kg_{-}\oplus \cdots\oplus \Ad(g'_d)\kg_{-}\oplus W'=\kg' \oplus C.
\end{align}

On the other hand, observing that the restriction of $\Proj_{R_1}$ to $\Ad(g_0)V_0\oplus \mathrm{Proj_{R_1}}(W)$ is injective while its restriction to $\kg' \oplus C$ vanishes, we have $(\Ad(g_0)V_0\oplus \mathrm{Proj_{R_1}}(W)) \cap (\kg' \oplus C)=\{0\}$, or equivalently
\begin{align}\label{relative-angle-pf-eq5}
\Ad(g_0)V_0\oplus \mathrm{Proj_{R_1}}(W) \oplus \kg' \oplus C=\kg
\end{align}
 because dimensions match. Combining \eqref{relative-angle-pf-eq1}, \eqref{relative-angle-pf-eq4}, \eqref{relative-angle-pf-eq5}, we obtain
$$  \Ad(g_0)\kg_{-} \oplus \Ad(g'_1)\kg_{-}\oplus \Ad(g'_2)\kg_{-}\oplus \cdots\oplus \Ad(g'_d)\kg_{-}\oplus W=\kg.$$
This validates the induction step and completes the proof.
\end{proof}

\subsection{Linear multislicing} \label{Sec-mult}

It remains to see how the non-concentration property established for the subspace $\Ad(u(-\ttb_{g})\kg^{-})$ in the previous subsection can be exploited to obtain a dimensional gain. In this subsection, we study this question in an abstract linear setting. We place ourselves in $\R^D$ where $D\geq 2$. We encapsulate the non-concentration property via the following definition.

\begin{definition}\label{MNC-def}
Let $k\in \llbracket 1, D-1\rrbracket$, let $C,c, \rho>0$.
Let $\Xi$ be a probability measure on $\Gr(\R^D,k)$. We say $\Xi$ satisfies the \emph{mild non-concentration property} $\MNC$ with parameters $(\rho, C,c)$ if there exist integers $q,m \in \N $ such that $D=qk +m$ and for every $W\in \Gr(\R^D, m)$, $r\geq \rho$,
\begin{equation*}
\Xi^{\otimes q} \left\{(F_{i})\,:\,  \| F_{1}\wedge \dots \wedge F_{q}\wedge W\|\leq r \right\}\leq Cr^c.
\end{equation*}
We also say $\Xi$ satisfies $\MNC^\perp$ with parameters $(\rho, C,c)$ if its image under $F\mapsto F^\perp$ satisfies $\MNC$ with parameters $(\rho, C,c)$.
\end{definition}

Our aim is to show that $\MNC$ or $\MNC^\perp$ allow for a \emph{supercritical multislicing estimate}, see \Cref{sup-mult}. For that, we first present a submodular inequality for covering numbers (\Cref{submod-cn}) and use it to connect $\MNC$ and $\MNC^\perp$ with the properties of the individual projectors $\pi_{F_{i}}$ (\Cref{cn-quasi-orth}). We then deduce that a random subspace $F$ whose law $\Xi$ has the property $\MNC$ or $\MNC^\perp$ must enjoy both supercritical and subcritical projection theorems (Lemmas \ref{MNC->sup}, \ref{MNC->sub}). Those estimates refine that of Bourgain \cite{Bourgain2010} and He \cite{He2020JFG} which were established under the stronger non-concentration condition that $\Xi$ satisfies $\MNC$ with $q=1$. From there, we use \cite{BH24} to combine our projection theorems into the multislicing estimate \Cref{sup-mult}
\bigskip

We now introduce the submodular inequality for covering numbers that we need. Let $\cP$ and $\cQ$ denote partitions of $\R^D$, let $A$ be a subset of $\R^D$.
We write $\cP(A)$ the set of cells of $\cP$ that meet $A$, that is,
\[
\cP(A) := \set{ P \in \cP : P \cap A \neq \emptyset},
\]
and set $\cN_\cP(A)$ the cardinality of $\cP(A)$.
We say $\cQ$ \emph{roughly refines} $\cP$ with parameter $L \geq 1$, and write $\cP \overprec{L} \cQ$, if
\[
\max_{Q \in \cQ}\, \cN_\cP(Q) \leq L.
\]
We also use the notation $\cP \overset{L}{\simeq} \cQ$ to say that both $\cP \overprec{L} \cQ$ and $\cQ \overprec{L} \cP$ hold.
Finally, we denote by $\cP \vee \cQ$ the partition obtained by taking the intersections of $\cP$-cells and $\cQ$-cells.

\begin{lemma}[Submodular inequality] \label{submod-cn}
Let $\cP,\cQ,\cR,\cS$ be partitions of $\R^D$, and $A$ a subset of $\R^D$.
Let $L \geq 1$.
Assume that $\cR \overset{L}{\simeq} \cP \vee \cQ$, and $\cS \overprec{L} \cP$, $\cS \overprec{L} \cQ$.
Then for every $c >0$, there is a subset $A' \subset A$ such that $\cN_{\cR}(A') \geq \frac{1-c}{L^3} \cN_{\cR}(A)$ and
\begin{equation}\label{submod-coveringnumber}
\cN_\cP(A) \cN_\cQ(A) \geq \frac{c^2}{4L^3} \cN_\cR(A) \cN_{\cS}(A').
\end{equation}
\end{lemma}

In the case where $L=1$, the result is exactly \cite[Lemma 2.6]{BH24}.
We deduce the refinement presented in \Cref{submod-cn}.
Such an upgrade is convenient to deal with situations where partitions $\cP, \cQ, \cR, \cS$ do not exactly fit together.

\begin{proof}
We start with a few general observations on the relation $\overprec{L}$. Note it is transitive in the sense that $\cP \overprec{L} \cP'$ and $\cP' \overprec{L'} \cP''$ implies $\cP \overprec{L L'} \cP''$.
It is also compatible with taking common refinements, that is, $\cP \overprec{L} \cQ$ and $\cP' \overprec{L'} \cQ'$ implies $\cP \vee \cP' \overprec{LL'} \cQ \vee \cQ'$. Finally, observe that $\cP \overprec{L} \cQ$ implies $\cN_\cP(A) \leq L \cN_\cQ(A)$ for any subset $A$.

Now consider $\cP_0 = \cP \vee \cS$, $\cQ_0 = \cQ \vee \cS$ and $\cR_0 = \cP \vee \cQ \vee \cS$.
By \cite[Lemma 2.6]{BH24} applied to $\cP_0$, $\cQ_0$, $\cR_0$ and $\cS$, there is a subset $A' \subset A$ such that
$\cN_{\cR_0}(A') \geq (1 - c) \cN_{\cR_0}(A)$ and
\begin{align}\label{submod-coveringnumberL=1}
\cN_{\cP_0}(A) \cN_{\cQ_0}(A) \geq \frac{c^2}{4} \cN_{\cR_0}(A) \cN_{\cS}(A').
\end{align}
Using the properties of $\overprec{L}$ recalled above, we derive from the assumptions that $\cP_0 \overprec{L} \cP$ and $\cQ_0 \overprec{L} \cQ$, as well as $\cR \overprec{L} \cR_0$ and $\cR_0 \overprec{L^2} \cR$.
The inequality \eqref{submod-coveringnumber} then follows from \eqref{submod-coveringnumberL=1}.
\end{proof}

In the next lemma, we consider a family of projectors of $\R^D$ whose images (resp. kernels) are in direct sum with controlled angle. Given a set $A\subseteq \R^D$, we relate the product of covering numbers of the projections of $A$ with the covering number of the projection of $A$ onto (resp. parallel to) the sum of the images (resp. kernels). Given $F\in \Gr(\R^d)$, we let $\pi_{F}, \pi_{||F}:\R^d\rightarrow \R^d$ denote respectively the orthogonal projectors of image or kernel $F$.
For $\rho>0$, we denote by $\cN_{\rho}(A)$ the least number of open balls of radius $\rho$ needed to cover $A$.

\begin{lemma} \label{cn-quasi-orth}
Let $(F_{i})_{i=1, \dots, q}$ be a (non-necessarily generating) collection of subspaces in $\R^D$. Assume for some $r\in (0, 1/2)$ that
\begin{equation}\label{quasi-orth}
\norm{F_{1} \wedge \dots \wedge F_{q}} \geq r.
\end{equation}
Then for any set $A\subseteq \R^D$, one has
\begin{equation}\label{prod-submod-to}
\prod_{i=1}^q \cN_{\rho}(\pi_{F_{i}}A)\geq r^{O_{D}(1)}\cN_{\rho}(\pi_{\oplus_{i} F_{i}}A),
\end{equation}
and
\begin{equation}\label{prod-submod-parto}
\prod_{i=1}^q \cN_{\rho}(\pi_{|| F_{i}}A)\geq r^{O_{D}(1)} \cN_{\rho}(A)^{q-1}\cN_{\rho}(\pi_{||\oplus_{i} F_{i}}A')
\end{equation}
for some subset $A' \subseteq A$ satisfying $\cN_{\rho}(A') \geq r^{O_D(1)} \cN_{\rho}(A)$.
\end{lemma}

\begin{proof} The first inequality~\eqref{prod-submod-to} is a simple counting, see e.g. \cite[Lemma 15]{He2020JFG}. We focus on proving \eqref{prod-submod-parto}.

By induction on $q$ together with the observation that
$$
\norm{(F_1\oplus \dots \oplus F_{q - 1}) \wedge F_{q}} \geq \norm{F_1 \wedge \dotsc \wedge F_{q}} \geq r,
$$
the proof of \eqref{prod-submod-parto} reduces to the case $q=2$.

Let $\cD_\rho$ denote the partition corresponding to the tiling of $\R^D$ by the cube $[0,\rho)^D$ and its $\rho \Z^D$-translates.
Consider $\cP = \pi_{||F_{1}}^{-1}(\cD_{\rho})$, $\cQ=\pi_{||F_{2}}^{-1}(\cD_{\rho})$, $\cR = \cD_\rho$ and $\cS = \pi_{||F_1 \oplus F_2}^{-1}(\cD_{\rho})$, so that
$\cN_\cP(A) \simeq_{D} \cN_\rho( \pi_{||F_1} A)$, $\cN_\cQ(A) \simeq_{D} \cN_\rho( \pi_{||F_2} A)$, $\cN_\cR(A) \simeq_{D} \cN_\rho(A)$ and $\cN_\cS(A) \simeq_{D} \cN_\rho(\pi_{||F_1 \oplus F_2} A)$.

From $\norm{F_1 \wedge F_2} \geq r$ we know that $\cR \overprec{r^{-O_D(1)}} \cP \vee \cQ$ and it is always true that $\cP \vee \cQ \overprec{O_D(1)} \cR$ and $\cS \overprec{O_D(1)} \cP$ and $\cS \overprec{O_D(1)} \cQ$.
Thus, the inequality \eqref{prod-submod-parto} follows from \Cref{submod-cn}.
\end{proof}

We show that $\MNC$ or $\MNC^\perp$ is a sufficient condition for the supercritical projection theorem.

\begin{lemma}[Supercritical projection] \label{MNC->sup} Let $k \in \llbracket 1, D-1\rrbracket$, let $c, \eps, \rho>0$.
Let $\Xi$ be a probability measure on $\Gr(\R^D,k)$ satisfying either $\MNC$ or $\MNC^\perp$ with parameters $(\rho, \rho^{-\eps},c)$.

 Let $A\subseteq B^{\R^D}_{1}$ be any subset satisfying for some
 $\alpha\in[c, 1-c]$,
 \begin{equation}\label{cn-A-lowbnd}
\cN_\rho(A) \geq \rho^{-D \alpha +\eps},
\end{equation}
and for $r \geq \rho$,
 \begin{equation}\label{nc-dim-dalpha}
\sup_{v \in \R^D} \cN_\rho(A \cap B^{\R^D}_r(v)) \leq \rho^{-\eps} r^{c} \cN_\rho(A).
\end{equation}

If $\eps,\rho\lll_{D,c}1$, then the exceptional set
 \begin{equation*}
\begin{split}
\cE:=\bigl\{\, F\in \Gr(\R^D, k)\,:\, \exists A' \subseteq A \,\,&\text{ with }\,\,\cN_{\rho}(A')\geq \rho^\eps \cN_{\rho}(A) \\
& \text{ and }\,\, \cN_{\rho}(\pi_{F}A') < \rho^{- \alpha k -\eps} \,\bigr\}
\end{split}
 \end{equation*}
satisfies $\Xi(\cE)\leq \rho^\eps$.
\end{lemma}

\begin{proof}
We focus on the scenario where $\Xi$ satisfies $\MNC^\perp$.
The case where $\Xi$ satisfies $\MNC$ can be handled similarly, and is only easier to justify as it involves \eqref{prod-submod-to} instead of \eqref{prod-submod-parto}.

Let $(q,m)$ be the pair of integers playing a role in the assumption $\MNC^\perp$ for $\Xi$. If $q=1$, the result is known. It is indeed the higher rank version of Bourgain's projection theorem \cite{Bourgain2010}, due to the second-named author~\cite{He2020JFG}.
We deduce from there the general case $q\geq 1$.

Note that throughout the proof, we may assume $A$ to be $2\rho$-separated, so that $\cN_\rho(S) = \abs{S}$ holds for any of its subsets $S \subset A$.
We may also allow the upper bound on $\rho$ to depend\footnote{Indeed, if we establish the lemma for a pair $(\eps, \rho)$ then it is automatically valid for $(\eps', \rho)$ with $\eps'\in (0, \eps)$, because when passing from $\eps$ to $\eps'$, assumptions get stronger and the conclusion gets weaker.} on $\eps$ (not only $D,c$).
We let $\eps_{1}, \eps_{2}>0$ be parameters to specify below in terms of $D$ and $c$.
We use the shorthand $\G:=\Gr(\R^D,k)$.

Provided $\eps+ \eps_{1}\leq c$, the assumption that $\Xi$ enjoys $\MNC^\perp$ with parameters $(\rho, \rho^{-\eps},c)$ implies
$$\cE_{1}:=\bigl\{\, \uF\in \G^q: \| F_{1}^\perp\wedge \dots \wedge F_{q}^\perp\|\leq \rho^{(\eps+\eps_{1})/c} \,\bigr\}\text{ satisfies } \Xi^{\otimes q}(\cE_{1})\leq \rho^{\eps_{1}}.$$
Let $\uF \in \G^q \smallsetminus \cE_{1}$.
Up to assuming $\eps\leq \eps_{1}$, Equation \eqref{prod-submod-parto} applied to the family $(F_{i}^\perp)_{i=1}^q$ implies that for every set $S\subseteq A$, there exists a subset $S' \subset S$ such that $\abs{S'} \geq \rho^{O_{D,c}(\eps_1)} \abs{S}$ and
\begin{equation*}
\prod_{i=1}^q\cN_{\rho}\bigl(\pi_{F_{i}}S\bigr)\geq \rho^{O_{D,c}(\eps_{1})} \abs{S}^{q - 1} \cN_{\rho}(\pi_{\bigcap_{i}F_{i}}S'),
\end{equation*}
which in particular yields
\begin{equation}\label{eq-ineq-vis}
\max_{i=1\dots q}\cN_{\rho}\bigl(\pi_{F_{i}}S\bigr)\geq \rho^{O_{D,c}(\eps_{1})} \abs{S}^{ 1- 1/q} \cN_{\rho}(\pi_{\bigcap_{i}F_{i}}S')^{1/q}.
\end{equation}
Taking $S$ to be a not too small subset of $A$, we will use \eqref{eq-ineq-vis} to obtain an explicit lower bound on $\max_{i=1, \dots, q}\cN_{\rho}(\pi_{F_{i}}S)$, see \eqref{eq-maxpiFthetai}.
As a lower bound on $\abs{S}^{q-1}$ comes directly from the assumption \eqref{cn-A-lowbnd}, we focus on $\cN_{\rho}(\pi_{\bigcap_{i}F_{i}}S')$.


Let $\Upsilon$ denote the restriction of $\Xi^{\otimes q}$ to $\G^q \smallsetminus \cE_{1}$, renormalised into a probability measure.
Note that, by definition and the observation that
$$\normbig{(F_1^\perp \oplus \dotsb \oplus F_q^\perp) \wedge W} \geq \normbig{F_1^\perp \wedge \dotsm \wedge F_q^\perp \wedge W},$$
the random $D-q(D-k)$-plane $(\bigcap_{i}F_{i})_{\uF\sim \Upsilon}$ satisfies the non-concentration condition $\MNC^\perp$ with parameters $(\rho, \rho^{-\eps-\eps_{1}},c)$ and $q=1 $, allowing us to apply the $q = 1$ case.
Therefore, provided that $\eps +\eps_{1}, \eps_2 \lll_{D,c} 1$ and $\rho\lll_{D,c}1$, there is an event $\cE_{2}\subseteq \G^q$ such that $\Upsilon(\cE_{2}) \leq \rho^{\eps_2}$,
and for all $\uF \in \G^q \smallsetminus \cE_{2}$, for all $A'\subseteq A$ with $\abs{A'} \geq \rho^{\eps_2} \abs{A}$, we have
\begin{equation}\label{eq-prohFuth}
\cN_{\rho}\bigl(\pi_{\bigcap_{i}F_{i}}A'\bigr) \geq \rho^{- \alpha (D - q(D-k)) -\eps_{2}}.
\end{equation}

Combining \eqref{eq-ineq-vis}, \eqref{cn-A-lowbnd} and \eqref{eq-prohFuth}, we obtain that for all $\uF\in \G^q \smallsetminus (\cE_{1}\cup \cE_{2})$, and $A''\subseteq A$ with $\abs{A''} \geq \rho^{\eps_1} \abs{A}$, we have
\begin{equation}\label{eq-maxpiFthetai}
\max_{i=1, \dots, q} \cN_{\rho}\bigl(\pi_{F_{i}}A''\bigr)\geq \rho^{- \alpha k -\frac{1}{2q}\eps_{2}},
\end{equation}
provided that $\eps \leq \eps_1 \lll_{D, c} \eps_2\lll_{D, c} 1$.

To conclude the proof, we argue by contradiction, assuming  $\Xi(\cE) > \rho^\eps$.
For each $F\in \cE$, let $A'_{F}\subseteq A$ be such that $\abs{A'_{F}}\geq \rho^\eps \abs{A}$ and $\cN_{\rho}(\pi_{F}A'_{F}) < \rho^{- \alpha k -\eps}$.
By \cite[Lemma 19]{He2020JFG} -- a Fubini argument -- applied to $q$ independent copies of $A'_{F}$, where $F$ is distributed according to $\Xi_{| \cE}$,
these independent copies are likely to intersect in a rather large subset:
\[
\cI :=\setbig{\uF \in \cE^q \,:\, \abs{A'_{F_{1}}\cap \dots \cap A'_{F_{q}}} \leq  \rho^{2q\eps} \abs{A}}
\quad \text{satisfies}\quad
\Xi^{\otimes q}(\cI)\geq \rho^{4q\eps} \Xi(\cE)^q \geq \rho^{5q\eps}.
\]
It follows that
$\Xi^{\otimes q}\bigl(\cI \smallsetminus (\cE_1 \cup \cE_2) \bigr) \geq \rho^{5q\eps} - \rho^{\eps_{1}} - \rho^{\eps_2}$
which is strictly positive provided that $\eps\lll_{D}\eps_{1} \leq \eps_2$ and $\rho\lll_{\eps}1$.
In particular, we may consider $\uF \in \cI \smallsetminus (\cE_1 \cup \cE_2)$.
Setting $A''=A'_{F_{1}}\cap \dots \cap A'_{F_{q}}$, we have $\abs{A''} \geq \rho^{2q \eps} \abs{A}$ while the inclusions $A''\subseteq A'_{F_{i}}$ yield
$$\max_{i=1, \dotsc, q}\cN_{\rho}\bigl(\pi_{F_{i}}A''\bigr) < \rho^{- \alpha k -\eps}.$$
This is in contradiction with \eqref{eq-maxpiFthetai} for $\eps\lll_{D} \eps_1 \leq \eps_2$.
\end{proof}

Without non-concentration assumption on $A$, we still derive from $\MNC$ or $\MNC^\perp$ a subcritical projection theorem.

\begin{lemma}[Subcritical projection] \label{MNC->sub} Let $k \in \llbracket 1, D-1\rrbracket$, let $C>1$ and $c, \eps, \rho \in (0, 1/2]$.
Let $\Xi$ be a probability measure on $\Gr(\R^D,k)$ satisfying either $\MNC$ or $\MNC^\perp$ with parameters $(\rho, \rho^{-\eps},c)$.
 Let $A\subseteq B^{\R^D}_{1}$ be any subset.

If $C\ggg_{D,c}1$ and $\rho\lll_{D,c, \eps}1$, then
 \begin{equation*}
\begin{split}
\cE:=\bigl\{\, F\in \Gr(\R^D,k)\,:\, \exists A' \subseteq A \,\,&\text{ with }\,\,\cN_{\rho}(A')\geq \rho^\eps \cN_{\rho}(A) \\
& \text{ and }\,\, \cN_{\rho}(\pi_{F}A') < \rho^{C\eps}\cN_{\rho}(A)^{ \frac{k}{D}} \,\bigr\}
\end{split}
 \end{equation*}
satisfies $\Xi(\cE)\leq \rho^\eps$.
\end{lemma}

\begin{remark}There is no condition on $\eps$ in the statement. However, it is only meaningful for $C\eps \leq D$, which requires in particular $\eps\lll_{D,c}1$. Indeed, if $C\eps > D$ and $\rho\lll_{D}1$, then  $\cE=\emptyset$.
\end{remark}

\begin{proof} The case $q=1$ is the subcritical projection theorem as stated in \cite[Proposition A.2]{BH24}. The general case can be deduced similarly to the proof of \Cref{MNC->sup}, but using this time the subcritical projection theorem instead of the supercritical projection theorem.
\end{proof}


\bigskip

We now combine Lemmas \ref{MNC->sup}, \ref{MNC->sub} into a multislicing estimate. We place ourselves in $\R^D$ where $D\geq 3$. We consider $d_{1}, d_{2}\in \N$ such that $1\leq d_{1}<d_{2}<D$ and $\bt=(t_{1}, t_{2}, t_{3})\in \R^3$ such that $0\leq t_{1}<t_{2}<t_{3}\leq 1$.
We denote by $\cF$ the collection of pairs $\sV=(V_{1}, V_{2})$ where $V_{i}\in \Gr(\R^D, d_{i})$ for $i=1,2$ and $V_{1}\subseteq V_{2}$.
Given $\sV\in \cF$ and $\rho \in (0,1)$, we set
$$B_{\rho^\bt}^\sV= B^{V_{1}}_{\rho^{t_{1}}} +B^{V_{2}}_{\rho^{t_{2}}} +B^{\R^d}_{\rho^{t_{3}}}.$$
Therefore $B_{\rho^\bt}^\sV$ represents a Euclidean box carried by the partial flag $\sV$ and of side length parameters $\rho^{t_{1}}> \rho^{t_{2}}>\rho^{t_{3}}$.
The multislicing theorem below considers a random partial flag $\sV$, and a measure $\nu$ which is Frostman above scale $\rho$. For most realizations of $\sV$, it gives an upper bound on the mass granted by $\nu$ to all translates of $B_{\rho^\bt}^\sV$. It requires a certain assumption on $\sV$, namely that each component of $\sV$ satisfies $\MNC$ or $\MNC^\perp$ above scale $\rho$.

\begin{proposition}[Supercritical multislicing] \label{sup-mult}
Let $D\geq 3$ and $d_{1},d_{2}, \bt$ be as above. Let $c, \eps,\rho >0$.

Let $\Xi$ be a probability measure on $\cF$. Assume that for each $i=1,2$, the distribution of the component $V_{i}$ as $\sV\sim \Xi$ satisfies either $\MNC$ or $\MNC^\perp$ with parameters $(\rho, \rho^{-\eps}, c)$.

Let $\nu$ be a Borel measure on $B^{\R^D}_{1}$ of mass at most $\rho^{-\eps}$, and such that for some $\alpha\in [c, 1-c]$, for all $v\in \R^D$, all $r\in [\rho^{t_3}, \rho^{t_1}]$, we have
$$\nu(B^{\R^D}_{r}+v)\leq\rho^{-\eps} r^{D\alpha}.$$

If $\eps, \rho\lll_{D, \bt, c} 1$, then there exists an event $\cE \subseteq \cF$ such that $\Xi(\cE)\leq \rho^{\eps}$ and for $\sV\in\cF \setminus \cE$, there is a set $A_{\sV} \subseteq \R^D$ with $\nu(\R^D \setminus A_{\sV}) \leq \rho^{\eps}$ and such that for every $v \in \R^D$,
\[ \nu_{ |A_{\sV}} \left(B^{\sV}_{\rho^\bt}+v\right) \leq \leb \left(B^{\sV}_{\rho^\bt}\right)^{\alpha+\eps}.\]
\end{proposition}

The proof is done by combining Lemmas~\ref{MNC->sup}, \ref{MNC->sub} with the line of reasoning of  \cite[Section 2]{BH24}.  Alternatively, those two lemmas can also be plugged into  the more advanced multislicing machinery \cite[Theorem 3.4]{BH25} to obtain \Cref{sup-mult} as a formal consequence. For the reader's convenience,  we  sketch below how to deduce \Cref{sup-mult} from Lemmas~\ref{MNC->sup}, \ref{MNC->sub}.


\begin{proof}[Sketch of proof]
In this proof, we allow the implied constant in the notations $\ll$ and $O(\cdot)$ to depend on $D$ and $\bt$. We may\footnote{Indeed, if we establish the proposition for a pair $(\eps, \rho)$ then it is automatically valid for $(\eps', \rho)$ with $\eps'\in (0, \eps)$, because when passing from $\eps$ to $\eps'$, assumptions get stronger and the conclusion gets weaker.} also allow $\rho$ to be small enough depending on $\eps$ (not only $D, \bt, c$). 

Note that if $\nu(\R^D) \leq \rho^\eps$, then we can simply take $A_\sV = \emptyset$ and the result is trivial.
Thus, we may assume $\nu(\R^D)\in [\rho^\eps, \rho^{-\eps}]$, and
after a renormalisation, reduce to the case where $\nu(\R^D)=1$.

We further reduce to the situation where $\nu$ is the uniform probability measure on a finite set with a regular-tree structure. For $i=1, 2,3$, set $\cQ_i=\cD_{\rho^{t_{i}}}$ the tiling of $\R^D$ by cubes of side length $\rho^{t_i}$, in particular $\cQ_1 \prec \cQ_2 \prec \cQ_3$. By invoking \cite[Lemma A.2]{BH25} (which only relies on dyadic pigeonholing) we obtain pairwise disjoint $\cQ_3$-measurable sets $A_k \subset \R^D$ indexed by a set $\cK$ of cardinality $\abs{\cK} \ll \abs{\log \rho}^{O(1)}$ such that
$\nu\bigl(\R^D \setminus \bigcup_{k \in K} A_k\bigr) \leq \rho^{2 \eps}$ and satisfying for each $k \in \cK$:
\begin{itemize}
\item $A_k$ is regular with respect to the filtration $\cQ_1 \prec \cQ_2 \prec \cQ_3$, in the sense of \cite[\S A.2]{BH25};
\item $\nu(A_k) \geq \rho^{8\eps}$;
\item for every $\cQ_3$-measurable subset $S \subset A_k$, we have
\[
\frac{\nu(S)}{\nu(A_k)} \simeq \frac{\cN_{\rho^{t_3}}(S)}{\cN_{\rho^{t_3}}(A_k)}.
\]
\end{itemize}
The bounds $\nu\bigl(\R^D \setminus \bigcup_{k \in K} A_k\bigr) \leq \rho^{2 \eps}$ and $\abs{\cK} \ll \abs{\log \rho}^{O(1)}$ justify  it is sufficient to prove the conclusion of  \Cref{sup-mult} for each measure $\nu_{|A_{k}}/\nu(A_{k})$. By discretizing at scale $\rho^{t_{3}}$, this in turn reduces to the case of the measure $\frac{1}{\cN_{\rho^{t_3}}(A_k)}\sum_{Q\in \cQ_{3}}\delta_{x_{Q}}$, where $x_{Q}$ denotes the center of the $\delta^{t_{3}}$-cell $Q$. 
Note these measures satisfy the non-concentration conditions required in \Cref{sup-mult} (with $9\eps$ instead of $\eps$). We are thus reduced to the case where $\nu$ is the uniform probability measure on a finite set $A=\supp \nu$, which is regular for $\cQ_1 \prec \cQ_2 \prec \cQ_3$, and intersects each cell of $\cQ_{3}$ in at most one point.

Let $\cP^\sV$ denote a tiling of $\R^D$ by translates of a prism of shape comparable to $B^\sV_{\rho^\bt}$.
Let $\cP^\sV_{\mathrm{bad}}$ denote the subset of tiles $P \in \cP^\sV$ such that $\nu(P) > 2^{-D} \leb \bigl(B^{\sV}_{\rho^\bt}\bigr)^{\alpha+\eps}$.
It suffices to show that $\Xi(\cE) \leq \rho^\eps$ for
\[
\cE := \setbig{ \sV \in \cF \,:\,  \nu(\cup\cP^\sV_{\mathrm{bad}}) > \rho^\eps}.
\]


Let $\sV \in \cE$ be arbitrary.
Let $E$ denote the smallest $\cQ_3$-measurable set containing $\cup\cP^\sV_{\mathrm{bad}}$.
On the one hand, using  the structure of $\nu$ and $\nu(E)>\rho^\eps$, 
we find
\begin{equation}
\label{eq:EggAk}
\cN_{\rho^{t_3}}(E) \gg \rho^{\eps} \cN_{\rho^{t_3}}(A).
\end{equation}
On the other hand, $E$ is contained in the $\rho^{t_3}$-neighborhood of $\cup\cP^\sV_{\mathrm{bad}}$ and $\rho^{t_3}$ is the length of the shortest side of tiles in $\cP^\sV$.
Hence,
\begin{equation}
\label{eq:Ellpow}
\cN_{\cP^\sV}(E) \ll \cN_{\cP^\sV}(\cup\cP^\sV_{\mathrm{bad}}) \leq \leb(P)^{- \alpha - 2\eps}
\end{equation}
where the last bound uses the definition of $\cP^\sV_{\mathrm{bad}}$ and that $\nu(\R^D)\leq 1$. 

Next, consider two tiling $\cS^\sV$ and $\cR^\sV$ of $\R^D$ consisting respectively of prisms of shape approximately  $B^\sV_{\rho^{(t_1,t_2,t_2)}}$ and $B^\sV_{\rho^{(t_2,t_2,t_3)}}$, in particular 
$\cS^\sV \overprec{O(1)} \cP^\sV, \cQ_2$, and $\cR^\sV \overset{O(1)}{\simeq} \cP^\sV \vee \cQ_2$. 
To anticipate a technical detail, we replace $E$ by a large subset which is moreover regular with respect to $\cQ_2 \overprec{O(1)} \cR^\sV \overprec{O(1)} \cQ_3$.
This can be done (using e.g. \cite[Lemma 2.5]{BH24}) while keeping trivially \eqref{eq:Ellpow} and slightly weakening \eqref{eq:EggAk} to
\begin{equation}
\label{eq:EggAk2}
\cN_{\rho^{t_3}}(E) \geq \rho^{2\eps} \cN_{\rho^{t_3}}(A).
\end{equation}

By our choices for  $\cS^\sV, \cR^\sV$,  \Cref{submod-cn} yields the existence of a subset $E' \subset E$ such that
\begin{equation}\label{eq:E'ggE}
    \cN_{\cR^\sV}(E') \gg \cN_{\cR^\sV}(E)
\end{equation}
and 
\begin{equation}\label{NRNSll0}
    \cN_{\cR^\sV}(E) \cN_{\cS^\sV}(E') \ll \cN_{\cP^\sV}(E) \cN_{\rho^{t_2}}(E),
\end{equation}
Using $E\subseteq A$ and \eqref{eq:Ellpow} , this yields the (crucial) inequality
\begin{equation}
\label{eq:NRNSll}
\cN_{\cR^\sV}(E) \cN_{\cS^\sV}(E') \ll \leb(P)^{- \alpha - 2\eps} \cN_{\rho^{t_2}}(A).
\end{equation}

We will now apply the projection estimates established in \Cref{MNC->sup}, \Cref{MNC->sub}  to bound from below $\cN_{\cR^\sV}(E)$ and $ \cN_{\cS^\sV}(E')$ for most $\sV$ selected by $\Xi$, and ultimately get a contradiction with \eqref{eq:NRNSll}.

Let us bound from below $\cN_{\cS^\sV}(E')$. It essentially represents the covering number of $E'$ by translates  of the box $B^\sV_{\rho^{(t_1,t_2,t_2)}}$. This box can be seen, locally at scale  $\rho^{t_{1}}$, as the preimages of a ball of radius $\rho^{t_{2}}$ by the orthogonal projector of kernel $V_{1}$. This motivates the application of a projection theorem.
To do so, we zoom into each $Q \in \cQ_1(A)$ by introducing $A^Q$ the $\times \rho^{- t_1}$-dilation of $A \cap Q$, and we apply one of the projection estimates to the set $A^Q$ at scale $\rho^{-t_{1}+t_{2}}$ and the random projector $(\pi_{||V_1})_{\sV\sim \Xi}$. In principle, this yields a lower bound for $\cN_{\rho^{- t_1+t_{2}}}( \rho^{- t_1}(E'\cap Q)) \simeq \cN_{\cS^\sV}(E'\cap Q)$  because $E'\cap Q$ is a large subset of $ A\cap Q$ for most $Q$ (due to $\cN_{\rho^{t_2}}(E') \gg \rho^{3\eps} \cN_{\rho^{t_2}}(A)$ which stems from \eqref{eq:EggAk2}, \eqref{eq:E'ggE} and the regularity properties of $A$, $E$). Summing over $Q$, we  obtain in the end a lower bound on $\cN_{\cS^\sV}(E')$. 
Indeed, if $\cN_{\rho^{t_i}}(A) > \rho^{- t_i D \alpha - 10 D\eps}$ holds for some $i\in\{1,2\}$, we apply the subcritical version \Cref{MNC->sub} to each $A^Q$, and ultimately get
\begin{align}
\cN_{\cS^\sV}(E')
&\geq \rho^{4\eps} \cN_{\rho^{t_1}}(A)^{d_{1}/D} \cN_{\rho^{t_2}}(A)^{D-d_{1}/D}  \nonumber\\
& \geq \rho^{-\alpha(t_{1}d_{1}+t_{2}(d_{2}-d_{1})) -5\eps}    \cN_{\rho^{t_2}}(A)^{D-d_{2}/D}            \label{lwbdE'}
\end{align}
where the second inequality also uses the Frostman-type condition on $\nu$.
If on the other hand $\cN_{\rho^{t_i}}(A) \leq \rho^{- t_i D \alpha - 10D \eps}$ for both $i=1,2$, we are in a position to apply the supercritical projection theorem \Cref{MNC->sup} to still obtain \eqref{lwbdE'}.
Indeed, in this case, the Frostman-type condition on $\nu$ together with the regularity properties of $A$ imply that for every ball $B \subset Q$ of radius $r \in [\rho^{t_2}, \rho^{t_1}]$,
\[
\cN_{\rho^{t_2}}(B \cap A) \leq \rho^{-20D\eps} \left(\frac{r}{\rho^{t_1}}\right)^{D\alpha} \cN_{\rho^{t_2}}(Q \cap A).
\]
This is precisely the required non-concentration needed to apply  \Cref{MNC->sup} to $A^Q$. In the end, we have justified \eqref{lwbdE'} regardless of the value of $(\cN_{\rho^{t_i}}(A))_{i=1,2}$.

To bound from below $\cN_{\cR^\sV}(E)$, we proceed similarly,  by zooming into each $Q \in \cQ_2(A)$ and applying \Cref{MNC->sub} to the $\times \rho^{- t_2}$-dilation of $A \cap Q$ and the random projector $(\pi_{||V_2})_{\sV\sim \Xi}$. We obtain the subcritical estimate
\begin{equation}\label{lwbdRVE}
\cN_{\cR^\sV}(E)\geq \rho^{-\alpha t_{3}(D-d_{2}) +2\eps}\cN_{\rho^{t_2}}(A)^{d_{2}/D}  .
\end{equation}

Now \eqref{eq:NRNSll}, \eqref{lwbdE'}, \eqref{lwbdRVE} are in contradiction, which concludes the proof. 
\end{proof}


\subsection{Linearizing charts} \label{Sec-lin-charts}

Recall $X=G/\Lambda$ where $G=\SL_{d+1}(\R)$ and $\Lambda$ is a fixed arbitrary lattice.
We define on $X$ a covering of linearizing charts which do not deform balls much, and most importantly send any $g$-translate $g B_{r}x$ ($g\in G, r>0, x\in X$) to an additive translate of the box $\Ad(g)B^\kg_{r}$ provided that $\Ad(g)B^\kg_{r}$ is not too distorted and lives at a suitable scale.
We point out that these charts live at a microscopic scale, contrary to those used  in our previous work \cite{BHZ24} about the case $d=1$.
This linearizing scheme is extracted from \cite[Lemma 6.3]{BH25}, which is itself inspired by Shmerkin \cite{Shmerkin}.

\begin{lemma}[{\cite[Lemma 6.3]{BH25}}]\label{lin-charts}
Let $0<\delta \lll1$. There exists a measurable map $\varphi : \{\inj \geq \delta\}\rightarrow B^\kg_{1}$ satisfying the following.
\begin{itemize}
\item[1)] For every $r\in (0, \delta)$, $v\in \kg$, the preimage $\varphi^{-1}(B^\kg_{r}+ v)$ is covered by $O(1)$ many balls $(B_{r}x)_{x\in X}$
\item[2)] For every $r\in (0, \delta)$, $g\in G$ such that $B^\kg_{\delta^2}\subseteq \Ad(g)B^\kg_{r}\subseteq B^\kg_{\delta}$, and $x\in X$, the translate $gB_{r}x \cap \{\inj \geq \delta\}$ is covered by $O(1)$ many preimages of boxes $(\varphi^{-1}(\Ad(g)B^\kg_{r}+ v))_{v\in \kg}$.
\end{itemize}
\end{lemma}

In this lemma, 
the inequality $B^\kg_{\delta^2}\subseteq \Ad(g)B^\kg_{r}\subseteq B^\kg_{\delta}$ controls the distortion allowed on $G$-translates of balls $gB_{r}x$ to be well represented by additive translates of boxes $\Ad(g)B^\kg_{r}+v$ via the linearization.

\subsection{Dimension increment and bootstrap} \label{Sec-diminc-proof}  
We combine the results of the three previous subsections to show that the dimensional properties of a prescribed measure $\nu$ on $X$ are improved under the action of the $\mu$-random walk on $X$. This is \Cref{dim-increment}. By iteration, we deduce the desired bootstrap to high dimension, \Cref{high-dim}.

\bigskip
Recall that $\Lyap>0$ denotes the top Lyapunov exponent of the $\Ad_{\star}\mu$-random walk on $\kg$, see \eqref{def-Lyap}.

\begin{proposition}[Dimension increment] \label{dim-increment}
Let $\kappa, \eps, \rho \in (0,1/10)$, $\alpha \in {[\kappa, 1-\kappa]}$, $\tau \geq 0$ be some parameters.
Consider on $X$ a Borel measure $\nu$ of mass at most $1$, which is $(\alpha,  \cB_{[\rho^{2/3}, \rho^{1/3}]}, \tau)$-robust, and supported on $\{\inj\geq \delta^\eps\}$. Denote by $n_{\rho} \geq 0$ the integer part of $\frac{1}{10 \Lyap} \abs{\log \rho}$.

Assume $\eps, \rho \lll_{ \kappa} 1$,
then
$$
\text{$\mu^{*n_{\rho}}*\nu$ is $(\alpha+\eps, \cB_{\rho^{1/2}}, \tau+\rho^{\eps})$-robust}.
$$
\end{proposition}

\begin{proof}
We may assume $\tau=0$, and $\rho$ small enough in terms of $\eps$ as well (not only $\Lambda,\mu,\kappa$). We write $n=n_{\rho}$.
By \Cref{effective-recurrence} integrated over $\nu$, we have
\[
\mu^{*n} * \nu \bigl\{ \inj \leq \rho^{1/2} \bigr\} \ll \rho^{c /2} (e^{-c n} \rho^{-C \eps} + 1)
\]
for some constants $c > 0, C > 1$ depending on $\Lambda$ and $\mu$.
We can require $\eps \lll \frac{c}{C(\ell + 1)}$ and $\rho \lll_{c}1$ so that this leads to $\mu^{*n} * \nu \{ \inj \leq \rho^{1/2} \} \leq \frac{\rho^{\eps}}{2}$.
Thus, it remains to show that $\mu^{*n} * \nu$ can be written as a sum $\mu^{*n} * \nu= \nu' + \nu''$ of Borel measures satisfying $\nu''(X) \leq \frac{\rho^{\eps}}{2}$ and
\begin{equation}
\label{eq:dim-incre}
\sup_{y \in X} \nu'\bigl(B_{\rho^{1/2}} y\bigr) \leq \rho^{\frac{1}{2}(\alpha + \eps)\dim X }.
\end{equation}

To this end, we first linearize the situation by looking through the covering of charts from \S\ref{Sec-lin-charts}.
More precisely, we apply \Cref{lin-charts} with parameter $\delta=\rho^{1/3}$.
This yields a map $\varphi : \{\inj \geq \rho^{1/3}\} \rightarrow B^\kg_{1}$, we set $\tnu=\varphi_{\star}\nu$.
The assumption that $\nu$ is $(\alpha,  \cB_{[\rho^{2/3}, \rho^{1/3}]}, 0)$-robust and has mass at most $1$ implies, via \Cref{lin-charts} item 1) and provided $\rho\lll_{\eps}1$, that for every $r \in [\rho^{2/3}, \rho^{1/3}]$,
$$\sup_{v\in \kg} \tnu(B^\kg_{r}+v)\leq \rho^{-\eps \dim X} r^{\alpha \dim X}.$$

We now aim to apply \Cref{sup-mult} to the measure $\tnu$, and for the random box $\Ad(g^{-1})B^\kg_{\rho^{1/2}}$ where $g\sim \mu^{*n}$, or rather its close companion
$$B^{\sV_{g}}_{\rho^\bt}:=
B^{\Ad(u(-\ttb_{g}))\kg_{-}}_{\rho^{2/5}} + B^{\Ad(u(-\ttb_{g}))\kg_{\leq 0}}_{\rho^{1/2}} + B^{\kg}_{\rho^{3/5}}$$
which is a good approximation of $\Ad(g^{-1})B^\kg_{\rho^{1/2}}$ (by \eqref{comparison-box-AdgB} below), and whose partial flag we know how to control thanks to \S\ref{Sec-nc-ineq}. Indeed, by \Cref{relative-angle} and \Cref{relative-angle-2}, the distributions of $(\Ad(u(-\ttb_{g}))\kg_{-})_{g\sim \mu^{*n}}$ and $(\Ad(u(-\ttb_{g}))\kg_{\leq 0})_{g\sim \mu^{*n}}$ satisfy respectively $\MNC$ and $\MNC^\perp$ with parameters $(e^{-n}, C, c)$, or equivalently $(\rho, C, c)$ up to dividing $c$ by $11 \Lyap$. Here $C,c>0$ are constants that only depend on $\mu$.

Provided $\rho, \eps\lll 1$, the multislicing \Cref{sup-mult} yields a subset $E_{1} \subseteq G$ and some constant $\eps_{0}=\eps_{0}(\mu)>0$ such that $\mu^{*n}(E_{1})\leq \rho^{\eps_{0}}$ and for $g\in G \setminus E_{1}$, there exists a set $\tA_{g} \subseteq \kg$ with $\tnu(\kg \setminus \tA_{g}) \leq \rho^{\eps_{0}}$
 and such that for every $v \in \kg$,
\begin{align} \label{tnu-sup-bound}
 \tnu_{ |\tA_{g}} \left(B^{\sV_{g}}_{\rho^\bt}+v\right) \leq \rho^{\eps_{0}} \leb \left(B^{\sV}_{\rho^\bt}\right)^{\alpha}.
 \end{align}

On the other hand, by the large deviation principle for $\log \ttr_{g}$ and
\Cref{moment-traj}, there exists a subset $E_{2}\subseteq G$ and a constant $\gamma=\gamma(\mu, \eps)>0$ such that $\mu^{*n}(E_{2})\ll \rho^\gamma$, and for every $g\in G\smallsetminus E_{2}$, we have $\ttr_{g} \in [\rho^{\frac{1}{10} + \eps  }, \rho^{\frac{1}{10} - \eps  }]$ and $\|\ttb_{g}\|\leq \rho^{-\eps}$.
In view of  \eqref{comparison-box}, we obtain in particular
\begin{equation} \label{comparison-box-AdgB}
B^{\sV_{g}}_{\rho^\bt}\, \overset{\rho^{-O(\eps)} }{\simeq} \,\Ad(g^{-1})B^\kg_{\rho^{1/2}}.
\end{equation}

Equations \eqref{tnu-sup-bound} and \eqref{comparison-box-AdgB} together imply that for every $g\in G \setminus (E_{1}\cup E_{2})$ and $v\in \kg$,
\begin{equation} \label{tnu-sup-bound2}
\tnu_{ |\tA_{g}} \left( \Ad(g^{-1})B^\kg_{\rho^{1/2}} +v \right) \leq \rho^{\eps_{0} - O(\eps)} \leb \left(B^{\sV}_{\rho^\bt}\right)^{\alpha}.
\end{equation}

We now get back to $X$.
To control the distortion of $\Ad(g^{-1})B^\kg_{\rho^{1/2}}$, we observe for $g\in G \setminus E_{2}$, we have $B^\kg_{\rho^{2/3}}\subseteq \Ad(g^{-1})B^\kg_{\rho^{1/2}}\subseteq B^\kg_{\rho^{1/3}}$ provided $\eps\lll1$.
Applying \Cref{lin-charts} item $2)$ and \eqref{tnu-sup-bound2}, we deduce that for all $g\in G\smallsetminus (E_1 \cup E_2)$, setting $A_{g}=\varphi^{-1}(\tA_{g})$, we have $\nu(G \setminus A_{g})\leq \rho^{\eps_{0}}$ and for every $y \in X$,
\begin{equation*} 
g_\star\nu_{| A_g} (B_{\rho^{1/2}}y) = \nu_{ |A_{g}} \left( g^{-1} B_{\rho^{1/2}}y \right) \ll \rho^{\eps_{0} - O(\eps)} \leb \left(B^{\sV}_{\rho^\bt}\right)^{\alpha} \ll \rho^{\frac{1}{2}\alpha \dim X+ \eps_{0} - O(\eps)}.
\end{equation*}
Taking $\nu' = \int_{G \setminus (E_1 \cup E_2)} g_{\star} \nu_{| A_g} \dd \mu^{*n}(g)$, and $\eps \lll \eps_{0}$, $\rho \lll_{\eps}1$, this concludes the proof of \eqref{eq:dim-incre}, whence that of the proposition.
\end{proof}

We now deduce high dimension (\Cref{high-dim}) from the combination of effective recurrence (\Cref{effective-recurrence}), initial positive dimension (\Cref{Proposition: initial dimension}), and dimension increment (\Cref{dim-increment}).

\begin{proof}[Proof of \Cref{high-dim}]
Let $A>0$ be a large enough constant depending on the initial data $\mu$. Combining \Cref{Proposition: initial dimension} and \Cref{effective-recurrence}, we may assume $\kappa>0$ small enough from the start, so that for any $M>0$, for every $\rho\lll_{M}1$ and $n\geq M  \abs{\log \rho}+ A \abs{\log \inj(x)}$, the measure
\[\text{{$\mu^{*n}*\delta_{x}$} is $(\kappa, \cB_{[\rho^M, \,\rho^{1/M}]}, \rho^{\kappa/M})$-robust}. \]

By \Cref{dim-increment}, there is some small constant $\eps=\eps(\mu, \kappa)>0$, such that up to imposing from the start $M\ggg_{\kappa}1$, we have for {all} $\rho\lll_{\kappa, M}1$, all $n\geq (\frac{1}{10\Lyap}+1)M  \abs{\log \rho}+ A \abs{\log \inj(x)}$ and $r\in [\rho^{3M/2}, \rho^{3/M}]$,
\[\text{{$\mu^{*n}*\delta_x$} is $(\kappa+2\eps, \cB_{r^{1/2}}, 2\rho^{\kappa/M})$-robust}. \]
These estimates for single scales can be combined using \cite[Lemma 4.5]{BH24} to get under the same conditions:
\[\text{{$\mu^{*n}*\delta_x$} is $(\kappa+\eps, \cB_{[\rho^{3M/4}, \rho^{3/(2M)}]}, O_{\kappa,M}(\rho^{\kappa/M}))$-robust}. \]

The argument in the last paragraph can be applied iteratively, adding at each step the value $+\eps$ to the dimension provided that the latter is not yet above $1-\kappa$. 
Noting the value of $\eps$ only depends on $\mu, \kappa$, we reach dimension $1-\kappa$ in at most $m := \lceil \eps^{-1} \rceil$ steps. Iteration is allowed provided that $M$ is chosen large enough to satisfy $(\frac{3}{4})^{m}M>(\frac{3}{2})^{m}\frac{1}{M}$, i.e. $M> 2^{m/2}$.
We could further require  $M > (3/2)^m$ from the beginning to guarantee that $\rho$ belongs to the range of scales for which we have $(1 - \kappa)$-robustness when the iteration ends.
This concludes the proof.
\end{proof}

\section{From high dimension to equidistribution}\label{Sec-equidistribution}
In this section, we establish \Cref{Khintchine-dyn} and \Cref{mu^n-equidistribution}. We further establish a double equidistribution estimate (\Cref{decorrelation-sigma}) which will be useful to prove the divergent case of \Cref{Khintchine-self-similar}.

We let $(\eta_t)_{t>0}$ denote the one-parameter family of probability measures on $G$ defined by
\[\dd\eta_t :=a(t)u(\bs)\dd\sigma(\bs).\]
The next proposition states that a probability measure $\nu$ on $X$ with dimension close to $\dim X$ equidistributes with exponential rate under convolution with $\eta_{t}$. It is in fact slightly more precise as the dimension assumption on $\nu$ concerns only a single scale $\rho$, and equidistribution is guaranteed for a corresponding interval of times $t\in [\rho^{-1/2}, \rho^{-1/4}]$.

\begin{proposition}\label{endgame}
There exist $\kappa,\rho_0>0$ such that the following holds for all $\rho\in (0,\rho_0]$ and $\tau \in \R_{\geq 0}$.

Let $\nu$ be a Borel measure on $X$ which is $(1-\kappa,\cB_{\rho},\tau)$-robust with $\nu(X)\leq 1$. Set $l=\lceil\frac{1}{2}\dim \SO(d+1) \rceil$. Then for all $t\in [\rho^{-1/2}, \rho^{-1/4}]$, for all $f\in B_{\infty,l}^{\infty}(X)$ with $m_X(f)=0$, we have
\begin{equation} \label{eq:decay-eta-t}
|\eta_t*\nu(f)|\leq (\rho^{\kappa}+\tau)\cS_{\infty,l}(f).
\end{equation}
\end{proposition}

\begin{proof}
The proof is similar to that of \cite[Proposition 5.1]{BHZ24}. We provide a sketch for completeness and refer the reader to \cite{BHZ24} for details.

Denote by $(P_{\eta_t})_{t>0}$ the family of Markov operators on $L^2(X)$ associated to $(\eta_{t})_{t>0}$. It is defined by: $\forall f\in L^2(X)$,
\[P_{\eta_t} f=\int_G f(g \,\cdot ) \dd\eta_t(g).\]
 The first step of the proof is to show a spectral gap property for $(P_{\eta_t})_{t>0}$ as $t\to +\infty$, namely:
there exists $c=c(G, \Lambda, \sigma)>0$ such that for all $f \in B^{\infty}_{2,l}(X)$ with $m_{X}(f)=0$, all $t>1$, one has
\begin{equation} \label{eq:decayPt}
\|P_{\eta_{t}} f\|_{L^2} \ll t^{-c} \cS_{2,l}(f).
\end{equation}
The proof of \eqref{eq:decayPt} exploits the quantitative decay of matrix coefficients (see \cite[Lemma 3]{Bekka98} and \cite[Equations (6.1), (6.9)]{EMV09}):
\begin{equation*}
\exists \delta_{0}=\delta_{0}(\Lambda)>0,\, \forall g \in G,\quad \abs{\langle f(g \,\cdot), f \rangle_{L^2}} \ll \norm{g}^{-\delta_{0}} \cS_{2,l}(f)^2,
\end{equation*}
and the non-concentration property of $\sigma$ from \Cref{non-conc-sigma-aff}, see \cite[Proposition 5.2]{BHZ24} for details.

Once \eqref{eq:decayPt} is established, we obtain \eqref{eq:decay-eta-t} as follows.  We introduce $\nu_{\rho}$ the mollification of $\nu$ at scale $\rho$, namely
\[\nu_{\rho}:= \frac{1}{m_G(B_{\rho})} \int_{B_\rho} g_{\star}\nu \dd m_{G}(g).\]
Given $f\in B_{\infty,l}^{\infty}(X)$ with $m_{X}(f)=0$, we then have for every $t> 1$,
\begin{align*}
| \eta_{t}*\nu(f)|= \abse{\int_{X} P_{\eta_{t}} f \dd \nu}
\leq \abse{\int_{X} P_{\eta_{t}} f \dd \nu - \int_{X} P_{\eta_{t}} f \dd \nu_{\rho}} \,+\, \abse{\int_{X} P_{\eta_{t}} f \dd \nu_\rho}.
\end{align*}
 The first integral in the right hand side is bounded by $\rho \cS_{\infty, 1}(P_{\eta_{t}} f )\ll \rho t \cS_{\infty, 1}(f)$. To bound the second integral, note we may assume from the start $\tau=0$. Then $\nu_{\rho}$ satisfies $\dd \nu_{\rho}(x) = \frac{\nu(B_\rho x)}{m_G(B_{\rho})} \dd m_{X}(x)\ll \rho^{-\kappa \dim X} \dd m_{X}(x)$. Applying \eqref{eq:decayPt}, we find the second term in the right hand side is bounded by $ O(t^{-c} \rho^{-\kappa \dim X} \cS_{2,l}(f)).$ The proof is concluded by taking $t\in [\rho^{-1/2}, \rho^{-1/4}]$,  $\kappa$ small enough in terms of $c, \dim X$, and $\rho_{0}$ small enough in terms of $G$, $\kappa$.
\end{proof}

In the next lemma, we invoke the self-similarity of $\sigma$ to relate $\eta_{t}$ and convolution powers of $\mu$.

\begin{lemma}[$\eta_{t}$-process vs $\mu$-walk]\label{lm:cocycle}
Given $t>0$, $n\geq 0$, we have
\[
\eta_t = \int_{P'} \delta_{k_g}*\eta_{t \ttr_{g}} *\delta_{g} \dd\mu^{*n}(g).
\]

\end{lemma}

\begin{proof}
We observe that for any $\bs\in \R^d$ and $g\in P'$,
\begin{align*}
  k_g a(t\ttr_g)u(\bs)g&=a(t)a(\ttr_g)k_gu(\bs)k_g^{-1}a(\ttr_g^{-1})u(\ttb_g)\\
  &=a(t)u(\ttr_g O_g\bs+\ttb_g)\\
  &=a(t)u(\phi_g(\bs)).
\end{align*}
The lemma follows by the equality $\lambda^{*n}*\sigma=\sigma$.
\end{proof}

We are now able to conclude the proof of \Cref{Khintchine-dyn}. The strategy is to use \Cref{lm:cocycle} to decompose $\eta_{t}$ as a random walk part $\mu^{*n}$ (where $n=n(\mu, t)$) which generates high dimension thanks to \Cref{high-dim}, followed by some $\eta_{t'}$-part (with $t'=t'(\mu,t)$) which will convert this high dimension into equidistribution via \Cref{endgame}. The apparent obstruction is that the decomposition appearing \Cref{lm:cocycle} does not separate the $\mu$ part and the $\eta$ part, because the term $\delta_{k_g}*\eta_{t \ttr_{g}}$ involves $g$. To deal with this obstacle, we partition the space of parameters $g$ into $O(\rho^{-\alpha})$ subsets ($\alpha\lll_{\kappa}1$) in which $\delta_{k_g}*\eta_{t \ttr_{g}}$ hardly depends on $g$.

\begin{proof}[Proof of \Cref{Khintchine-dyn}]

Up to replacing $\lambda$ by a suitable convolution power until a stopping time (as in \cite[Lemma 5.3]{BHZ24}), we may assume that $\lambda$-almost every $\phi$ is orientation preserving. This way we place ourselves in the context of \Cref{Sec-notations}, and all results obtained until now are applicable.

Observe that given $t,r_0,r_1>0$ and $k_0,k_1\in K'$, we have
\[\delta_{k_0}*\eta_{t r_0}=\delta_{k_0k_1^{-1} a(r_0r_1^{-1})}*\delta_{k_1}*\eta_{tr_1}.\]
Hence, given any Borel measure $\nu$ on $X$ and $f\in B_{\infty,l}^{\infty}(X)$, we have
\begin{align}\label{etat-01}
  \abs{\delta_{k_0}*\eta_{tr_0}*\nu(f)-\delta_{k_1}*\eta_{tr_1}*\nu(f)}\ll (\norm{\Id-k_0k_1^{-1}}+\abs{\log r_0r_1^{-1}})\cS_{\infty,l}(f)\nu(X).
\end{align}
Let $\alpha, \rho>0$ be parameters to be specified later, with $\alpha$ depending only on $\Lambda,\mu$, and $\rho$ on $\Lambda, \mu, t$. We discretize the set of  $k_g$ and $\ttr_g$ for $g\in P'$ as follows. We partition the compact group $K'$ into $\rho^{-O(\alpha)}$ disjoint measurable sets $\{K'_i:i\in I\}$ such that each $K'_i$ is contained in a ball of radius $\rho^{\alpha}$ centered at some $k_i\in K'$.
We set \[\sR:=\{(1+\rho^\alpha)^k\,:\, k\in \Z\}.\]
For $i\in I, r\in \sR$, we define
\[P'(i,r):=\{g\in P'\,: \,k_g\in K'_i\ \text{ and } \ttr_g\in [r, r(1+\rho^\alpha) [\}. \]
Then
\[P'=\bigsqcup_{r\in \sR,i\in I} P'(i, r).\]
Hence, by \Cref{lm:cocycle} and \eqref{etat-01}, we obtain for every $n\geq 1$,
\begin{align}\label{eq-dec1}
\abs{\eta_t*\delta_x(f)}&= \abs{\int_{P'} \delta_{k_g}*\eta_{t\ttr_g}*\delta_g*\delta_x(f) d\mu^{*n}(g)} \nonumber \\
&\leq \sum_{i\in I, r\in \sR} \abs{\delta_{k_i}*\eta_{tr}*\mu^{*n}_{|P'(i,r)}*\delta_x(f)}+O(\rho^{\alpha}\cS_{\infty,l}(f)).
\end{align}

We now bound each term in the sum from \eqref{eq-dec1}. Let $\kappa=\kappa(\Lambda,\mu)>0$ as in \Cref{endgame}. Assume $\inj(x) \geq \rho$.
By \Cref{high-dim}, there are constants $C=C(\Lambda,\mu) > 1$ and $\eps_1=\eps_{1}(\Lambda,\mu) > 0$ such that, provided $\rho\lll 1$, the measure $\mu^{*n} * \delta_x$ on $X$ is $(1 - \kappa, \cB_\rho, \rho^{\eps_1})$-robust for any $n \geq C \abs{\log \rho}$.
For the rest of this proof, we choose $\rho$ and $n$ depending on $\Lambda, \mu, t$ so that
\begin{equation}
\label{eq:t in rho}
t = \rho^{-C \Lyap - 3/8} \,\,\,\,\,\,\,\,\,\,\,\,\,\,\,\,\,\,\,\,\,\,\,\,\,\,\,\,\,\,\,\, n=\lceil C \abs{\log \rho}\rceil,
\end{equation}
where $\Lyap$ has been defined in \eqref{def-Lyap}. Consider
\[
\sR' = \set{r \in \sR : \rho^{-1/4} \leq t r \leq \rho^{-1/2}} = \sR \cap [\rho^{C\Lyap +1/8}, \rho^{C\Lyap - 1/8}].
\]

On the one hand, by the principle of large deviations and our choice for $n$, we have for some $\eps_{2}=\eps_{2}(\mu, C)>0$,
\begin{equation}\label{eq:bound1-Th1.2}
\mu^{*n} \{g\,:\, \ttr_{g} \notin \sR' \} \leq \rho^{\eps_{2}}
\end{equation}

On the other hand, for $i\in I$, $r\in \sR'$, observing that $\mu^{*n}_{|P'(i,r)}*\delta_x\leq \mu^{*n}*\delta_x$ is $(1 - \kappa, \cB_\rho, \rho^{\eps_1})$-robust, and $\cS_{\infty,l}(f\circ k_{i})\ll \cS_{\infty,l}(f)$ , $m_X(f\circ k_{i})=m_X(f)$, we get via \Cref{endgame}, for $t\ggg1$,
\begin{equation}\label{eq:bound2-Th1.2}
\abs{\eta_{t r} * \mu^{*n}_{|P'(i,r)} * \delta_x (f)} \leq (\rho^\kappa + \rho^{\eps_1}) \cS_{\infty, l}(f).
\end{equation}

Combining \eqref{eq:t in rho}, \eqref{eq:bound1-Th1.2}, \eqref{eq:bound2-Th1.2} and choosing $\alpha$ small enough in terms of $\eps_{1},\eps_{2}, \kappa$, we obtain the bound announced by \Cref{Khintchine-dyn}. So far, we have worked under the condition $\inj(x) \geq t^{-(C \Lyap + 3/8)^{-1}}$. Noting the claim in \Cref{Khintchine-dyn} is trivial otherwise, the proof  is complete.

\end{proof}

\begin{proof}[Proof of \Cref{mu^n-equidistribution}]
By a similar argument, we see  \Cref{endgame} is still valid with $(t,\eta_{t})$ replaced by $(e^{k}, \mu^k)$.
Replacing \eqref{lm:cocycle} by the equality $\mu^{*(k+n)}=\mu^{*k}*\mu^{*n}$, we can then argue as in the proof of \Cref{Khintchine-dyn}. Details are left to the reader.
\end{proof}

\noindent{\bf Double equidistribution}.
We conclude this section by upgrading \Cref{Khintchine-dyn} into a double equidistribution property. This upgrade will play a role to prove the divergent case of the Khintchine dichotomy.

 Given bounded measurable functions $f_1,f_2:X\to \R$ and $t_2\geq t_1\geq 0$, we introduce the double equidistribution coefficient
\begin{equation} \label{double-eq-function}
\Delta^{\sigma}_{f_1,f_2}(t_1,t_2) := \abse{\int_{\R^d} f_{1}\bigl(a(t_1)u(\bs)x_{0}\bigr) f_{2} \bigl(a(t_2)u(\bs)x_{0}\bigr) \dd \sigma(\bs) - m_{X}(f_{1})m_{X}(f_{2})}.
\end{equation}
We recall that in the above, $x_0=\Lambda/\Lambda$ denotes the identity coset of $X$.

The following proposition gives a quantitative upper bound on $\Delta^{\sigma}_{f_1,f_2}(t_1,t_2) $ provided the times $t_1,t_2$ are sufficiently separated.

\begin{proposition}[Effective double equidistribution of expanded fractals] \label{decorrelation-sigma}
For every $\eta>0$, there exist $C, c>0$ such that for all $ t_{1},t_{2}>1$ with $t_2\geq t^{1+\eta}_{1}$ and $f_{1}, f_{2}\in B^\infty_{\infty,l}(X)$, we have
\begin{equation} \label{double-eq-prop-sigma}
\Delta^{\sigma}_{f_1,f_2}(t_1,t_2)
\leq C \cS_{\infty, l}(f_1) \abs{m_{X}(f_2)} t_1^{-c} + C\cS_{\infty, l}(f_{1})\cS_{\infty, l}(f_{2}) t_2^{-c}.
\end{equation}
\end{proposition}
\medskip

\begin{remark} A quantitative bound on $\Delta^{\sigma}_{f_1,f_2}(t_1,t_2)$ without separation condition on $t_{1},t_{2}$ will be extrapolated below, see Equation \eqref{eq:all-range}.
\end{remark}

\begin{proof}
The proof is the  same as that of \cite[Proposition 6.1]{BHZ24}, using \Cref{Khintchine-dyn} in the place of \cite[Theorem B']{BHZ24}.
\end{proof}

\section{Quantitative Khintchine dichotomy in $\R^d$ from equidistribution}\label{Sec-Khintchine-dich}

We show that an arbitrary probability measure $\xi$ on $\R^d$ obeys the Khintchine dichotomy provided that the pushfoward $a(t)u(\bs)\SL_{d+1}(\Z)\dd\xi(\bs)$ satisfies certain effective equidistribution properties in $\SL_{d+1}(\R)/\SL_{d+1}(\Z)$ for large $t$. Combined with \Cref{decorrelation-sigma}, this yields \Cref{Khintchine-self-similar}.

Throughout the section, notations refer to \Cref{Sec-notations}, and we further specify $\Lambda=\SL_{d+1}(\Z)$, in particular $X=\SL_{d+1}(\R)/\SL_{d+1}(\Z)$.
We also set $x_0 = \Lambda/\Lambda \in X$.

\begin{definition} Let $\xi$ be a probability measure on $\R^d$.
We say that $\xi$ satisfies the \emph{effective single equidistribution property on $X$} if there are constants $C,c>0$ and $l\in \N$ such that
\begin{multline}\label{single-eq-prop}
\forall f \in B^\infty_{\infty,l}(X),\, \forall t > 1,\\
\abse{\int_{\R^d} f\bigl(a(t) u(\bs) x_0\bigr) \dd \xi(\bs) - m_{X}(f) } \leq C \cS_{\infty, l}(f) t^{-c}.
\end{multline}

We say that $\xi$ satisfies the \emph{effective double equidistribution property on $X$} if for every $\eta > 0$, there are constants $C,c>0$ and $l\in \N$ such that
\begin{multline}\label{double-eq-prop}
\forall f_{1},f_{2} \in B^\infty_{\infty,l}(X),\, \forall t_{1} > 1,\, \forall t_{2}> t_{1}^{1+\eta},\\
\Delta^{\xi}_{f_1,f_2}(t_1,t_2)
\leq C\cS_{\infty, l}(f_1) \abs{m_{X}(f_2)} t_1^{-c} + C\cS_{\infty, l}(f_{1})\cS_{\infty, l}(f_{2}) t_2^{-c}
\end{multline}
where the notation $\Delta^{\xi}_{f_1,f_2}(t_1,t_2) $ is defined in \eqref{double-eq-function}.
\end{definition}
Taking $f_{2}=1$ and $t_{2}\to+\infty$, we see that effective double equidistribution \eqref{double-eq-prop} implies effective single equidistribution \eqref{single-eq-prop}. Note that \eqref{double-eq-prop} assumes some separation $t_{2}> t_{1}^{1+\eta}$ between $t_{1}$ and $t_{2}$. As it turns out, \eqref{double-eq-prop} (with small enough $\eta$) can in fact be automatically upgraded to the following full range estimate:
there exist (potentially different) constants $C,c > 0$, $l\in \N$, such that for every $f_1,f_2 \in B^{\infty}_{\infty,l}(X) $ and all $t_2 \geq t_1 > 1$.
\begin{multline}
\label{eq:all-range}
\Delta^{\sigma}_{f_1,f_2}(t_1,t_2)
\leq C \cS_{2,l}(f_1) \cS_{2,l}(f_2) t_1^{c} t_2^{-c} \\ + C \cS_{\infty, l}(f_1) \abs{m_{X}(f_2)} t_1^{-c} + C \cS_{\infty, l}(f_{1})\cS_{\infty, l}(f_{2}) t_2^{-c}.
\end{multline}
The implication from \eqref{double-eq-prop} to \eqref{eq:all-range} (again, parameters $C,c,l$ may differ) is explained in \cite[Section 7.1]{BHZ24}. The idea is that \eqref{double-eq-prop} implies single equidistribution, which in turn, by decay of matrix coefficients, yields effective double equidistribution in the short range regime $t_{1}\leq t_{2}\leq t_{2}^{1+\eta}$ for sufficiently small $\eta$.

\bigskip

The following result of Khalil-Luethi \cite{KL23} guarantees that effective single equidistribution implies the convergent case of the Khintchine dichotomy.

\begin{thm}[Convergent case {\cite[Theorem 9.1]{KL23}}] \label{convergent-Khintchine}
Let $\xi$ be a probability measure on $\R^d$ satisfying the effective single equidistribution property \eqref{single-eq-prop} on $\SL_{d+1}(\R)/\SL_{d+1}(\Z)$. Then for every non-increasing function $\psi: \N \to \R_{\geq 0}$ such that $\sum_{q\in \N} \psi(q)^d<\infty $, we have
$$\xi (W(\psi))=0. $$
\end{thm}

We show that effective double equidistribution implies the divergent case of the Khintchine dichotomy. Our result is in fact quantitative: we provide the asymptotic of the number of solutions of the Khintchine inequality when bounding the denominator.

\begin{thm}[Divergent case]\label{General-quant-Khint} Let $\xi$ be a probability measure on $\R^d$ satisfying the effective double equidistribution property \eqref{double-eq-prop} on $\SL_{d+1}(\R)/\SL_{d+1}(\Z)$.
Let $\psi: \N \to \R_{\geq 0}$ be a non-increasing function such that $\sum_{q\in \N} \psi(q)^d=\infty$.

Then $\xi (W(\psi))=1$ and for $\xi$-a.e. $\bs\in \R^d$, we have as $N\to +\infty$:
 \begin{align}\label{eq:primsol}
  \big|\set{(\boldsymbol{p},q)\in \Z^{d}\times \llbracket 1,N\rrbracket \,:\, \forall i\in \llbracket1, d\rrbracket, \, 0\leq qs_i-p_i<\psi(q) } \big| \sim_{\bs, \psi} \sum_{q=1}^N \psi(q)^d.
 \end{align}
\end{thm}

\begin{remark} 
A light variation of the proof allows to estimate the number of \emph{primitive} solutions of the Khintchine inequality. More precisely, consider $\cP(\Z^{d+1}):=\Z^{d+1}\smallsetminus \cup_{k\geq 2}k\Z^{d+1}$ the set of primitive vectors in $\Z^{d+1}$. Set $\cP(\Z^{d+1})_{N}=\cP(\Z^{d+1})\cap (\Z^{d}\times \llbracket 1,N\rrbracket)$. Then for $N\to +\infty$, we have
 \begin{align}\label{primitiv-count}
  \big|\{(\boldsymbol{p},q)\in \cP(\Z^{d+1})_{N}\,:\, \forall i\in \llbracket1, d\rrbracket, \, 0\leq qs_i-p_i<\psi(q) \} \big| \sim_{\bs, \psi} \zeta(d+1)^{-1}\sum_{q=1}^N \psi(q)^d,
 \end{align}
where $\zeta(t)=\sum_{n\geq 1} n^{-t}$ denotes the Riemann zeta function.

Note also that in both cases (non-primitive and primitive), given an arbitrary subset of subscripts $I\subseteq  \llbracket1, d\rrbracket$, we may replace the conditions $0\leq qs_{i}-p_{i}<\psi(q)$ ($i\in I$) in the above sets by $0\leq p_{i}- qs_{i}<\psi(q)$ ($i \in I$) without affecting the asymptotic.
\end{remark}

The work conducted in previous sections guarantees that self-similar measures satisfy the conditions required in Theorems \ref{convergent-Khintchine}, \ref{General-quant-Khint}. Provided   \Cref{General-quant-Khint} holds, we directly deduce \Cref{Khintchine-self-similar}.
\begin{proof}[Proof of \Cref{Khintchine-self-similar}] 
The convergent case  follows from combining Theorems \ref{Khintchine-dyn}, \ref{convergent-Khintchine}. The divergent case is a consequence of \Cref{decorrelation-sigma} and \Cref{General-quant-Khint} (along with its subsequent remark).
\end{proof}

It remains to show \Cref{General-quant-Khint}.  We will distinguish the cases where $d\geq 2$ and $d=1$. The reason for that distinction is that the Siegel transform of the characteristic function of a ball in $\R^d$ has finite second moment when $d\geq 2$ and infinite second moment when $d=1$. We present the case $d\geq 2$ and explain afterward how the proof can be adapted, by the mean of a suitable truncation, to obtain the case $d=1$.

\bigskip
\subsection{Case $d\geq 2$, the lower bound} \label{Sec-lwbnd-d2}
We show the lower asymptotic in \Cref{General-quant-Khint}.

Given $N\geq 1$ and $\bs\in \R^d$, we denote the left-hand side of \eqref{eq:primsol} by
\[\sT_N(\bs):= \bigl\lvert\set{(\boldsymbol{p},q)\in \Z^{d}\times \llbracket 1,N\rrbracket \,:\, \forall i\in \llbracket1, d\rrbracket, \, 0\leq qs_i-p_i<\psi(q) }  \bigr\rvert.\]
We extend $\psi$ to a non-increasing function on $\R_{\geq 0}$ by setting $\psi(q)=\psi(\lceil q \rceil)$. From now on we fix a parameter $\tau\in (1,2]$. For any $k\in \N, \bs\in \R^d$, let
\[ \sS_{k}(\bs) := \bigl\lvert\set{(\boldsymbol{p},q)\in \Z^{d}\times \rrbracket \tau^{k-1}, \tau^{k}\rrbracket \,:\, \forall i\in \llbracket1, d\rrbracket, \,0\leq q s_i-p_i<\psi(\tau^k) }  \bigr\rvert. \]
For $N\geq 1$, letting $n\in \N$  such that $\tau^n\leq N<\tau^{n+1}$, and using that $\psi$ is non-increasing, we  have
\begin{equation}\label{sTsT}
\sT_N(\bs)\geq \sT_{\tau^n}(\bs)\geq \sum_{k=1}^n \sS_k(\bs).
\end{equation}

We will obtain the lower bound for $\sT_N(\bs)$ via an asymptotic lower bound for $\sum_{k=1}^n \sS_k(\bs)$. More precisely, we will show

\begin{proposition}\label{pr:rough}
Under the assumptions of \Cref{General-quant-Khint} and with $d\geq 2$, we have for every $\tau \in (1,2]$, for $\xi$-almost all $\bs \in \R^d$, for every $\eta >0$, for all large enough $n$, 
\begin{equation}\label{prrough-eq}
\sum_{k=1}^n \sS_{k}(\bs) \geq (1-\tau^{-1}-\eta) \sum_{k=1}^n \psi(\tau^k)^d \tau^k.
\end{equation}
\end{proposition}

The lower bound in \Cref{General-quant-Khint} follows directly:

\begin{proof}[Proof of the lower bound in \eqref{eq:primsol} using \Cref{pr:rough}]
Let $\eps \in (0, 1/2)$.
As $\psi$ is non-increasing, we have for large $k$,
$$(1-\tau^{-1}) \psi(\tau^k)^d\tau^{k+1} \geq (1-\eps)(\lceil \tau^{k+1} \rceil - \lceil \tau^{k} \rceil)\psi(\tau^k)^d \geq (1-\eps)\sum_{q= \lceil \tau^k \rceil}^{\lceil \tau^{k+1} \rceil - 1} \psi(q)^d$$
Summing over $k$ and using the divergence $\sum_{q\in \N} \psi(q)^d=\infty$, we obtain that for every large enough $N$, and $n\geq 1$ such that $\tau^n \leq N < \tau^{n+1}$,
\begin{align}\label{eq1taupsi}
(1-\tau^{-1})\tau \sum_{k=1}^n \psi(\tau^k)^d\tau^{k} \geq (1 - 2\eps) \sum_{q=1}^N \psi(q)^d
\end{align}
We now compare the left-hand side of \eqref{eq1taupsi} with the right-hand side of \eqref{prrough-eq}. For that, choose $\tau$ close enough to $1$ so that $\tau^{-1}\geq 1-\eps$.
Taking $\eta>0$ with $\eta\lll_{\tau, \eps}1$, we then have $(1-\tau^{-1}-\eta) \geq (1-\eps)^2(1-\tau^{-1})\tau$. Using \Cref{pr:rough}, then \eqref{sTsT}, we obtain that for $\xi$-almost every $\bs\in \R^d$, for large enough $N$,
$$\sT_N(\bs)\geq (1 - 2\eps)^3 \sum_{q=1}^N \psi(q)^d. $$
This justifies the lower bound asymptotic $\liminf_{N \to +\infty} \frac{\sT_N(\bs)}{\sum_{q=1}^N \psi(q)^d} \geq 1$.
\end{proof}

We now focus on the proof of \Cref{pr:rough}. For that, we need a dynamical interpretation of $\sS_k(\bs)$.
Recall $X=\SL_{d+1}(\R)/\SL_{d+1}(\Z)$ throughout the section.
Given a measurable non-negative function $f:\R^{d+1}\rightarrow \R_{\geq0}$, we denote by $\widetilde{f} : X\rightarrow [0, +\infty]$  its \emph{Siegel transform}.
It is given by: for $g\in G$,
$$\widetilde{f}(gx_{0})= \sum_{v\in \Z^{d+1}\smallsetminus \{0\}} f(gv). $$
We interpret $\sS_k(\bs)$ dynamically by the mean of a Siegel transform.
We fix some $\tau \in (1,2)$.
For each $k\in \N$, define $r_k,t_k\in \R_{>0}$ by the relations
\[\psi(\tau^k)=r_k t_k^{-\frac{1}{d+1}},\quad \tau^k=t_k^{\frac{d}{d+1}}r_k,\]
or equivalently,
\begin{align}\label{tk and rk}
   \psi(\tau^k)^d\tau^k= r_k^{d+1}, \quad \tau^k\psi(\tau^k)^{-1}=t_k.
\end{align}
Consider the box
\[R_k=[0,r_k)^d\times (\tau^{-1}r_k,r_k]\subset \R^{d+1}.\]
By direct computation, we have
\begin{equation}\label{formula-sSk-Sieg}
\sS_k(\bs)=\widetilde\1_{R_k}\bigl(a(t_{k})u(\bs)x_{0}\bigr).
\end{equation}

Let $\gamma_{1} \in (0,1/2)$ be a small parameter to be specified later.
We partition the subscripts $k$'s into two families :
\[\Kb:=\set{k\geq 3\,:\, r_{k} > t_{k}^{\gamma_{1}}} \quad \text{and} \quad
    \Ks:=\set{k \geq 3 \,:\, k \notin \Kb}. \]
Given $n\geq 1$, we set $\Kb(n)=\Kb\cap \llbracket 1, n\rrbracket$, and $\Ks(n)=\Ks\cap \llbracket 1, n\rrbracket$.
We will establish the lower asymptotics required by \Cref{pr:rough} for the sums $\sum_{k \in \Kb(n)} \sS_{k}(\bs)$ and $\sum_{k \in \Ks(n)} \sS_{k}(\bs)$ separately.

\bigskip

\noindent{\bf Lower asymptotic over $\Kb$}.
We start with the lower asymptotic for the sum $\sum_{k \in \Kb(n)} \sS_{k}(\bs)$. For this, we only use that $\xi$ satisfies effective single equidistribution \eqref{single-eq-prop} and we do not need any restriction on $d$ (i.e. $d=1$ is allowed). Below, implicit constants in $\ll$, $\lll$ and $O(\cdot)$ will be allowed to depend not only on $\lambda$, but also on $\psi$, and the constants $C,c>0$, $l\in \N$ from \eqref{single-eq-prop}. 

\bigskip
We first show that a typical geodesic trajectory sampled by $\xi$ has at most a very slow escape to infinity along the sequence of times $(t_{k})_{k\geq1}$.

\begin{lemma}\label{logarithm law}
For $\xi$-almost every $\bs\in \R^d$, for all sufficiently large $k\geq 1$ (depending on $\bs$), we have
\[\dist(a(t_{k})u(\bs)x_0, x_{0}) \ll \log \log t_{k}.\]
\end{lemma}

\begin{proof}
For $r>1$, let $f_{r}:X\rightarrow [0, 1]$ be a smooth function such that $\cS_{\infty, l}(f_{r})\ll 1$ and $\1_{\dist(\cdot, x_{0})\geq r} \leq f_{r} \leq \1_{\dist(\cdot, x_{0})\geq r/2} $.
Then by \Cref{lm:volCusp},
\[m_{X}(f_{r})\leq m_{X} \{\dist(\cdot, x_{0})\geq r/2\} \leq M e^{- r/M}\]
for some $M=M(d)>0$.
Applying effective single equidistribution \eqref{single-eq-prop} with test function $f_{r}$ at time $t > 1$, we get
$$\xi\{\bs \,:\, \dist(a(t)u(\bs)x_0, x_{0}) \geq r \} \leq \int_{\R^d} f_{r} (a(t)u(\bs) x_{0}) \dd \xi(\bs) \ll M e^{- r/M} + t^{-c}.$$
Recalling $t_{k}\gg \tau^k$ where $\tau>1$, the right hand side has converging series over $(t,r) \in \{(t_{k}, (M+1) \log \log t_{k}) \,:\, k\geq 1\}$, and the claim follows by the Borel-Cantelli Lemma.
\end{proof}

The next lemma expresses that the counting measure on a covolume $1$ lattice of $\R^{d+1}$ is a good volume estimate for a box in $\R^{d+1}$, provided the box has large enough side length depending on the distortion of the lattice.


\begin{lemma}\label{counting}
Let $R \subset \R^{d + 1}$ be a subset of the form
\[
R = v + \prod_{i=1}^{d+1}[0, T_{i}]
\]
where $v \in \R^{d + 1}$ and $(T_{i})_{i=1}^{d+1}\in \R_{>0}^{d+1}$.
Let  $g \in G$ with $\norm{g} \leq \min_{1 \leq i \leq d + 1} \frac{T_i}{\sqrt{d+1}}$. Then
\[
\abse{ \abs{g \Z^{d+1} \cap R} - \Leb(R)} \leq 2^{d+1}\sqrt{d+1} \max_{1 \leq i \leq d + 1} \frac{\norm{g}}{T_i} \Leb(R).
\]
\end{lemma}

\begin{proof}
Set $Q := g (-\frac{1}{2},\frac{1}{2})^{d + 1}$.
The symmetric difference of $(g \Z^{d + 1} \cap R) + Q$ and $R$ is contained in $\partial R + Q$.
Taking the volume, we obtain
\[
\abse{ \abs{g \Z^{d+1} \cap R } - \Leb(R)} \leq \Leb(\partial R + Q).
\]
Note that $Q \subset B^{\R^{d+1}}_\rho$ where $\rho := \frac{\sqrt{d + 1}}{2} \norm{g} \leq \frac{1}{2}\min_{i} T_i$, in particular 
$$\partial R + Q\subseteq v\,+\,\prod_{i=1}^{d+1}[-\rho, T_{i}+\rho] \smallsetminus \prod_{i=1}^{d+1}(\rho, T_{i}-\rho).$$
It follows that
\begin{align*}
\frac{\Leb(\partial R + Q)}{\Leb(R)}& \leq \prod_{i = 1}^{d+ 1} \bigl(1 + \frac{2\rho}{T_i}\bigr) - \prod_{i = 1}^{d+ 1} \bigl(1 - \frac{2\rho}{T_i}\bigr)\\
& \leq \bigl(1 + \max_i \frac{2\rho}{T_i}\bigr)^{d+1} - \bigl(1 - \max_i \frac{2\rho}{T_i}\bigr)^{d+1}\\
& \leq 2^{d+2} \max_i \frac{\rho}{T_i}
\end{align*}
where the last bound is obtained by expanding the power $d+1$ and using $\frac{2\rho}{T_i}\leq 1$. This yields the desired estimate.
\end{proof}

We infer from Lemmas \ref{logarithm law}, \ref{counting} the asymptotic lower bound for $\sum_{k \in \Kb(n)} \sS_{k}(\bs)$.

\begin{lemma}
 \label{Kbig-asymp}
Assume $\sum_{q\in \Kb}\psi(\tau^k)^d\tau^k= +\infty$. Then for $\xi$-almost every $\bs\in \R^d$, for every $\eta>0$, for large enough $n$,
$$ \sum_{k\in \Kb(n)} \sS_{k}(\bs) \geq (1-\tau^{-1}-\eta) \sum_{k\in \Kb(n)}  \psi(\tau^k)^d \tau^k$$
\end{lemma}

\begin{proof}
Recall that for every $\bs\in \R^d$, $k\geq 3$, we have $\sS_{k}(\bs)=\widetilde{\1}_{R_{k}}(a(t_{k})u(\bs)x_{0})$. Assuming $k\in \Kb$, we have that $R_{k}$ is a box of minimal side length $(1-\tau^{-1})r_{k}\geq (1-\tau^{-1}) t^{\gamma_{1}}_{k}$. Moreover \Cref{logarithm law} guarantees that for $\xi$-almost every $\bs\in \R^d$, for large enough $k$, say $k\geq k_{\bs}$, we have $a(t_{k})u(\bs)\in g\SL_{d+1}(\Z)$ where $\|g\|\leq (\log t_{k})^{O(1)}\leq t_{k}^{\gamma_{1}/100}$. Invoking \Cref{counting}, we obtain in this context
$$\widetilde{\1}_{R_{k}}(a(t_{k})u(\bs)x_{0})\geq (1- t^{-\gamma_{1}/4}_{k}) \Leb(R_{k}).$$
Recalling from \eqref{tk and rk} that $ \Leb(R_{k}) =(1-\tau^{-1})\psi(\tau^k)^d\tau^k$, we deduce
$$ \sum_{\substack{k\in \Kb(n) \\ k\geq k_{\bs}}} \sS_{k}(\bs) \,\geq\,(1-\tau^{-1}) \sum_{\substack{k\in \Kb(n) \\ k\geq k_{\bs}}} (1- t^{-\gamma_{1}/4}_{k}) \psi(\tau^k)^d \tau^k$$
and the lemma follows using the right hand side is divergent by hypothesis.
\end{proof}

\bigskip
\noindent{\bf Lower asymptotic over $\Ks$}.
We now establish an asymptotic lower bound for the partial sums $\sum_{k\in \Ks(n)} \sS_{k}(\bs)$. Let $\eps, \gamma_{2}\in (0,1/2)$ be small parameters to be specified below.
For $k\in \Ks$, set
$$R^-_k=[\eps r_{k},(1-\eps)r_k)^d\times ((\tau^{-1}+\eps)r_k,(1-\eps)r_k]$$
the rectangle obtained from $R_{k}$ by shrinking sides via $\eps r_{k}$.
Let $\chi_{k}: X\rightarrow \{0, 1\}$ be the truncation function given by
\begin{equation}\label{truncation}
  \chi_k(x)=\begin{cases}
    1 \quad \text{if }\inj(x)\geq t_k^{-\gamma_2},\\
    0 \quad \text{ otherwise.}
  \end{cases}
\end{equation}
Let $\theta : B_{\eps/10}\rightarrow \R_{\geq 0}$ be a smooth bump function such that $m_{G}(\theta)=1$ and $\cS_{\infty,l}(\theta)\leq \eps^{-D}$ where $D=D(G,l)>0$.
Set
$$\varphi_{k}=\theta*(\chi_{k}\widetilde{\1}_{R_{k}^-}).$$
We view $\varphi_{k}$ as a \emph{bounded} and \emph{smooth}  approximation of $\widetilde{\1}_{R_{k}^-}$. Note that every $g\in B_{\eps/10}$ satisfies $g R_{k}^{-} \subseteq R_{k}$, whence $\varphi_{k}\leq\widetilde{\1}_{R_{k}}$, and in particular
$$\sS_{k}(\bs)\geq \varphi_{k}(a(t_{k})u(\bs)x_{0}).$$
Therefore, we will focus on establishing a lower bound for the partial sums of terms $\varphi_{k}(a(t_{k})u(\bs)x_{0})$ as $k$ runs along $\Ks$.

Below, implicit constants in $\ll$, $\lll$ and $O(\cdot)$ will be allowed to depend not only on $\lambda$, but also on $\psi$, $\tau$, $\eps$, and the constants $C,c>0$, $l \in \N$ from the full-range double equidistribution estimate \eqref{eq:all-range}.

\bigskip
We first recall well-known moment estimates for the Siegel transform of characteristic functions of bounded sets, see \cite[pages 2-3]{Schmidt60}. We emphasize here that we work under the assumption $d\geq 2$ (otherwise \eqref{Sieg2} does not hold, see \S\ref{Sec-d=1}).
\begin{fact}[Moments of Siegel transforms \cite{Schmidt60}] \label{fact-Siegel}
Let $R\subseteq \R^{d+1}$ be a bounded measurable subset. Then
\begin{align}
&\int_{X} \widetilde{\1}_{R} \dd m_{X}=\Leb(R), \label{Sieg1}\\
 &\int_{X} (\widetilde{\1}_{R})^2 \dd m_{X} = \Leb(R)^2 + O(\Leb(R)). \label{Sieg2}
 \end{align}
\end{fact}

We also record that convolution with a (signed) bump function does not increase the $L^2$-norm.
\begin{lemma} \label{Young-ineq}
For every measurable functions $\iota \in L^1(G)$, $F \in L^2(X)$, we have
$\|\iota*F\|_{L^2}\leq \|\iota\|_{L^1}\|F\|_{L^2}.$
\end{lemma}
\begin{proof}
Using the triangle inequality for the $L^2$-norm, and the fact that $\|g_{\star}F\|_{L^2}=\|F\|_{L^2}$, we have $\|\iota*F\|_{L^2} = \|\int_{G} \iota(g) g_{\star}F \dd m_{G}(g)\|_{L^2} \leq \int_{G} \abs{\iota(g)} \, \|g_{\star}F\|_{L^2} \dd m_{G}(g)= \|\iota\|_{L^1} \|F\|_{L^2}$.
\end{proof}
\bigskip

We deduce from \Cref{fact-Siegel} and \Cref{Young-ineq} several moment estimates for the functions $\varphi_{k}$.
\begin{lemma}\label{norm-estim}
 If $\gamma_{1} \lll \gamma_{2}$, then for some $M=M(d)>1$, every $k\in \Ks$, we have
\begin{align}
&m_X(\varphi_k) = \Leb(R_{k}^-) - O(t_{k}^{-\gamma_{2}/M}) \label{L1norm-varphik},\\
&\cS_{\infty,l}(\varphi_k) \ll t_{k}^{M\gamma_{2}} \label{eq:controlSobolev1},\\
&\cS_{2,l}(\varphi_k) \ll m_X(\varphi_k) + \sqrt{m_X(\varphi_k)} +t_{k}^{-\gamma_{2}/M}. \label{eq:controlSobolev2}\
\end{align}
\end{lemma}

\begin{proof} Let us prove \eqref{L1norm-varphik}. Note first $m_{X}(\varphi_{k})=m_{X}(\chi_{k}\widetilde{\1}_{R_{k}^-})$.
 Applying \eqref{Sieg1}, followed by the Cauchy-Schwarz inequality, \Cref{lm:volCusp} and \eqref{Sieg2}, we find
\begin{align*}
0\leq \Leb(R_{k}^-) -m_X(\varphi_k)
 &= \int_{X} \1_{\inj(x)<t^{-\gamma_{2}}_{k}} \widetilde{\1}_{R_{k}^-}(x)\, \dd m_{X}(x) \\
 &\leq \sqrt{m_{X}\{\inj <t^{-\gamma_{2}}_{k}\} }\, \|\widetilde{\1}_{R_{k}^-}\|_{L^2}\\
 &\ll t_{k}^{-\gamma_{2}/M} \max ( \Leb(R_{k}^-), \sqrt{\Leb(R_{k}^-)})
\end{align*}
for some $M=M(d)>1$.
But $\Leb(R_k^-) \leq (1 - \tau^{-1}) r_k^{d + 1} \ll t_k^{O(\gamma_1)}$ since $k \in \Ks$.
Thus upon letting $\gamma_1 \lll \gamma_2$, the right-hand side can be bounded by $t_k^{-\gamma_2/(2M)}$,
validating \eqref{L1norm-varphik}.

We now deal with \eqref{eq:controlSobolev1}. Note that $\cS_{\infty,l}(\varphi_k) \ll \cS_{\infty,l}(\theta) \| \chi_{k} \widetilde{1}_{R_{k}^-}\|_{L^\infty}$. By construction, we have $\cS_{\infty,l}(\theta) \ll 1$. On the other hand, $\|\chi_{k} \widetilde{\1}_{R_{k}^-}\|_{L^\infty} = \sup_{x \,:\, \inj(x) \geq t_k^{-\gamma_2}}\widetilde{\1}_{R^-_{k}}(x)$.
For such $x$, Equation \eqref{eq:cominjdist} allows to write $x = g x_0$ where  $g \in G$ satisfies
$\norm{g^{-1}} \ll t_k^{M \gamma_2}$ for some $M = M(d) > 1$.
Then $\widetilde{\1}_{R^-_{k}}(x) = \abs{\Z^{d+1} \cap g^{-1} R^-_k}$ with $g^{-1} R^-_k$  contained in a ball of radius $O( \norm{g^{-1}} r_k)=O(t_k^{M \gamma_2}r_{k})$.
Recalling the assumption $k \in \Ks $, this implies $\widetilde{\1}_{R^-_{k}}(x) \ll t_k^{(d+1)M (\gamma_2 + \gamma_1)}$.
This shows \eqref{eq:controlSobolev1}.

Finally, we check \eqref{eq:controlSobolev2}.
By \Cref{Young-ineq} followed by \eqref{Sieg2}, we  have $$\cS_{2,l}(\varphi_k) \leq \cS_{1,l}(\theta) \|\chi_{k} \widetilde{\1}_{R_{k}^-}\|_{L^2}\ll \|\widetilde{\1}_{R_{k}^-}\|_{L^2}\ll \Leb(R_{k}^-)+ \sqrt{\Leb(R_{k}^-)}.$$
 Now \eqref{eq:controlSobolev2} follows from \eqref{L1norm-varphik}.
\end{proof}

We consider $(\R^d,\xi)$ as a probability space.
Expectation $\E[\,\cdot\,]$ refers implicitely to this probability space.
 For every $k \in \Ks $, we introduce the random variable
\[ Y_k : \R^d \to \R, \quad \bs \mapsto \varphi_k\bigl(a(t_{k})u(\bs)x_{0}\bigr).\]
We write
\[
y_k = m_X(\varphi_k) \in \R_{\geq 0}
\]
and set $Z_k = Y_k - y_k$ the (quasi-recentered) companion of $Y_k$. The next lemma bounds the second moment of $Z_{k}$ by $y_{k}$, provided $y_{k}$ is not too small. It relies on single effective equidistribution of expanding translates of $u(\bs)x_{0}\dd\xi(\bs)$.

\begin{lemma} \label{variance-bound} Assume $\gamma_{1}\lll\gamma_{2}\lll 1$. Then for every $k\in \Ks$ such that $y_{k}\geq 1$, we have
$$\E[Z_{k}^2]\ll y_{k}.$$
\end{lemma}

\begin{proof}
By effective single equidistribution of expanding translates \eqref{single-eq-prop}, we have
\begin{align}
\E[Z_{k}^2] &= \int_{\R^d}(\varphi_{k}(a(t_{k})u(\bs)x_{0}) - y_{k})^2\dd \xi(\bs) \nonumber\\
&= \int_{X}(\varphi_{k}(x) - y_{k})^2\dd m_{X}(x) + O(\cS_{\infty,l}([\varphi_{k}-y_{k}]^2) t_{k}^{-c}).\label{eq-Zk2}
\end{align}

Let us bound the error term in \eqref{eq-Zk2}. By \eqref{eq:controlSobolev1}, we have
$$\cS_{\infty,l}([\varphi_{k}-y_{k}]^2)\ll \cS_{\infty,l}(\varphi_{k})^2 + y^2_{k}\ll t_{k}^{2M\gamma_{2}}+y^2_{k}.$$
Taking $\gamma_{2}\lll 1$, we have $t_{k}^{2M\gamma_{2} -c}\ll 1$. Taking $\gamma_{1}\lll 1$, and observing $y_{k}^2\leq t_{k}^{2(d + 1)\gamma_{1}}$ by \eqref{L1norm-varphik} and definition of $\Ks$, we also find $y^2_{k}t_{k}^{-c}\ll1$. Therefore, the error term in \eqref{eq-Zk2} is bounded by $O(1)$.

We now estimate the main term of \eqref{eq-Zk2}.
By expanding the square, then using \Cref{Young-ineq} and  $m_G(\theta)=1$, we see that
\[
\int_{X}(\varphi_{k}(x) - y_{k})^2\dd m_X(x) = \norm{\varphi_k}_{L^2}^2 - y_k^2 \leq \norm{\chi_k\tilde{\1}_{R^-_k}}_{L^2}^2 - y_k^2 \leq \norm{\tilde{\1}_{R^-_k}}_{L^2}^2 - y_k^2.
\]
Using \eqref{Sieg2}, \eqref{L1norm-varphik} and $y_{k}\geq 1$, we obtain that the  main term is bounded by $O(y_k)$.
The result follows.
\end{proof}

From the effective double equidistribution hypothesis on $\xi$, we deduce an upper bound on the second moment of a sum of $Z_k$'s where $k\in \Ks$.

\begin{proposition} \label{L2bound-Zj-general}
Assume $\gamma_{1}\lll\gamma_{2}\lll 1$ and
\begin{equation}
\label{eq:qlogq}
\psi(q) \geq q^{-1/d}\log^{-2/d}(q), \quad \forall q \in \Ks.
\end{equation}
Then for every finite subset $J \subseteq \Ks$ with $\inf J\ggg_{\gamma_{2}}1$, we have
\[\E\left[(\sum_{j\in J}Z_j)^2\right] \ll \left(1+\sum_{j\in J} y_{j}\right)^{3/2}.\]
\end{proposition}

\begin{proof}
We use the shorthand
\[
Y_J := \sum_{k \in J} Y_k, \quad \quad  \quad  y_J := \sum_{k \in J} y_k \quad \quad \text{and}\quad \quad Z_J=Y_{J}-y_{J}.
\]
Set $J_{1}:=\set{j \in J\,:\, y_{j} < 1}$, write $n:=\abs{J\smallsetminus J_{1}}$.
Further partition $J$ into $J=J_{1} \sqcup J_{2}\sqcup J_{3}$ where $J_{2},J_{3}$ are respectively determined by the condition $y_{j}\in [1,n^{2})$, and $y_{j}\in [n^2,+\infty)$.
Using the inequalities $(a+b)^2\leq 2(a^2+b^2)$ and $a^{3/2}+ b^{3/2}\leq (a+b)^{3/2}$ valid for all $a,b\in \R_{\geq 0}$, we just need to check the upper bound for the sum over each $J_{i}$ independently.

\bigskip

\noindent \underline{Case of $J_{1}$.}

By definition, for all $j \leq k \in \Ks$,
\begin{equation}\label{eq-double-eqZjk}
\E[Z_jZ_k] = \E[Y_jY_k] - y_j y_k - \E[Z_j]y_k - y_j\E[Z_k].
\end{equation}

By double equidistribution   \eqref{eq:all-range}, we have 
\begin{equation*}
\abs{\E[Y_jY_k] - y_j y_k} \ll \cS_{2,l}(\varphi_j) \cS_{2,l}(\varphi_k) t_j^c t_k^{-c} + \cS_{\infty,l}(\varphi_j) y_k t_j^{-c} + \cS_{\infty,l}(\varphi_j) \cS_{\infty,l}(\varphi_k) t_k^{-c}.
\end{equation*}
Assume $j,k \in J_1$ so that $y_j,y_k<1$.
Using \eqref{eq:controlSobolev1}, \eqref{eq:controlSobolev2}, the above becomes
\begin{align} 
\abs{\E[Y_jY_k] - y_j y_k}
& \ll (\sqrt{y_j}+ t_j^{-\gamma_{2}/M})(\sqrt{y_k}+ t_k^{-\gamma_{2}/M})t_j^{c} t_k^{-c} + y_k t_j^{-c+M\gamma_{2}} + t_j^{M\gamma_{2}}t_k^{-c+M\gamma_{2}}\nonumber \\
& \ll \sqrt{y_j y_k}t_j^{c} t_k^{-c} + y_k t_j^{-c/2} + t_k^{-c/2}+ t_j^{c-\gamma_{2}/M} t_k^{-c} \label{bnd-corrY}
\end{align}
where the second inequality assumes $\gamma_{2} \leq c /(4M)$, and uses that $t_{j}\leq t_{k}$.

On the other hand, by single equidistribution \eqref{single-eq-prop} and the norm control \eqref{eq:controlSobolev1}, we have for $j \in J_1$,
\begin{equation}\label{Esp[Zj]}
\abs{\E[Z_j]} \ll t_j^{-c+M\gamma_{2}}\ll t_j^{-c/2}.
\end{equation}

By expanding the square, using the above bounds \eqref{eq-double-eqZjk}, \eqref{bnd-corrY}, \eqref{Esp[Zj]}, and recalling from \eqref{tk and rk} that $t_j \gg \tau^j$ and $t_k/t_j \geq \tau^{k-j}$ for $j\leq k$, we deduce
\[\E [Z_{J_1}^2]
 \ll\sum\nolimits_{j,k \in J_{1}, j\leq k} \left(\sqrt{y_j y_k} \tau^{-c(k-j)} + y_k \tau^{-cj/2} + \tau^{-ck/2} + \tau^{-c(k-j) - j\gamma_{2}/M} \right).\]
Using $\sqrt{y_j y_k}\leq y_j + y_k$ and the convergence $\sum_{n = 0}^\infty \tau^{-cn} < + \infty$, the first sum satisfies $\sum\nolimits_{j,k \in J_{1}, j\leq k} \sqrt{y_j y_k} \tau^{-c(k-j)}\ll y_{J_{1}}$.
The convergence $\sum_{n = 0}^\infty \tau^{-cn} < + \infty$ bounds similarly the second sum.
To bound the third sum, note that combining \eqref{tk and rk} with our assumption \eqref{eq:qlogq},  we have
\[
\tau^{-c j/2} \ll (j\log \tau)^{-2} \leq r_j^{d+1} \ll y_j,
\]
where the last inequality relies on the estimate \eqref{L1norm-varphik}, and the  bound $t_{j}^{-\gamma_{2}/M}\leq (j\log \tau)^{-2} \leq r_j^{d+1}$ (which holds under the assumption $\inf J\ggg_{\gamma_{2}}1$).
Hence $\tau^{-c k/2} \ll y_j \tau^{-c(k-j)/2}$, so  $\sum\nolimits_{j,k \in J_{1}, j\leq k} \tau^{-c k/2} \ll y_{J_{1}}$ as for the first sum. The fourth sum can be handled similarly to the third one.
In the end, we have bounded $\E [Z_{J_1}^2]$ by  $O(\sum_{J_{1}}y_{j})\ll (1+y_{J_{1}})^{3/2}$ as desired.
\bigskip

\noindent \underline{Case of $J_{2}$}.

 Set $m:=|J_{2}|$. We start with the case where $m$ is very small compared to $n=|J_{2}\sqcup J_{3}|$, more precisely we assume $m^2\leq n$. In this scenario, we have by the Cauchy-Schwarz inequality and \Cref{variance-bound},
\begin{align*}
\E \bigl[Z_{J_2}^2\bigr] \leq m \sum_{j\in J_{2}} \E[Z_{j}^2] \ll m \sum_{j\in J_{2}}y_{j} \leq m^2n^2\leq n^3 \ll y_{J_{2}\sqcup J_{3}}^{3/2},
\end{align*}
whence the desired bound.
Assume now $m^2> n$. Decompose $J_{2}$ into $J_{2}=J'_{2} \sqcup J_{2}''$ according to whether $j\geq \sqrt{n}$ or not. The preceding argument gives
\begin{equation*}
\E\bigl[Z_{J''_{2}}^2\bigr] \ll y_{J_{2}\sqcup J_{3}}^{3/2}.
\end{equation*}
We now focus on $J'_{2}$. Note that
\begin{align}
 \E \bigl[Z_{J'_{2}}^2\bigr]
 &= \sum_{|j-k|< \sqrt{m}}\E[Z_{j}Z_{k}]+\sum_{|j-k|\geq \sqrt{m}}\E[Z_{j}Z_{k}] \label{eq1J2'}\\
 & \ll \sqrt{m} \sum_{j \in J'_{2}} \E[Z_{j}^2] + \Bigl\lvert\sum_{|j-k|\geq \sqrt{m}}\E[Z_{j}Z_{k}]\Bigr\rvert \label{eq2J2'}
\end{align}
where the second inequality uses the trivial bound $\E[Z_{j}Z_{k}]\leq \E[Z_j^2 + Z_k^2]$ and the observation that each subscript $j$ in the first sum of \eqref{eq1J2'} appears at most $O(\sqrt{m})$ many times.

For subscripts $j \leq k\in J_{2}'$ such that $|j-k|\geq \sqrt{m}$, we have by \eqref{eq-double-eqZjk}, double equidistribution \eqref{eq:all-range}, and \Cref{norm-estim},  that
$$\E[Z_{j}Z_{k}] \ll y_j y_k (t_j^{c} t_k^{-c} + t_j^{-c/2} + t_k^{-c/2})\leq n^4 \tau^{-c \sqrt{m}/2} $$
where the second inequality relies on the definition of $J'_{2}$.
Plugging this bound and \Cref{variance-bound} into \eqref{eq2J2'}, we obtain
\begin{align*}
 \E \bigl[Z_{J'_{2}}^2\bigr]\ll\sqrt{m} y_{J'_{2}} + \underbrace{m^2n^4 \tau^{-c \sqrt{m}/2} }_{O(1)} \,\ll\, y_{J_{2}}^{3/2}.
 \end{align*}

\bigskip
\noindent \underline{Case of $J_{3}$.}

We finally deal with $J_{3}$. Applying the Cauchy-Schwarz inequality then \Cref{variance-bound}, we obtain
 \[\E \bigl[Z_{J_3}^2\bigr] \leq |J_{3}| \sum_{j\in J_{3}} \E[Z_{j}^2] \ll |J_{3}| \sum_{j\in J_{3}} y_{j}\leq y_{J_3}^{3/2}.\]
 This concludes the proof.
\end{proof}

We also need the next lemma, which is a variant of \cite[Lemma 1.5]{Harman}.
It converts a variance control as in \Cref{L2bound-Zj-general} into an asymptotic estimate.

\begin{lemma}\label{Lemma: Schmidt's argument++}
Let $(Y_j)_{j\geq1}$ be a sequence of non-negative real random variables.
Let $(y_{j})_{j\geq1}, (y'_{j})_{j\geq1} \in \R_{\geq 0}^{\N}$ be sequences of non-negative real numbers.
Set $Z_j=Y_j-y_j$.
Assume  $\sum_{j=1}^{\infty} y_j = +\infty$, and that for some $C_1 \geq 1$,  for all $n \geq m \geq C_{1}$, we have 
$y_{m}\leq y_{m}'$ and
\begin{equation}\label{qeq-2}
\E\Bigl[\bigl(\sum_{j=m}^n Z_j\bigr)^2\Bigr]\leq C_1 \bigl(1+\sum_{j=m}^n y'_k\bigr)^{3/2}.
\end{equation}
Then almost surely, for large enough $n$, we have
\[
\Bigl\lvert \sum_{k=1}^n Z_k \Bigr\rvert \leq \Bigl(\sum_{k=1}^n y'_k\Bigr)^{4/5}
\]

\end{lemma}

\begin{proof}
For an interval $J\subseteq \R_{>0}$, we use the notation $Z_J=\sum_{j\in \N\cap J} Z_j$ and  define similarly $y_{J}$, $y'_J$.
We prove the following slightly stronger statement : there is an almost-surely finite random variable $C_2$ such that for all $N \geq C_2$, we have
\begin{equation}
\label{eq:Schmidt arg}
\abs{Z_{(0,N]}} \leq (\log (y'_{(0,N]} + 2))^{2} (y'_{(0,N]} + 2)^{3/4} + C_2.
\end{equation}
For this, up to throwing away a finite number of terms, we may assume that the relations $y_{m}\leq y'_{m}$ and \eqref{qeq-2} hold for all $n\geq m\geq 1$.

Under this assumption, we have the following.
\begin{lemma}
\label{lm:aux Schmidt}
Let $0 = N_0 < N_1 < N_2 < \dotsb$ be an increasing sequence of integers such that
\begin{equation}
\label{eq:yprimegeq1}
\forall i \geq 0,\quad y'_{(N_i,N_{i+1}]} \geq 1.
\end{equation}
Then almost-surely, for sufficiently large $i$,
\begin{equation}
\label{Zbound}
Z_{(0,N_{i}]}^2 \leq (\log y'_{(0,N_i]})^{4} (y'_{(0,N_i]})^{3/2}.
\end{equation}
\end{lemma}
\begin{proof}[Proof of \Cref{lm:aux Schmidt}]
Denote by $\sD$ the set of integers $T\geq 2$ such that the associated dyadic interval $(2^{T-1}, 2^{T}]$ meets the collection $(y'_{(0, N_{i}] })_{i\geq1}$.
Define a sequence of integers $(M_T)_{T\in \sD}$ by
$$M_{T} = \max \{N_{i} \,:\, y'_{(0, N_{i}]}\in (2^{T-1}, 2^{T}] \}.$$
For each $T \in \sD$, consider the following collection of intervals
\[\cK_T=\setBig{ (N_{r2^t}, N_{(r+1)2^t}] \,: \, r\geq 1, \,t\geq 0, \text{ and } N_{(r+1)2^t} \leq M_{T}}.\]
By  assumption, $y'_{(0,N_{i}]}\geq i$ for every $i \geq 0$.
It follows that $M_{T}\leq N_{2^T}$.
Therefore, every integer in $\llbracket 1, M_{T}\rrbracket$ is contained in at most $T + 1$ intervals of $\cK_T$.
Applying the assumption~\eqref{qeq-2} and the inequality $y'_{(N_{i}, N_{i+1}]}\geq 1$, followed by the relation $a^{3/2}+b^{3/2}\leq (a+b)^{3/2}$ for all $a,b\in \R_{\geq 0}$, we deduce that
\[
\sum_{I \in \cK_T} \E[Z_I^2] \leq  {C_{1}} \left(2\sum_{I \in \cK_T} y'_I\right)^{3/2} \leq {C_{1}} \left(2 (T+1) y'_{(0,M_T]}\right)^{3/2} \leq 2^3 {C_{1}} \left(T 2^T\right)^{3/2}.
\]
By Markov's inequality,
\begin{equation*}
\mathbb{P}\Bigl[ \sum\nolimits_{I \in \cK_T} Z_I^2 > 2^{-10} T^{3} 2^{3T/2} \Bigr] \leq 2^{13} {C_1} T^{-3/2}.
\end{equation*}
The latter being summable over $T \in \sD$, we can use the Borel-Cantelli lemma to deduce that almost surely, for large enough $T$, we have 
\begin{equation}
\label{eq:ZIsquare}
 \sum_{I \in \cK_T} Z_I^2 \leq 2^{-10} T^{3} 2^{3T/2}.
\end{equation}
Let $i \geq 1$. Assume $i$ large enough so that the unique element $T \in \sD$ such that $2^{T-1}< y'_{(0,N_{i}]} \leq 2^{T}$ satisfies \eqref{eq:ZIsquare} as well.
By considering $i$ in base $2$ and using $N_{i}\leq M_{T}\leq N_{2^T}$, we may cover $(0,N_{i}]$ with at most $T$ non-overlapping intervals from $\cK_T$.
Let $\cK_{T, i}$ be such a collection of intervals.
Then by the Cauchy-Schwarz inequality and \eqref{eq:ZIsquare},
\begin{equation*}
Z_{(0,N_{i}]}^2 = \Bigl(\sum_{I \in \cK_{T, i}} Z_I\Bigr)^2\leq \abs{\cK_{T, i}} \sum_{I \in \cK_{T, i}} Z_I^2 \leq 2^{-10} T^{4} 2^{3T/2}.
\end{equation*}
We obtain the desired bound using $2^{T-1} < y'_{(0,N_i]}$.
\end{proof}

To show \eqref{eq:Schmidt arg}, we provide lower and upper bounds for $Z_{(0,N]}$.
Let $\sM$ be the set of integers $m\geq 2$ such that the interval $(m-1, m]$ meets the collection $(y'_{(0, j] })_{j\geq1}$.
Consider $\sN=\{n_{1}<n_{2}<\dotsb\}\subseteq \sM$ a subset satisfying $\inf_{i\neq j} |n_{i}-n_{j}|\geq 2$ and maximal conditionally to this property.

To obtain the lower bound, we set for $i \geq1$,
$$N_{i} := \min \{j \geq 1\,:\, y'_{(0, j]}\in (n_{i}-1, n_{i}] \}.$$
The advantage of using $\sN$ and not $\sM$ to define $N_{i}$ is that we can guarantee \eqref{eq:yprimegeq1}.
Thus we can apply \Cref{lm:aux Schmidt} to the sequence $(N_i)$.
We obtain that almost surely, if $N \geq 1$ is sufficiently large, then the unique $i \geq 1$ such that $N\in [N_{i}, N_{i+1})$ satisfies \eqref{Zbound}.
Recalling that the $Y_j$'s are almost-surely non-negative, we obtain that
\begin{equation*}
Z_{(0,N]}
\geq Z_{(0,N_i]} - y_{(N_{i},N]} \geq  -(\log y'_{(0,N_i]})^{2} (y'_{(0,N_i]})^{3/4} - y'_{(N_{i},N]}.
\end{equation*}
and the desired lower bound follows as noting that by construction, we have $0 \leq  y'_{(N_i, N]} \leq 2$. 

The upper bound can be handled similarly, but for that we need to modify the sequence $(N_{i})_i$ into a certain $(N'_{i})_i$, guaranteeing that when $N$ ranges within $(N'_{i-1}, N'_{i}]$, the value of $y'_{(0,N]}$ does not vary much. More precisely, we replace $N_{i}$ with
$$N'_{i} = \max \{j \geq 1\,:\, y'_{(0, j]}\in (n_{i}-1, n_{i}] \}.$$
Applying \Cref{lm:aux Schmidt} with $(N'_i)_{i}$, we have that almost-surely, for large enough $N$, for $i$ such that $N\in(N'_{i-1},N'_i]$, 
\begin{equation*}
Z_{(0,N]}
\leq Z_{(0,N'_{i-1}]} + y_{(N,N'_{i}]} \leq  (\log y'_{(0,N'_i]})^{2} (y'_{(0,N'_i]})^{3/4} + y_{(N,N'_{i}]}.
\end{equation*}
The desired lower bound follows, noting that by construction, we have $0\leq y'_{(N,N_i']}\leq 2$. 
This finishes the proof of \eqref{eq:Schmidt arg}.
\end{proof}

Combining \Cref{L2bound-Zj-general} and \Cref{Lemma: Schmidt's argument++}, we obtain the following counting estimate for parameters in $\Ks$. Given $n\geq 1$, we recall  $\Ks(n)=\Ks\cap \llbracket 1, n\rrbracket$.

\begin{corollary} \label{cor-Ksmall-asymp}
Assume $\sum_{k\in \Ks}\psi(\tau^k)^d\tau^k= +\infty$, as well as $\gamma_{1}\lll1$. Then for $\xi$-almost every $\bs\in \R^d$, for large enough $n$,
\begin{equation} \label{eq-Ksm-as}
 \sum_{k\in \Ks(n)} \sS_{k}(\bs) \geq (1-\tau^{-1}- C(d)\eps) \sum_{k\in \Ks(n)}  \psi(\tau^k)^d \tau^k
 \end{equation}
where $C(d)>0$ is a constant depending on $d$ only.
\end{corollary}

\begin{proof} By the convergent case of the Khintchine dichotomy, replacing $\psi(q)$ by $\max(\psi(q), q^{-1/d}\log^{-2/d}(q))$ may only perturb both sums in \eqref{eq-Ksm-as} by a bounded additive constant (which may depend on $\bs$). As the right hand side of \eqref{eq-Ksm-as} is divergent by hypothesis, this perturbation does not affect asymptotics, so we may assume $\psi(q) \geq q^{-1/d}\log^{-2/d}(q)$ for all $q\in \Ks$.
Recalling \eqref{formula-sSk-Sieg} and $\widetilde{\1}_{R_{k}}\geq \varphi_{k}$, then combining \Cref{L2bound-Zj-general} and \Cref{Lemma: Schmidt's argument++}, we obtain that $\xi$-almost surely, for large enough $n$,
$$ \sum_{k\in \Ks(n)} \sS_{k}(\bs) \geq \sum_{k\in \Ks(n)} Y_{k}(\bs) \geq (1-\eps) \sum_{k\in \Ks(n)} y_{k}. $$
By \eqref{L1norm-varphik}, we have $\sum_{k\in \Ks(n)} y_{k} = O_{\gamma_{2}}(1)+\sum_{k\in \Ks(n)} \Leb(R_{k}^-)$. Extracting from \eqref{tk and rk} that $\Leb(R_{k}^-) = (1-\tau^{-1}-2\eps)(1-2\eps)^d\psi(\tau^k)^d \tau^k$, and using that the associated series diverges by hypothesis, the result follows.
\end{proof}

\bigskip

\noindent{\bf Conclusion for the lower bound (case $d\geq 2$)}

\begin{proof}[Proof of \Cref{pr:rough}]
It follows by combining \Cref{Kbig-asymp} and \Cref{cor-Ksmall-asymp}.
\end{proof}

 \subsection{Case $d\geq 2$, the upper bound}
\label{Sec-upp-bd}
The proof of the asymptotic upper bound in \Cref{General-quant-Khint} (case $d\geq 2$) is similar to that of the lower bound. We briefly sketch the proof to highlight the relevant changes.
We extend $\psi$ to $\R_{\geq 0}$ by setting $\psi(q)=\psi(\lfloor q\rfloor)$ for non-integer values of $q$.
We let $\tau\in (1, 2]$ and note that for all integers $N,n\geq 1$ such that $\tau^n \leq N < \tau^{n+1}$, for every $\bs \in \R^d$, we have
\[
\sT_N(\bs) \leq \sum_{k=0}^n \sS_k^+(\bs)
\]
where for $k \geq 0$, we set
\[ \sS^+_{k}(\bs) :=   \big|\set{(\bp,q)\in \Z^{d}\times \llbracket \tau^{k}, \tau^{k+1}\llbracket \,:\, \forall i\in \llbracket1, d\rrbracket, \,0\leq q s_i-p_i<\psi(\tau^k) }  \big|. \]
Then it suffices to show the upper bound analogue of \Cref{pr:rough}.
\begin{lemma}\label{pr:rough+}
Under the assumptions of \Cref{General-quant-Khint} and with $d\geq 2$, we have for every $\tau \in (1,2]$,  for $\xi$-almost all $\bs \in \R^d$, for every $\eta >0$, for all large enough $n$,
\begin{equation*}
\sum_{k=1}^n \sS^+_{k}(\bs) \leq (\tau-1 +\eta) \sum_{k=1}^n \psi(\tau^k)^d \tau^k.
\end{equation*}
\end{lemma}
To show \Cref{pr:rough+}, we note that
\[
\sS_k^+(\bs) = \widetilde{\1}_{ P_{k}} \bigl( a(t_k)u(\bs)x_0 \bigr)
\]
where $P_{k}:=[0,r_k)^d \times [r_k,\tau r_k)$ and $r_{k},t_{k}>0$ have been defined in \eqref{tk and rk}. Keep the notations $\gamma_{1}$, $\Kb$, $\Ks$, from the proof of the lower bound. We establish the upper bound announced in \Cref{pr:rough+} along the subsums  
$\sum_{k\in \Kb(n)} \sS^+_{k}(\bs)$ and $\sum_{k\in \Ks(n)} \sS^+_{k}(\bs)$ separately, and under the assumptions that the corresponding $\sum_{k \in \Kb} \psi(\tau^k)^d \tau^k$, $\sum_{k \in \Ks} \psi(\tau^k)^d \tau^k$ diverge. This is enough in view of the convergent case of the Khintchine dichotomy, \Cref{convergent-Khintchine}. 

The proof of the asymptotic upper bound for $\sum_{k\in \Kb(n)} \sS^+_{k}(\bs)$ is the same as that of \Cref{Kbig-asymp}, but using this time the upper bound from \Cref{counting} instead of the lower bound. 

To deal with $\sum_{k\in \Ks(n)} \sS^+_{k}(\bs)$, we recall the parameters $\gamma_{2}, \eps$, $\chi_{k}$, $\theta$ from the proof of the lower bound. We introduce the thickened box
$$ P_{k}^+ := [-\eps r_{k},(1+\eps)r_k)^d \times [(1-\eps)r_k, (\tau+\eps) r_k)$$
 and note that $\widetilde{\1}_{ P_{k}}\leq \theta*\widetilde{\1}_{P_{k}^+}$
 (because $\theta$ is supported on $B_{\eps/10}$). We consider the smooth truncated companion
 $$\varphi_{k}^+=\theta*(\chi_{k}\widetilde{\1}_{P_{k}^+}),$$
where $\chi_k$ is defined as in \eqref{truncation}.
Note that contrary to the lower bound case, the truncation plays against comparing $\widetilde{\1}_{ P_{k}}$ with $\varphi_{k}^+$. However,  in view of \eqref{eq:cominjdist}, $\varphi_k^+$ coincides with $\theta*\widetilde{\1}_{P_{k}^+}$ on the subset $\{ \dist( \cdot, x_0) \leq \frac{\gamma_2}{M} \log t_k - M \}$ for some $M = M(d) > 1$, in particular we do have $\widetilde{\1}_{ P_{k}}\leq \varphi_{k}^+$ on that domain. This is enough for us because, for $\xi$-typical $\bs$, and large enough $k$, the points $a(t_{k})u(\bs)x_{0}$ are in this domain, indeed \Cref{logarithm law} gives \[\dist(a(t_{k})u(\bs)x_{0}, x_{0})\ll \log \log t_{k}.\]

We get for $\xi$-almost every $\bs$, for large enough $k$,
\begin{equation*}\label{upp-bnd-phi}
\sS^+_{k}(\bs)\leq \varphi^+_{k}(a(t_{k})u(\bs)x_{0}).
\end{equation*}
Therefore, we  only need to show the upper bound analogue of \Cref{cor-Ksmall-asymp}:
\begin{equation}\label{upbndp}
\sum_{k\in \Ks(n)} \varphi^+_{k}(a(t_{k})u(\bs)x_{0}) \leq (\tau-1+C(d)\eps) \sum_{k\in \Ks(n)}  \psi(\tau^k)^d \tau^k
 \end{equation}
 where $C(d)>0$ is a constant depending on $d$ only.
The estimate \eqref{upbndp} follows mutatis mutandis from the argument establishing the lower bound for partial sums over $\Ks$, in which we replace $R_{k}^-$ by $P_{k}^+$. 

We have thus established \Cref{pr:rough+}, whence the asymptotic upper bound
\[
\limsup_{N \to +\infty} \frac{\sT_N(\bs)}{\sum_{q=1}^N \psi(q)^d} \leq 1
\]
of \Cref{General-quant-Khint} (case $d\geq2$).


 \subsection{The case $d= 1$.} \label{Sec-d=1}

It remains to establish \Cref{General-quant-Khint} in the case where $d=1$. The proof is similar to the higher dimensional case but a certain number of refinements are required due to poorer moment estimates for Siegel transforms.

Let us start with the {\bf lower bound}. Keep the notations $\tau$, $R_{k}$, $\gamma_{1}$, $\Kb$, $\Ks$, $\gamma_{2}$,  $\chi_{k}$, $\eps$, $R_{k}^-$,  $\theta$, from \S\ref{Sec-lwbnd-d2}. The asymptotic lower bound for $\sum_{k\in \Kb(n)} \sS_{k}(\bs)$ given in \Cref{Kbig-asymp}
is still valid because the argument works without restriction on $d$.

We thus focus on the lower asymptotic for $\sum_{k\in \Ks(n)} \sS_{k}(\bs)$, more precisely on extending \Cref{cor-Ksmall-asymp} to the case $d=1$. The difference with the higher dimensional case is that \emph{the Siegel transform of a ball in $\R^2$ does not have finite second moment}, in particular \eqref{Sieg2} is not valid anymore. To deal with this obstacle, we restrict the Siegel transform by counting only lattice points $(p,q)$ which are bounded multiple of a primitive point. Namely, given $m >0$,  we set
$$ \cP^{(m)}(\Z^{2}):=\set{(p,q)\in \Z^2\smallsetminus \{0\}  \,:\, \GCD(p,q)\leq m} $$
where $\GCD(p,q)\in \N^*$ denotes the greatest common divisor of $p$ and $q$.
Given a measurable function $f : \R^2 \rightarrow \R_{\geq 0}$, we define its restricted Siegel transform $\widetilde{f}^{(m)}:X\rightarrow [0, +\infty]$ by: $\forall g\in G$,
$$\widetilde{f}^{(m)}(gx_{0}) :=\sum_{v\in \cP^{(m)}(\Z^{2})} f(gv).$$
In this context, we have the following moment estimates. Their vocation is to replace \Cref{fact-Siegel} from the higher dimensional case.
\begin{proposition}[Moments of restricted Siegel transforms] \label{fact-Siegel-lprim}
Let $m \in \N^*$, let $k\in \Ks$. Let $c_{m}:=\zeta(2)^{-1}\sum_{t=1}^m t^{-2}$ and let $R \subset \R^2$ be a bounded measurable subset.
We have
\begin{align} 
&\int_{X} \widetilde{\1}^{(m)}_{R} \dd m_{X}= c_{m} \Leb(R) \label{Sieg-prim-1}\\
 &\int_{X} (\widetilde{\1}^{(m)}_{R} )^2 \dd m_{X} = c_{m}^2\Leb(R)^2 + O\left(\Leb(R) \log^2 (1+m) \right). \label{Sieg-prim-2}
 \end{align}
\end{proposition}

{\begin{proof}
Those estimates appear in the literature for primitive Siegel transforms, i.e., for $m=1$. We explain how to deduce the case $m\geq1$. Note that
$$\widetilde{\1}^{(m)}_{R}=\sum_{q\in \llbracket 1, m\rrbracket} \widetilde{\1}^{(1)}_{R/q}= \widetilde{h}^{(1)}$$
where $h:=\sum_{q\in \llbracket 1, m\rrbracket} {\1}_{R/q}$.
The Siegel summation formula for primitive Siegel transform \cite[Equation ($8$)]{Schmidt60} guarantees $\int_{X}  \widetilde{h}^{(1)} \dd m_{X}= \zeta(2)^{-1}  \int_{\R^2}h \,\dd\text{Leb}=c_{m}\leb(R)$, whence \eqref{Sieg-prim-1}.

To justify \eqref{Sieg-prim-2}, we apply \cite[Theorem 1 \& Remark 0.9]{KY21}, and the Cauchy-Schwarz inequality to bound the integral involving the isometry in that theorem, to get
\begin{align*}
\int_{X}(\widetilde{h}^{(1)})^2 \dd m_{X}
&= \left| \zeta(2)^{-1}\int_{\R^2}h \,\dd\text{Leb}\right|^2+ O\left(\int_{\R^2}h^2 \,\dd\text{Leb}\right).
\end{align*}
The first term on the right-hand side is equal to $c_{m}^2\Leb(R)^2$. The second term can be bounded via
\begin{align*}
\int_{\R^2}h^2 \,\dd\text{Leb}
&=\sum_{q,q'\in \llbracket 1,m\rrbracket}\int_{\R^2}\1_{R/q}\1_{R/q'} \,\dd\text{Leb}
\leq \sum_{q,q'\in \llbracket 1,m\rrbracket}\|\1_{R/q}\|_{L^2}\|\1_{R/q'}\|_{L^2} \\
&\leq  \leb(R)\sum_{q,q'\in \llbracket 1,m\rrbracket} q^{-1} q'^{-1} \ll  \leb(R)\log^2(1+m).
 \end{align*}
 This concludes the proof of \eqref{Sieg-prim-2}.
\end{proof}

From there, the proof of the higher dimensional case goes through with a few adaptations.
Let $(m_{k})_{k\in \Ks} \in \R_{>0}^{\Ks}$ be the sequence satisfying
\begin{equation}
\label{eq:defmk}
\forall k \in \Ks, \quad \log^2(1 + m_k) = \Bigl(\sum_{j \in \Ks(k)} \psi(\tau^j) \tau^j \Bigr)^{1/8}.
\end{equation}
Set $\varphi_{k}:=\theta *(\chi_{k} \widetilde{\1}^{(m_{k})}_{R_{k}^-} )$ where $\chi_k$ is as in (\ref{truncation}), 
and define  positive constants  $$y_{k}:=m_X(\varphi_k),\qquad y'_{k}:=m_X(\varphi_k) \log^2(1+m_k).$$
Noting that by hypothesis, we have $\lim_{k}m_{k}=+\infty$ as $k$ tends to infinity along $\Ks$, we find that $y_{k}\leq y'_{k}$ for all large $k\in \Ks$. 

Replacing \Cref{fact-Siegel} by \Cref{fact-Siegel-lprim} in the proof of \Cref{norm-estim}, we obtain the following.
\begin{lemma}\label{norm-estim-prim}
 If $\gamma_{1} \lll \gamma_{2}$, then for some $M=M(d)>1$, every $k\in \Ks$, we have
\begin{align}
&y_{k} = c_{m_{k}} \Leb(R_{k}^-) - O(t_{k}^{-\gamma_{2}/M}) \label{L1norm-varphik-prim},\\
&\cS_{\infty,l}(\varphi_k) \ll t_{k}^{M\gamma_{2}}, \nonumber \\ 
&\cS_{2,l}(\varphi_k) \ll y_{k}+ \sqrt{y'_{k}} +O(t_{k}^{-\gamma_{2}/M}). \nonumber 
\end{align}
\end{lemma}

Next, writing for $\bs\in \R$,
\[ Y_k (\bs) = \varphi_k\bigl(a(t_{k})u(\bs)x_{0}\bigr),\qquad Z_{k}(\bs)=Y_{k}(\bs)-y_{k},\]
\Cref{variance-bound} becomes (with a similar proof) the following.
\begin{lemma} \label{variance-bound-prim} Assume $\gamma_{1}\lll\gamma_{2}\lll 1$. Then for every $k\in \Ks$ such that $y'_{k}\geq 1$, we have
$$\E[Z_{k}^2]\ll y'_{k}.$$
\end{lemma}

We deduce the following replacement for \Cref{L2bound-Zj-general}.
\begin{proposition} \label{L2bound-Zj-general-prim}
Assume $\gamma_{1}\lll\gamma_{2}\lll 1$ and
\begin{equation*}
\psi(q) \geq q^{-1}\log^{-2}(q), \quad \forall q \in \Ks.
\end{equation*}
Then for every finite subset $J \subseteq \Ks$ with $\inf J\ggg_{\gamma_{1},\gamma_{2}}1$, we have  
\[\E \bigl[Z_{J}^2\bigr] \ll \left(1+y'_{J}\right)^{3/2}.\]
\end{proposition}
We recall that $Z_{J}=\sum_{j\in J}Z_{j}$ and $y'_{J}=\sum_{j\in J}y'_{j}$. Similar notations $Y_{J},y_{J}$ will be used below.
\begin{proof}
Same as for \Cref{L2bound-Zj-general} but using $y_{j}'$ to define the partition $J=J_{1}\sqcup J_{2} \sqcup J_{3}$, and noting that all the upper bounds in the proof of \Cref{L2bound-Zj-general} are valid with $y_{j}'$ at the place of $y_{j}$ thanks to Lemmas \ref{norm-estim-prim}, \ref{variance-bound-prim}, and the inequality $y_{j}'\geq y_{j}$ (which is valid thanks to the assumption $\inf J\ggg_{\gamma_{1}}1$).
\end{proof}

We can now combine \Cref{L2bound-Zj-general-prim} and \Cref{Lemma: Schmidt's argument++} (note here that we allow $(y'_j)$ to be different from $(y_j)$ in the latter) to obtain the following.
\begin{lemma} \label{lemZky'k}
For $\xi$-almost every $\bs\in \R$, for large enough $n$, we have
\begin{equation} \label{Zky'k}
\abs{Z_{\Ks(n) }(\bs)}\leq \left(y'_{\Ks(n)}\right)^{4/5}.
\end{equation}
\end{lemma}
We now claim that the right-hand side in \eqref{Zky'k} is negligible compared to $y_{\Ks(n)}$.
To see why, note first that by definition, the sequence $(m_{k})_{k\in \Ks}$ is non-decreasing, therefore
\begin{equation}\label{y'Ksnleqylog}
y'_{\Ks(n)}\leq y_{\Ks(n)} \log^2 (1+m_{\max \Ks(n)}).
\end{equation}
Moreover, by \eqref{tk and rk}, \eqref{L1norm-varphik-prim}, we have $\psi(\tau^k) \tau^k \ll y_{k} + t^{-\gamma_{2}/M}_{k}$, so using the definition of $m_{k}$, we get
\begin{equation}\label{logmkbndy}
\forall k \in \Ks, \quad \log^2(1 + m_k) \ll \left(y_{\Ks(k)} + O_{\gamma_2}(1)\right)^{1/8}.
\end{equation}
Equations \eqref{y'Ksnleqylog} and \eqref{logmkbndy} together justify the claim.

\Cref{lemZky'k} and the above claim yield in particular that for $\xi$-almost every $\bs\in \R$, for sufficiently large $n$,
\[
Y_{\Ks(n)}(\bs) \geq (1 - \eps) y_{\Ks(n)}.
\]
By construction, we  know that $\sS_{k}(\bs)\geq Y_{k}(\bs)$. On the other hand, it follows from \eqref{L1norm-varphik-prim} that $y_{\Ks(n)} \geq \sum_{k \in \Ks(n)} c_{m_{k}}\Leb(R_k^-) - O_{\gamma_2}(1)$, where $c_{m_{k}}\to_{k} 1$ as $k$ goes to infinity along $\Ks$. Using \eqref{tk and rk} to see that $\Leb(R_{k}^-)=(1-\tau^{-1}-2\eps)(1-2\eps)\psi(\tau^k)\tau^k$, we infer that for $\xi$-almost every $\bs \in \R$, for all large enough $n$,
\begin{equation*}
\sum_{k\in \Ks(n)} \sS_{k}(\bs) \geq (1-\tau^{-1}- 6 \eps) \sum_{k\in \Ks(n)}  \psi(\tau^k) \tau^k.
\end{equation*}
This concludes the proof of the lower bound.

\bigskip
Let us now justify the {\bf upper bound} in the case $d=1$. Similarly to the higher dimensional case, we can mimic the proof of the lower bound estimate in the case $d=1$ to establish an upper bound estimate. However, this upper bound concerns the restricted Siegel transforms which are used in the proof, while we aim for an upper bound without restriction on $\GCD(p,q)$ when counting solutions $(p,q)$ of the Khintchine inequality. We explain below how to deal with this obstacle.

We keep the notations of \S\ref{Sec-upp-bd}, in particular we fix $\tau \in (1,2]$, and consider for $k\geq 0$, $\bs\in \R$,
\begin{align*}
\sS_k^+(\bs)
&= \big|\{(p,q)\in \Z\times \llbracket \tau^{k}, \tau^{k+1}\llbracket \,:\, 0\leq q \bs -p<\psi(\tau^k) \}  \big| \\
&= \widetilde{\1}_{ P_{k}} \bigl( a(t_k)u(\bs)x_0 \bigr).
\end{align*}
where $P_{k}=[0,r_k) \times [r_k,\tau r_k)$.
The goal is to show \Cref{pr:rough+} when $d=1$. Provided $\sum_{k\in \Kb} \psi(\tau^k) \tau^k=+\infty$, the argument for \Cref{Kbig-asymp} yields the desired upper bound for the partial sums $\sum_{\Kb(n)}\sS_k^+(\bs)$.
Therefore we only need to deal with $\sum_{\Ks(n)}\sS_k^+(\bs)$, and under the assumption $\sum_{k\in \Ks} \psi(\tau^k) \tau^k=+\infty$ (as noted for \Cref{pr:rough+}).

Recall from \S\ref{Sec-upp-bd} that $\eps > 0$ is an arbitrarily small number and $P_{k}^+$ denotes the $\eps r_{k}$-thickening of $P_{k}$.
Define for $k \in \Ks$ and $m >0$,
\[
\varphi^{+(m)}_{k} = \theta*(\chi_{k}\widetilde{\1}^{(m)}_{P_{k}^+}) \quad \text{and}\quad Y^{+(m)}_{k}(\bs)=\varphi^{+(m)}_{k} \bigl( a(t_k)u(\bs)x_0 \bigr).
\]
When no restriction on the $\GCD$ is prescribed,
we simply write $\varphi^{+}_{k}:=\varphi^{+(\infty)}_{k}$, and $Y^{+}_{k}:=Y^{+(\infty)}_{k}$.
Let $(m_k)_{k \in \Ks}$ be as in \eqref{eq:defmk}.
The argument used for the lower bound (case $d=1$) then shows that, provided $\gamma_{1}\lll\gamma_{2}\lll1$, we have for $\xi$-almost every $\bs\in \R$, for large enough $n$,
\begin{equation}\label{uppbndsmk}
\sum_{k\in \Ks(n)} Y^{+(m_k)}_{k}(\bs) \leq (\tau-1+ 10 \eps) \sum_{k\in \Ks(n)} \psi(\tau^k) \tau^k.
\end{equation}

We now compare the left hand side of \eqref{uppbndsmk} with $\sum_{k\in \Ks(n)} \sS_k(\bs)$. In other terms, we need to show that the truncation of the cusp induced by $\chi_{k}$, and most importantly the reduction of the Siegel transform to counting $m_{k}$-primitive lattice points does not affect too much the asymptotic of the partial sums. The next lemma is a preliminary step to suppress the $\GCD$ restriction due to $m_{k}$.

\begin{lemma} \label{Repl-m'-1}
Provided that $\gamma_{2}\lll1$, we have for every $k\in \Ks$,
$$\E\bigl[Y^{+}_{k} - Y^{+(m_k)}_{k}\bigr] \ll r_{k}^2 m^{-1}_{k} +t_{k}^{-\gamma_{2}}. $$
\end{lemma}

\begin{proof}
Unfolding definitions, then using effective single equidistribution \eqref{single-eq-prop} while noting that $\cS_{\infty,l}(\varphi_k^{+} - \varphi_k^{+(m_{k}) })\ll \|\varphi_k^{+}\|_{L^\infty} \ll t_k^{M \gamma_2}$ for some $M$ as in \Cref{norm-estim-prim}, we have
\begin{align*}
\E\bigl[Y^{+}_{k} - Y^{+(m_k)}_{k} \bigr]
&=\int_{\R^d} \theta*[\chi_{k}(\widetilde{\1}_{P_{k}^+} -\widetilde{\1}^{(m_{k})}_{P_{k}^+})] \bigl( a(t_k)u(\bs)x_0 \bigr)\dd\xi(\bs)\\
&\ll m_{X}(\widetilde{\1}_{P_{k}^+} -\widetilde{\1}^{(m_{k})}_{P_{k}^+} ) \,+\, t_{k}^{-c + M\gamma_2}\\
&\ll r_k^2 m^{-1}_{k} +t_{k}^{-\gamma_{2}},
\end{align*}
where the last inequality uses \eqref{Sieg-prim-1}, the definition of $c_{m_k}$ and assumes $ \gamma_{2}$ small enough depending on $c$.
\end{proof}
We deduce that \eqref{uppbndsmk} is still valid for $Y^{+}_{k}$ in the place of $Y^{+(m_k)}_{k}$.
\begin{lemma}\label{uppbndsm'k-lem}
For $\xi$-almost every $\bs\in \R$, for large enough $n$, we have
\begin{equation*}\label{uppbndsm'k}
\sum_{k\in \Ks(n)} Y^{+}_{k}(\bs) \leq (\tau-1+ 11\eps) \sum_{k\in \Ks(n)} \psi(\tau^k) \tau^k.
\end{equation*}
\end{lemma}

\begin{proof}
Let $n\geq 1$, set $I_{n}:= \{k\in \Ks\,:\, \sum_{\Ks(k)}r_{j}^2 \in (n,n+1]\}$.
We have by \Cref{Repl-m'-1}
$$\E\Bigl[\sum_{k\in I_{n}} \bigl(Y^{+}_{k} - Y^{+(m_k)}_{k} \bigr)\Bigr] \ll \sum_{k\in I_{n}} \bigl(r_{k}^2 m^{-1}_{k} +t_{k}^{-\gamma_{2}}\bigr) \leq  \frac{n +1}{e^{n^{1/16}} - 1}+ \sum_{k\in I_{n}} t_{k}^{-\gamma_{2}},$$
where the last inequality follows from the definitions of $m_{k}$  in \eqref{eq:defmk} and $I_n$.
Therefore, the right hand side is summable over $n\geq 1$.
Hence, for $\xi$-almost every $\bs$, the total sum $\sum_{k\in \Ks}(Y^{+}_{k} - Y^{+(m_k)}_{k})$ is $\xi$-almost-surely finite.
Then the result follows from \eqref{uppbndsmk}.
\end{proof}

It remains to check the restriction induced by the truncation $\chi_{k}$ can also be ignored. This can be done exactly as for the $d\geq 2$ case (by the mean of \Cref{logarithm law} controlling excursions in the cusp). We get that for $\xi$-almost every $\bs\in \R$, for large enough $k$, we have
\begin{equation}\label{S+Ym'}
\sS^+_{k}(\bs) \leq Y^{+}_{k} (\bs).
\end{equation}

The combination of  \Cref{uppbndsm'k-lem} and Equation \eqref{S+Ym'} yields for $\xi$-almost every $\bs\in \R$, for large enough $n$,
\begin{equation*}
\sum_{k\in \Ks(n)} \sS^+_{k}(\bs) \leq (\tau-1+11\eps) \sum_{k\in \Ks(n)} \psi(\tau^k) \tau^k.
\end{equation*}
 This shows \Cref{pr:rough+} holds for $d=1$, and concludes the proof of the asymptotic upper bound in \Cref{General-quant-Khint} in the $1$-dimensional case.

\bibliographystyle{abbrv} 
\bibliography{khintchine}
\end{document}